\documentclass[titlepage,final,11pt]{article}
\usepackage{enumerate}
\usepackage{enumitem}
\usepackage{array}
\usepackage[reqno]{amsmath}
\usepackage{amssymb}
\usepackage{mathtools}
\usepackage{latexsym}
\usepackage{makeidx}
\usepackage{fancybox}
\input epsf.sty
\usepackage[dvips]{graphicx,color}
\usepackage{showlabels}
\usepackage{subcaption}
\usepackage[dvipsnames]{xcolor}
\usepackage[normalem]{ulem}
\usepackage{amsthm}
\usepackage{algorithm}
\usepackage{algpseudocode}
\usepackage{tabularx}
\usepackage{booktabs}

\theoremstyle{change}
\newtheorem{proclaim}{PROCLAIM}[section]
\newtheorem{theorem}[proclaim]{Theorem}

\newtheorem{lemma}[proclaim]{Lemma}
\newtheorem{proposition}[proclaim]{Proposition}

\newtheorem{example}[proclaim]{Example}

\newtheorem*{assumption*}{Assumption}

\numberwithin{equation}{section}

\outer\def\proclaim #1. #2\par{\medbreak \noindent{\bf#1.\enspace}{\sl#2}\par
  \ifdim\lastskip<\medskipamount
  \removelastskip\penalty55\medskip\fi}
\def\state #1. { \noindent{\bf#1.\enspace}}
\def\algo #1. { \noindent{\bf#1.\enspace}}

\DeclareMathOperator{\cl}{cl}

\DeclareMathOperator{\con}{con}

\DeclareMathOperator{\dist}{dist}

\DeclareMathOperator{\dom}{dom}
\DeclareMathOperator{\gph}{gph}
\DeclareMathOperator{\epi}{epi}

\DeclareMathOperator{\nt}{int}

\DeclareMathOperator{\pos}{pos}

\DeclareMathOperator{\ri}{ri}
\DeclareMathOperator{\spn}{span}

\newcommand{\comp}{\,{\raise 1pt \hbox{$\scriptstyle\circ$}}\,}

\newcommand{\reals}{\mathbb{R}}
\newcommand{\Reals}{\overline{\mathbb{R}}}
\newcommand{\natnums}{{{\rm l} \kern -.13em {\rm N} }}
\newcommand{\nats}{\mathbb{N}}
\newcommand{\snats}{{I\kern -.29em N}}
\newcommand{\rats}{{Q\kern -.64em \raise 1pt \hbox{$\scriptstyle |$}\;\,}}
\newcommand{\srats}
	{{Q\kern -.56em \raise 1.2pt \hbox{$\scriptscriptstyle /$}\,}}
\newcommand{\ints}{Z\kern -.46em Z}
\newcommand{\ball}{\mathbb{B}}
\newcommand{\pluss}{\hskip1pt \raise1pt\vbox{\hrule width6pt \vskip1pt \hrule
                    width6pt} \kern-4pt{\lower1pt\hbox{\vrule height6pt
		    \kern1pt\vrule height6pt}}\hskip5pt}
\newcommand{\eop}
	{\hfill{$\vcenter{\hrule height1pt \hbox{\vrule width1pt height5pt
   	 \kern5pt \vrule width1pt} \hrule height1pt}$} \medskip}

\newcommand{\setd}{{ d \kern -.15em l}}
\newcommand{\hatsetd}{ d \hat{\kern -.15em l }}

\renewcommand{\epsilon}{\varepsilon}
\renewcommand{\phi}{\varphi}

\newcommand{\lset}{\big\lbrace}
\newcommand{\mset}{{\,\big\vert\,}}
\newcommand{\rset}{\big\rbrace}

\newcommand{\tto}{\;{\lower 1pt \hbox{$\rightarrow$}}\kern -12pt
           \hbox{\raise 2.5pt \hbox{$\rightarrow$}}\;}
\newcommand{\overto}[1]{\,{\raise 0pt\hbox{$\rightarrow$}}\kern -9pt
     \hbox{\lower 3pt \hbox{$\scriptscriptstyle#1$}}\hskip6pt}
\newcommand{\underto}[1]{\,{\lower 1pt\hbox{$\rightarrow$}}\kern -9pt
     \hbox{\raise 4pt \hbox{$\,\scriptscriptstyle#1$}}\hskip7pt}
\newcommand{\bigoverto}[1]{{\raise 0pt\hbox{$\,\longrightarrow$}}\kern -16pt
     \hbox{\lower 3pt \hbox{$\scriptscriptstyle#1$}}\hskip4pt}
\newcommand{\bigunderto}[1]{\,{\lower 1pt\hbox{$\longrightarrow$}}\kern -16pt
     \hbox{\raise 4pt \hbox{$\,\scriptscriptstyle#1$}}\hskip6pt}
\newcommand{\bigbigto}[2]{\,{\raise 0pt\hbox{$\,\longrightarrow$}}\kern -16pt
     \hbox{\lower 3pt \hbox{$\scriptscriptstyle#2$}}\kern -10pt
     \hbox{\raise 4pt \hbox{$\,\scriptscriptstyle#1$}}\hskip7pt}
\newcommand{\downto}{{\raise 1pt \hbox{$\scriptscriptstyle \,\searrow\,$}}}
\newcommand{\upto}{{\raise 1pt \hbox{$\scriptscriptstyle \,\nearrow\,$}}}

\newcommand{\notimply}
	{\quad\hbox{$\Longrightarrow \kern -14pt {/}$}\hskip6pt\quad}

\newcommand{\lto}{\,{\lower 1pt\hbox{$\rightarrow$}}\kern -10pt
     \hbox{\raise 4pt \hbox{$\, \scriptstyle l$}}\hskip7pt}
\newcommand{\eto}{\,{\lower 1pt\hbox{$\rightarrow$}}\kern -10pt
     \hbox{\raise 4pt \hbox{$\, \scriptstyle e$}}\hskip7pt}
\newcommand{\hto}{\,{\lower 1pt\hbox{$\rightarrow$}}\kern -11pt
     \hbox{\raise 4pt \hbox{$\, \scriptstyle h$}}\hskip7pt}
\newcommand{\pto}{\,{\lower 1pt\hbox{$\rightarrow$}}\kern -11pt
     \hbox{\raise 4.5pt \hbox{$\, \scriptstyle p$}}\hskip7pt}
\newcommand{\cto}{\,{\lower 1pt\hbox{$\rightarrow$}}\kern -11pt
     \hbox{\raise 4pt \hbox{$\, \scriptstyle c$}}\hskip7pt}
\newcommand{\gto}{\,{\lower 1pt\hbox{$\rightarrow$}}\kern -11pt
     \hbox{\raise 4.5pt \hbox{$\, \scriptstyle g$}}\hskip7pt}
\newcommand{\sto}{\,{\lower 1pt\hbox{$\rightarrow$}}\kern -10pt
     \hbox{\raise 4pt \hbox{$\, \scriptstyle s$}}\hskip7pt}
\newcommand{\awto}{\,{\lower 1pt\hbox{$\rightarrow$}}\kern -15pt
     \hbox{\raise 4pt \hbox{$\, \scriptstyle aw$}}\hskip7pt}
\def\Nto{\,{\raise 1pt\hbox{$\rightarrow$}}\kern -13pt
     \hbox{\lower 3pt \hbox{$\, \scriptstyle N$}}\hskip7pt}
\def\Cto{\,{\raise 1pt\hbox{$\rightarrow$}}\kern -14pt
     \hbox{\lower 3pt \hbox{$\, \scriptstyle C$}}\hskip7pt}
\def\fto{\,{\raise 1pt\hbox{$\rightarrow$}}\kern -14pt
     \hbox{\lower 3pt \hbox{$\, \scriptstyle f$}}\hskip7pt}

\newcommand{\low}[1]{{\lower1pt \hbox{$\scriptstyle #1$}}}
\newcommand{\loww}[1]{{\lower2pt \hbox{$\scriptstyle #1$}}}
\newcommand{\high}[1]{{\raise1pt \hbox{$\scriptstyle #1$}}}

\newcommand{\cI}{{\cal I}}
\newcommand{\cJ}{{\cal J}}

\renewcommand{\liminf}{\mathop{\rm liminf}}
\renewcommand{\limsup}{\mathop{\rm limsup}}

\newcommand{\ninf}{\mathop{\rm inf}\nolimits}

\newcommand{\nargmin}{\mathop{\rm argmin}\nolimits}

\newcommand{\lwdy}[2]{\mathrel{\mathop
        {\raisebox{0.1ex}{\null$#1$}}{\hbox{\kern -1.0em
	{\raisebox{-0.8ex}{$\scriptstyle{\;\to #2}$}}}}}}
\newcommand{\lwwdy}[2]{\mathrel{\mathop
        {\raisebox{0.2ex}{\null$#1$}}{\hbox{\kern -1.0em
	{\raisebox{-1.1ex}{$\scriptstyle{\;\to #2}$}}}}}}
\newcommand{\slwdy}[2]{\scriptsize{{\mathrel{\mathop
        {\raisebox{0.1ex}{\null$#1$}}{\hbox{\kern -1.0em
	{\raisebox{-0.8ex}{$\scriptstyle{\;\to #2}$}}}}}}}}
\newcommand{\slwwdy}[2]{\scriptsize{{\mathrel{\mathop
        {\raisebox{0.2ex}{\null$#1$}}{\hbox{\kern -1.0em
	{\raisebox{-1.1ex}{$\scriptstyle{\;\to #2}$}}}}}}}}

\definecolor{lightgray}{gray}{0.75}
\definecolor{myred}{rgb}{0.55,0,0}
\definecolor{myblue}{rgb}{0,0,0.5} % hex: #00007f
\definecolor{mygreen}{rgb}{0,0.5,0} % hex: #00007f
\definecolor{purple}{rgb}{0.5,0,0.5} % hex: #00007f
\definecolor{turq}{rgb}{0,0.805,0.816} % hex: #00007f
\definecolor{maroon}{rgb}{0.51,0,0}
\definecolor{MAROON}{rgb}{0.51,0,0}
\definecolor{redor}{rgb}{0.78,0.078,0.078}
\definecolor{dgreen}{rgb}{0,0.3,0}
\newcommand{\bcdot}{\,{\raise .2ex \hbox{$\centerdot$}}\,}

\newcommand{\bbI}{\mathbb{I}}

\begin{document}

%% FOOTNOTE MARK: *, \dag ... instead of  1, 2, ...
%\renewcommand{\thefootnote}{\fnsymbol{footnote}}

\begin{center}
\begin{large}
{\bf Optimistic Bilevel Optimization with Composite Lower-Level Problem}
\smallskip
\end{large}
\vglue 0.5truecm
\begin{tabular}{cc}
  \begin{large} {\sl Mattia Solla and Johannes O. Royset} \end{large}\\
  Daniel J. Epstein Department of Industrial and Systems Engineering\\
  University of Southern California
\end{tabular}

\end{center}

\vskip 0.6truecm

\noindent {\bf Abstract}. This paper introduces a novel double regularization scheme for bilevel optimization problems whose lower-level problem is composite and convex, but not necessarily strongly convex, in the lower-level variable. The analysis focuses on the primal-dual solution mapping of the regularized lower-level problem and exploits its properties to derive an almost-everywhere formula for the gradient of the regularized hyper-objective under mild assumptions. The paper then establishes conditions under which the hyper-objective of the actual problem is well defined and shows that its gradient can be approximated by the gradient of the regularized hyper-objective. Building on these results, a gradient sampling-based algorithm computes approximately stationary points of the regularized hyper-objective, and we prove its convergence to stationary points of the actual problem. Three numerical examples from machine learning and transportation engineering demonstrate the proposed approach.

\vskip 0.2truecm


\noindent {\bf Date}:\quad \ \today 

\baselineskip=15pt

\section{Introduction}

Bilevel optimization is a class of optimization problems that deals with a hierarchical structure consisting of two levels. The lower-level decision variables are constrained to belong to the set of minimizers of a lower-level problem parametrized by the upper-level decision variables. Bilevel optimization appears in a wide range of applications, including hyperparameter optimization \cite{kunisch2013bilevel,alcantara2025unified}, adversarial learning \cite{jiang2021learning, zhang2022revisiting}, reinforcement learning \cite{yang2019provably, zheng2024safe}, economics \cite{zhu2023bilevel, mou2019bi}, and transportation \cite{meng2001equivalent, zhang2009bilevel}, and are known to be challenging due to the instability of the lower-level minimizers \cite{BeckBienstockSchmidtThurauf.23,royset2026stability}. 

In this paper, we consider a class of bilevel programs with a convex lower-level problem, possibly {\em not} strongly convex \textcolor{black}{but satisfying {\em local} properties such as a second-order sufficient condition at a lower-level minimizer for an upper-level cluster point}. The lower-level problem can be expressed by an extended real-valued, epi-polyhedral, and nondecreasing function composed with a twice continuously differentiable ($C^2$) mapping. Specifically, we adopt the optimistic perspective (cf. \cite[Chapter 5]{dempe2002foundations}) and consider the problem 
\begin{equation}\label{prob:UL}
    \underset{x\in X,y\in \reals^m}{\text{minimize  }}f(x,y) \quad \text{subject to } \ y\in Y(x)
\end{equation}
with $X=\{ x\in \reals^n ~|~ c(x)\leq 0\}$ and 
\begin{equation}\label{prob:LL}
    Y(x) = \nargmin_y g(x,y)+h\big(G(x,y)\big),
\end{equation}
where we assume {\em throughout the paper} that $f:\reals^n\times\reals^m\to\reals$ is lower bounded and continuously differentiable ($C^1$), $g:\reals^n\times\reals^m\to \reals$ is $C^2$,  $h:\reals^s\to \Reals = [-\infty,\infty]$ is proper, epi-polyhedral, and nondecreasing, $c:\reals^n\to \reals^r$ is locally Lipschitz continuous, and 
\[
G(x,y)=\big(g_1(x,y), g_2(x,y),\hdots, g_s(x,y)\big)
\]
for $C^2$ functions $g_i:\reals^n\times\reals^m\to \reals$, $i=1,\hdots,s$, with $g(x,\cdot)$ and each $g_i(x,\cdot)$ being convex for all $x\in\reals^n$. 

The lower-level problem \eqref{prob:LL} is an extended nonlinear program \cite{rockafellar1999extended} and provides a flexible structure capable of representing numerous formulations arising in applications. If $h(z)=0$ when $z\leq 0$ and $h(z) = \infty$ otherwise, then the lower-level problem amounts to an inequality constrained convex program. If $h(z) =\max\{z_1, \dots,z_{s}\}$, then it becomes a convex minimax problem. Other expressions for $h$ represent nonsmooth functions like the $\ell_1$-norm and the hinge-loss as well as arbitrary polyhedral constraints on $y$ in \eqref{prob:LL}. Current literature in bilevel optimization focuses mainly on the unconstrained and the nonlinear programming settings. The lower-level problem \eqref{prob:LL} allows us to handle directly a broader class of objectives, including many regularizers from estimation and learning, and therefore avoids reformulations that may result in a large number of additional variables and constraints. 

This paper develops descent-type algorithms for the bilevel problem \eqref{prob:UL} using first-order information about $x \mapsto f(x,Y(x))$, the {\em hyper-objective} of the problem. The hyper-objective may be set-valued (and the function $x\mapsto \min_{y\in Y(x)}f(x,y)$ may be discontinuous) when the lower-level problem is only convex but not strongly convex. In fact, even when differentiable, computing stationary points of the hyper-objective can be intractable \cite{chen2024finding}. \textcolor{black}{We address these difficulties by introducing a doubly regularized lower-level problem whose primal-dual solution mapping is globally well behaved. Using this regularization, the algorithms avoid relying on $Y$ in \eqref{prob:LL} being globally single-valued and they can be initialized at points where $Y$ is set-valued and even empty-valued. We show that any cluster point of a sequence of stationary points of the resulting regularized problems are stationary for the actual bilevel problem \eqref{prob:UL} under local regularity assumptions.} We extend the hyper-gradient method of \cite{dempe2000bundle} from the case of smooth inequality constraints to the composite form \eqref{prob:LL}, and also relax the assumptions underpinning the method. In the process, we develop verifiable conditions ensuring that the primal-dual solution mapping of \eqref{prob:LL} is single-valued and $C^1$ around a point, \textcolor{black}{extending the results in \cite{ralph1995directional} by allowing for a broader class of problems and in \cite{dontchev2009implicit} by guaranteeing local smoothness of the solution mapping instead of semidifferentiability}. The key to achieving such generalization is to leverage advances in implicit mapping theory \cite{nghia2025geometric,benko2024primal,hang2025smoothness} and, especially, \cite{hang2024role}. 

The hyper-gradient method is widely studied for special cases of \eqref{prob:UL}; see \cite[Chapter 6]{dempe2002foundations} for a classic exposition. The first complexity results for deterministic and stochastic bilevel programs appear in \cite{ghadimi2018approximation} and rely on a descent-type algorithm with further refinements in \cite{ji2021bilevel}. However, these articles only consider unconstrained lower-level problems with a strongly convex objective function.

Other recent advances in efficiently estimating the hyper-gradient include \cite{khanduri2023linearly,chen2022single,khanduri2025doubly} and focus specifically on the case when only stochastic information is available, which is then leveraged to obtain nonasymptotic guarantees of achieving a small hyper-gradient in expectation. The value-function approach of \cite{kornowski2024first,liu2022bome,kwon2023fully}, with a penalty reformulation, avoids the need for second-order information about the lower-level objective function but relies on simple constraints: \cite{chen2022single,liu2022bome,kwon2023fully} deal with unconstrained lower-level problems, while \cite{kornowski2024first,khanduri2023linearly,khanduri2025doubly} address linear constraints. In all these cases, the lower-level objective function is strongly convex.

Going beyond strong convexity, \cite{shen2023penalty,kwon2024penalty} extend methods based on penalty reformulations to handle possibly nonconvex lower-level objective functions, but they assume a growth condition at the (global) minimizers of the lower-level problem for \emph{every} $x$, a jointly smooth lower-level objective function, and an $x$-independent lower-level feasible set. The recent work \cite{chen2026set} introduces the concept of set smoothness to construct an algorithm for finding approximately Clarke stationary points of the hyper-objective, but the analysis is limited to unconstrained lower-level problems. \textcolor{black}{In \cite{lu2024first,lu2026solving}, the authors introduce a minimax reformulation of a bilevel problem and design a method that produces approximately stationary points for such reformulations when the lower-level problem is convex; under a further strong convexity assumption, they also show that such points are stationary in the hyper-gradient sense. However, these results require global assumptions on the problem data, such as a Slater condition for the lower-level problem at every $x$ and, for the hyper-gradient result, smoothness in $x$ and $y$ and strong convexity in $y$ of the lower-level objective.} It is also possible to reformulate a bilevel problem in terms of the Moreau envelope of the lower-level minimum value, obtaining stationary points under the assumption that the lower-level objective is convex in $y$ and weakly convex in $x$. The approach can rely on a difference-of-convex algorithm \cite{gao2022value,gao2026moreau} or an alternating proximal gradient method \cite{liu2024moreau} to solve this reformulation. However, \cite{liu2024moreau} considers a lower-level feasible set that is independent of $x$, while \cite{gao2022value,gao2026moreau} assume that this set is defined by smooth inequality constraints that are convex jointly in $x$ and $y$.

The seminal work in \cite{dempe2000bundle} addresses lower-level problems with a convex objective function and general convex inequality constraints by adding a norm-squared regularization to the lower-level objective function. In \cite{chen2024lower}, the authors use a similar regularization together with a relaxed reformulation based on Fenchel duality to develop an algorithm for solving a class of bilevel problems from hyperparameter optimization, though the article only shows convergence to a stationary point of the relaxed problem. We propose a novel, {\em double} regularization of the lower-level problem \eqref{prob:LL} that extends the work in \cite{dempe2000bundle}. In addition to the norm-squared term, we consider a further regularization built on Moreau envelopes. Moreau envelopes have already been used in bilevel optimization \cite{gao2022value,gao2026moreau,liu2024moreau}, but, to the best of our knowledge, not in the context of the hyper-gradient method. This double regularization enables us to analyze not only the primal solution mapping $Y$ in \eqref{prob:LL} (and the corresponding primal solution mapping of the regularized lower-level problem), but also the {\em primal-dual} solution mapping of the lower-level problem, which comes with the benefit of relaxed assumptions related to the rank of the Jacobian $\nabla_y G(x,y)$. 

Our algorithmic approach centers on computing approximately stationary points in terms of the Goldstein subdifferential of the regularized hyper-objective, or its constrained counterpart, and showing convergence to stationary points of the actual problem in the sense of Clarke, Goldstein, or extensions depending on the setting. At each iteration, such approximately stationary points are computed using a gradient sampling method (cf. \cite{burke2005robust,kiwiel2007convergence,burke2020gradient}). However, our results about the properties of the primal-dual solution mapping of the lower-level problem are flexible enough to accommodate other nonconvex nonsmooth optimization methods.

In summary, the article makes four main contributions. First, we introduce a double regularization of the lower-level problem \eqref{prob:LL} that yields a globally piecewise smooth primal-dual solution mapping merely under the assumptions stated after the problem formulation in \eqref{prob:UL}, with a computable Jacobian available under mild additional conditions. Second, we establish conditions under which the primal-dual solution mapping of \eqref{prob:LL} is single-valued and $C^1$ around a point, and provide a computable formula for the Jacobian of $Y$ in that case. Third, we prove that a cluster point $\bar x$ of approximately stationary points of the (constrained) regularized hyper-objective generated by the algorithms is stationary for the actual bilevel problem \textcolor{black}{when $\bar x$ satisfies a lower-level basic qualification and a second-order sufficient condition {\em or}, if stationarity is accepted in a weaker sense, when $\min_{y\in Y(x)} f(x, y) = \max_{y\in Y(x)} f(x, y)$ for $x$ near $\bar x$}. Fourth, we extend the method for bilevel optimization in \cite{dempe2000bundle} by including lower-level problems that are not stated in terms of smooth inequality constraints, and relaxing the assumptions presented there. In particular, the double regularization avoids the constant rank assumption (CR) of \cite{dempe2000bundle}, while the sampling scheme allows the proposed algorithms to work without the so-called (NE) assumption in \cite{dempe2000bundle}, which relates to strict complementary slackness. 

The paper is organized as follows. Section 2 summarizes definitions and notation. In Section 3, we introduce a double regularization of the lower-level objective function and analyze properties of the solution mappings of the regularized and actual lower-level problems. In Section 4, we establish the convergence of primal-dual solutions of the regularized lower-level problem to solutions of \eqref{prob:LL}. Section 5 constructs algorithms for the bilevel problem \eqref{prob:UL}, with and without upper-level constraints, and prove their convergence using results from prior sections. Section 6 contains numerical experiments.

\section{Notation and Preliminaries}

Largely following \cite{royset2021optimization,VaAn}, we use $|\cdot|$ for the Euclidean norm in $\reals^n$ and write $\ball(x,\rho)=\{x'\in \reals^n\mid |x'-x|\leq \rho\}$ for closed balls. For any subsequence $N\subset\nats=\{1,2,\dots\}$, we write $x^\nu\Nto x$ to indicate convergence of $\{x^\nu\mid \nu\in N\}$ to $x$. Given $C\subset\reals^n$, $\iota_C(x)=0$ if $x\in C$ and $\iota_C(x)=\infty$ otherwise. We use the notation $\nt C$, $\ri C$, $\pos C$, and $\con C$ to indicate the interior, relative interior, positive hull, and convex hull of $C$, respectively. The polar cone to $C$ is $C^*=\{x'\in \reals^n\mid \langle x, x'\rangle\leq 0~ \forall x\in C\}$ and, when $C$ is a subspace, its orthogonal subspace is $C^\perp =\{x'\in \reals^n\mid \langle x, x'\rangle=0~\forall x\in C\}$. The sets $C^\nu$ converge to $C$ in the sense of Painlevé-Kuratowski, written as $C^\nu\sto C$, if $C$ is closed and the point-to-set distance $\dist(x,C^\nu)\to \dist(x,C)$ for all $x\in \reals^n$; recall that $\dist(x,C)=\inf_{x'\in C}|x-x'|$. Given vectors $x,y\in \reals^n$, we write $x\geq y$ if the inequality holds component-wise.

The domain of $f:\reals^n\to \Reals$ is $\dom f=\{x\in \reals^n\mid f(x)<\infty\}$ and the epigraph is $\epi f =\{(x,\alpha)\mid f(x)\leq \alpha\}$. The function is proper if $\dom f\neq \emptyset$ and $f(x)\neq -\infty$ for all $x\in \reals^n$. It is lower semicontinuous (lsc) if $\epi f$ is a closed subset of $\reals^n\times \reals$. It is nondecreasing if $x\geq y$ implies $f(x)\geq f(y)$. If $f$ is proper lsc and $\lambda>0$, then the Moreau envelope $e_{\lambda}f$ of $f$ is defined as
\begin{equation}\label{def:moreau-env}
    e_{\lambda }f(x) = \ninf_w \big\{f(w)+\tfrac{1}{2\lambda}|w-x|^2\big\}. 
\end{equation}
If $f$ is also convex, then $e_{\lambda}f$ is convex and $C^1$, with $\nabla e_{\lambda}f(x) = (x-w_\lambda)/\lambda$, 
where $w_\lambda$ is the minimizer attaining the infimum in \eqref{def:moreau-env}. The functions $f^\nu$ epi-converge to a function $f$, denoted by $f^\nu \eto f$, if $\epi f^\nu\sto \epi f$.

For a nonempty set $C\subset \reals^n$ and $\bar{x}\in C$, the tangent cone to $C$ at $\bar{x}$ is $T_C(\bar{x})=\big\{w\in \reals^n ~\big|~ \exists\{x^\nu\}\subset C, \tau^\nu\searrow0 \text{ such that } (x^\nu-\bar{x})/\tau^\nu \to w  \big\}$.
The regular normal cone to $C$ at $\bar{x}$ is $\widehat{N}_C(\bar{x})=T_C(\bar{x})^*$ and the normal cone to $C$ at $\bar{x}$, denoted by $N_C(\bar{x})$, is the set of $u\in \reals^n$ such that there exist $x^\nu\in C\to \bar{x}$ and $u^\nu\in \widehat{N}_C(x^\nu)\to u$. For $f:\reals^n\to \Reals$ and a point $x$ at which $f(x)$ is finite, a vector $u\in \reals^n$ is a subgradient of $f$ at $x$ when $(u,-1)\in N_{\epi f}(x,f(x))$. The set of all subgradients of $f$ at $x$ is the subdifferential $\partial f(x)$. If $f$ is also lsc, we say that $u$ is a horizon subgradient of $f$ at $x$, denoted $u\in \partial^\infty f(x)$, if $(u,0)\in N_{\epi f}(x,f(x))$. The critical cone of $f$ at $x$ for $u\in \partial f(x)$ is $K_f(x,u)$ $=$ $N_{\partial f(x)}(u)$. If $f$ is locally Lipschitz continuous, then its Clarke and Goldstein subdifferentials at $x$ are 
\begin{equation*}
    \bar{\partial }f(x) = \con \partial f(x) = \con \{ u \,|\, \exists x^\nu \in D \text{ s.t. } x^\nu\to x, \nabla f(x^\nu)\to u\};~   \partial_\epsilon f(x)=\con \{\bar{\partial}f(z)\,|\, z\in \ball(x,\epsilon)\},
\end{equation*}
respectively, where $\epsilon > 0$ and $D\subset \reals^n$ is the set of points at which $f$ is differentiable, which has full measure by Rademacher's theorem. For $w\in\reals^m$ and a $C^2$ mapping $F:\reals^n\to \reals^m$, with $F(x) = (f_1(x), \dots, f_m(x))$, we use the notation $\nabla_{xx}^2F(x)w=\sum_{i=1}^m\nabla_{xx}^2 f_i(x)w_i$ to indicate the Hessian of $x\mapsto \langle F(x), w\rangle$.

For a matrix $A\in \reals^{m\times n}$, we define $\ker A=\{x\in \reals^n\mid Ax=0\}$ and $\spn A = \{y\in \reals^m\mid \exists x\in \reals^n, Ax=y\}$. We denote by $\bbI_n$ the $n\times n$-identity matrix. 

A function $f:\reals^n\to \Reals$ is epi-polyhedral if its epigraph is a polyhedral set. Then there are finite index sets $I$ and $J$ as well as $\{a^j\in\reals^n, b^i\in\reals^n, \alpha_j\in\reals, \beta_i\in\reals, j\in J, i\in I\}$ such that 
\begin{equation*}
    f(x)=\max_{j\in J}\{\langle a^j,x\rangle-\alpha_j\}+\iota_{\dom f}(x), ~\text{with}~\dom f=\{x\in \reals^n\mid \langle b^i,x\rangle-\beta_i\leq 0\ \forall i\in I\}.
\end{equation*}
The active index sets at $x\in \dom f$ are $I(x)=\{i\in I\mid \langle b^i,x\rangle-\beta_i= 0\}$ and $J(x)=\{j\in J\mid f(x)=\langle a^j,x\rangle-\alpha_j\}$. Then, by \cite[Proposition 3.3]{mordukhovich2016generalized}, one has $u\in \partial f(x)$ if and only if there exists a representation
\begin{equation}\label{decomposition}
    u=\sum_{j\in J(x)}\sigma_j a^j+\sum_{i\in I(x)}\tau_ib^i
\end{equation}
with $\sigma_j,\tau_i\geq0$ and $\sum_{j\in J(x)}\sigma_j=1$. For any such representation, we define the sets of strongly active indices by $I^+(x,\tau)=\{i\in I(x)\mid \tau_i>0\}$ and $J^+(x,\sigma)=\{j\in J(x)\mid \sigma_j>0\}$. By \cite[Proposition 3.3]{mordukhovich2016generalized}, $w\in K_f(x,u)$ if and only if
\begin{equation}\label{sys:kh}
\begin{aligned}
    \langle a^i-a^j,w\rangle& =0 \quad i,j\in J^+(x,\sigma), & \quad \langle a^i-a^j,w\rangle&\leq 0 \quad  i\in J(x)\backslash J^+(x,\sigma), ~j\in J^+(x,\sigma), \\
    \langle b^i,w\rangle&=0 \quad  i\in I^+(x,\tau), & \quad \langle b^i,w\rangle&\leq 0 \quad  i\in I(x)\backslash I^+(x,\tau).
\end{aligned}
\end{equation}
For $x\in \dom f$, $u\in \partial f(x)$ and a representation $(\sigma,\tau)$ of $u$, we define the families of sets
\begin{equation}\label{def:cicj}
    \cI(x,\tau) = \{ I^*\mid I^+(x,\tau)\subset I^*\subset I(x)\}, \qquad \cJ(x,\sigma) = \{J^*\mid J^+(x,\sigma)\subset J^*\subset J(x)\}.
\end{equation}

From \cite[Corollary 3.1]{hang2024role}, we obtain relationships between an epi-polyhedral function $f$ and its conjugate $f^*$: $u \in \ri(\partial f(x))$ if and only $x\in \ri(\partial f^*(u))$. Moreover, for $u\in \partial f(x)$, one has
\begin{equation}\label{eq:polar}
    K_f(x,u) = (K_{f^*}(u,x))^*.
\end{equation}

The graph of a set-valued mapping $F:\reals^n\rightrightarrows\reals^m$ is $\gph F = \{(x,y)\mid y\in F(x)\}$. For $(x,y)\in \gph F$, we say that $F$ has a single-valued localization around $x$ for $y$ if there are neighborhoods $U$ and $V$ of $x$ and $y$ such that the mapping $x\mapsto F(x)\cap V$ is single-valued when restricted to $U$.
The following theorem, foundational for our development, describes the conditions under which the solution mapping of a class of generalized equations is $C^1$.
\begin{theorem}\label{thm:hang} {\rm (\cite[Theorem 5.3]{hang2024role}).} 
    For $C^1$ mapping $\psi:\reals^n\to \reals^n$ and epi-polyhedral function $f:\reals^n\to \Reals$, define the set-valued mapping
    \[
    s(u)=\{x\in \reals^n\mid u\in \psi(x)+\partial f(x)\}, ~~~~u\in\reals^n.
    \]
    Suppose that $\bar{x}$ satisfies $-\psi(\bar{x})\in \ri(\partial f(\bar{x}))$. Then $\overline{K}=K_f(\bar{x},-\psi(\bar{x}))$ is a subspace and $s$ has a Lipschitz continuous single-valued localization (which we also call $s$) around $0\in \reals^n$ for $\bar{x}$ if and only if
    \begin{equation}\label{cond:b}
        \{w\in \reals^n\mid \nabla\psi(\bar{x})^\top w\in \overline{K}^\perp\}\cap \overline{K}=\{0\}.
    \end{equation}
    In this case, $s$ is $C^1$ in a neighborhood of 0 and $\nabla s(u) = B(B^\top \nabla \psi(s(u))B)^{-1}B^\top$
    for all $u$ sufficiently close to 0, where $B$ is a matrix whose columns form a basis of $\overline{K}$.
\end{theorem}

For the {\em remainder of the paper}, we assume without loss of generality that the proper nondecreasing epi-polyhedral function $h:\reals^s\to \Reals$ in \eqref{prob:LL} is of the form 
\begin{equation}\label{def:h}
    h(z)=\max_{j\in J}\{\langle a^j,z\rangle-\alpha_j\}+\iota_{\dom h}(z),
\end{equation}
with $\dom h=\{z\in \reals^s\mid \langle b^i,z\rangle-\beta_i\leq 0\ \forall i\in I\}$, for some finite index sets $I$ and $J$ and $a^j,b^i\geq 0$, $\alpha_j,\beta_i\in \reals$.

\section{Properties of Solution Mappings}
 
As a stepping stone toward algorithms for the bilevel problem \eqref{prob:UL}, this section introduces a doubly regularized approximation of the lower-level problem \eqref{prob:LL}. This allows us to circumvent the troublesome set-valued mapping $Y$ and focus on the desirable properties of primal {\em and} dual solutions of the approximation. We also return to the actual lower-level problem, and study the properties of its primal-dual solution mapping under additional local assumptions, which provide the foundation for optimality conditions for the bilevel problem. 

\textcolor{black}{We introduce assumptions with different scopes and purposes. The standing assumptions stated in Section~1 suffice for the primal-dual solution mapping of the doubly regularized lower-level approximation to be globally single-valued and piecewise smooth. Pointwise relative-interior conditions enter only for the purpose of a computable Jacobian formula. A basic qualification and a second-order condition provide local stability of the lower-level problem \eqref{prob:LL} near a specified primal-dual solution; they are not assumed to hold for every upper-level point $x$. Table \ref{tab:assumptions} summarizes the assumptions as they appear throughout the paper.}

\begin{table}[h]\color{black}
\centering
\caption{Summary of assumptions used in Sections~3--5.}
\label{tab:assumptions}
\small
\setlength{\tabcolsep}{4pt}
\renewcommand{\arraystretch}{1.18}

\begin{tabularx}{\textwidth}{
    >{\raggedright\arraybackslash}p{0.20\textwidth}
    >{\raggedright\arraybackslash}p{0.22\textwidth}
    >{\raggedright\arraybackslash}p{0.26\textwidth}
    >{\raggedright\arraybackslash}p{0.22\textwidth}
}
\toprule
Condition
& How imposed
& Purpose
& Main results \\
\midrule

Relative interior condition. 
& At a primal-dual solution of \eqref{prob:LL} or \eqref{prob:LLr}.
& Local \(C^1\) solution mapping and computable Jacobian.
& Theorems \ref{thm:c1}, \ref{thm:c1nor}, and \ref{prop:conv2}; Proposition \ref{prop:conv}. \\

\addlinespace
Basic qualification \eqref{CQ:llnor}.
& At a primal solution of \eqref{prob:LL}.
& Valid primal-dual representation and stability of multipliers.
& Theorems \ref{thm:c1nor}, \ref{thm:lipnor}-\ref{thm:const}; Propositions \ref{prop:convS}-\ref{prop:weak}. \\

\addlinespace
Second-order condition \eqref{SSOC}.
& At a primal-dual solution of \eqref{prob:LL}.
& Local single-valuedness of actual lower-level solution mapping.
& Theorems~\ref{thm:c1nor}, \ref{thm:lipnor}-\ref{thm:const};
  Propositions~\ref{prop:convS} and \ref{prop:conv}.\\

\addlinespace
Analytic \(g\) and \(G\).
& Globally.
& Sampling encounters points of smoothness almost surely.
& Theorems~\ref{thm:alg}-\ref{thm:const}.\\

\bottomrule
\end{tabularx}
\end{table}

\subsection{Regularized Lower-Level Problem}

For $x\in \reals^n$ and $\alpha,\beta>0$, we define a doubly regularized lower-level problem and its argmin-mapping:
\begin{equation}\label{prob:LLr}
    Y_{\alpha,\beta}(x) = \nargmin_y g(x,y)+e_{\alpha }h(G(x,y)) + \frac{\beta}{2}|y|^2.
\end{equation}
By the following proposition, a necessary and sufficient condition for $y$ to be a minimizer in \eqref{prob:LLr} is that  
\begin{equation}\label{opt:ge}
    0\in \psi_x(y,p)+\partial\tilde{h}(y,p)
\end{equation}
for some multiplier vector $p\in \reals^s$, where 
\begin{equation}\label{for:psih}
\psi_x(y,p)=\begin{pmatrix}
        -\nabla_y g(x,y)-\nabla_y G(x,y)^\top p-\beta y\\
        -G(x,y) + \alpha p
    \end{pmatrix},\quad\quad \tilde{h}(y,p)=h^\ast(p). 
\end{equation}
Corresponding to the primal-dual solutions from \eqref{prob:LLr}, we denote by 
\begin{equation*}
        S_{\alpha,\beta}(x)=\lset(y,p)\mset y \text{ and }p \text{ satisfy } \eqref{opt:ge}\rset 
\end{equation*}
the solution mapping of the generalized equation \eqref{opt:ge} as a function of $x$.

\begin{proposition}\label{prop:lagrangian}
    Given $\alpha,\beta>0$ and $x\in \reals^n$, the sets $Y_{\alpha,\beta}(x)$ and $S_{\alpha,\beta}(x)$ are singletons. Moreover, 
    \[
    \{y\} = Y_{\alpha,\beta}(x)  ~~\Longleftrightarrow~~  \exists p\in \reals^s \text{ with } \{(y,p)\} = S_{\alpha,\beta}(x). 
    \]
\end{proposition}
\begin{proof}
    \textcolor{black}{Since $h$ is proper, convex and lsc, its Moreau envelope is finite, convex, and $C^1$ with
    \begin{equation*}
        \nabla e_\alpha h(z) = (z-w^\star)/\alpha,
    \end{equation*}
    where $\{w^\star\}=\nargmin_w \{h(w)+\tfrac{1}{2\alpha}|w-z|^2\}$. Set $p=(z-w^\star)/\alpha$. By the Fermat rule, $p$ satisfies $p\in \partial h(w^\star)$. In particular, since $h$ is nondecreasing, $p\geq 0$. This implies that $\nabla e_\alpha h(z)$ is nonnegative for all $z\in \reals^s$, and so $e_\alpha h$ is nondecreasing.}
    Since $e_{\alpha }h$ is also convex, the composition $y\mapsto e_{\alpha}h( G(x,y))$ is convex and then the objective function in \eqref{prob:LLr} is strongly convex. 

    Turning to an optimality condition for \eqref{prob:LLr}, we use tools from \cite[Chapter 5]{royset2021optimization} and define the Rockafellian $\varphi(y,u)=g(x,y)+e_{\alpha}h(G(x,y)+u)+\frac{\beta}{2}|y|^2$; cf. \cite[Proposition 5.16]{royset2021optimization}. By \cite[Proposition 5.28]{royset2021optimization}, the Lagrangian associated with this Rockafellian is 
    \begin{equation*}
\begin{aligned}
    L_{\alpha,\beta}(y,p) =g(x,y)+\frac{\beta}{2}|y|^2+\langle G(x,y),p\rangle-h^*(p)-\frac{\alpha}{2}|p|^2.
\end{aligned}
\end{equation*}
    By \cite[Theorems 4.64 and 4.75, and Proposition 5.36]{royset2021optimization}, combined with the convexity of the objective function of \eqref{prob:LLr}, $y\in Y_{\alpha,\beta}(x)$ if and only if there exists $p\in \reals^s$ such that $0\in \partial_y L_{\alpha,\beta}(y,p)$ and $0\in \partial_p(-L_{\alpha,\beta})(y,p)$. This results in the system
\begin{equation}\label{opt:ll}
    \begin{pmatrix}
        0\\
        0
    \end{pmatrix} \in \begin{pmatrix}
        -\nabla_y g(x,y)-\nabla_y G(x,y)^\top p-\beta y\\
        -G(x,y) + \alpha p
    \end{pmatrix} + \{0\}\times \partial h^\ast(p),
\end{equation}
    which is equivalent to \eqref{opt:ge}. Moreover, since $ L_{\alpha,\beta}(\cdot,p)$ is strongly convex for all $p\geq 0$ (note that $\dom h^*\subset [0,\infty)^s$, and so if $p\geq 0$ does not hold, then $ L_{\alpha,\beta}(y,p)=-\infty$ for all $y$) and $ L_{\alpha,\beta}(y,\cdot)$ is strongly concave for all $y\in \reals^m$, we conclude that there exists a unique pair $(y,p)$ that satisfies \eqref{opt:ll} and therefore \eqref{opt:ge}. With this, $S_{\alpha,\beta}(x)=\{(y,p)\}$ and $Y_{\alpha,\beta}(x)=\{y\}$, concluding the proof. 
\end{proof}

Since the proposition establishes that $S_{\alpha,\beta}$ and $Y_{\alpha,\beta}$ are single-valued, we will treat them as such for the remainder of the paper. Proposition \ref{prop:lagrangian} holds even when $Y(x)=\emptyset$.

\subsection{Local Properties}

In order to compute a descent direction for $x\mapsto f(x,Y_{\alpha,\beta}(x))$ at a point, we seek conditions ensuring smoothness at the point as well as a convenient formula for the Jacobian of $Y_{\alpha,\beta}$. 

In \eqref{for:psih}, $\psi_x$ is a $C^1$ mapping and $\tilde{h}$ is epi-polyhedral, since the conjugate of an epi-polyhedral function is also epi-polyhedral \cite[Theorem 11.14]{VaAn}. Hence, \eqref{opt:ge} fits the framework of Theorem \ref{thm:hang}.  

From the earlier discussion (see the proof of Proposition \ref{prop:lagrangian}), the Lagrangian associated with the regularized lower-level problem in \eqref{prob:LLr} is
\begin{equation}\label{eq:lagrangian}
    L_{\alpha,\beta}(x,y,p) = g(x,y)+\frac{\beta}{2}|y|^2+\langle G(x,y),p\rangle -h^*(p)-\frac{\alpha}{2}|p|^2,
\end{equation}
where we highlight the dependence on $x$. With this notation, we are now ready to enunciate the first theorem of the section.
\begin{theorem}\label{thm:c1}{\rm (local properties of regularized solution mapping).}
    For $\alpha,\beta>0$ and $\bar x\in \reals^n$, suppose that $(\bar{y},\bar{p})=S_{\alpha,\beta}(\bar{x})$ and the relative interior condition $G(\bar{x},\bar{y})-\alpha \bar{p}\in \ri(\partial h^*(\bar{p}))$ is satisfied. Then $S_{\alpha,\beta}$ is $C^1$ in a neighborhood of $\bar{x}$. Moreover, we have that
    \begin{equation}\label{for:gradSr}
        \nabla S_{\alpha,\beta}(\bar{x})= -B(B^\top AB)^{-1}B^\top\begin{pmatrix}
            \nabla_x\nabla_y L_{\alpha,\beta}(\bar{x},\bar{y},\bar{p}) \\
            \nabla_xG(\bar{x},\bar{y})
        \end{pmatrix},
    \end{equation}
    where 
    $$A=\begin{pmatrix}
            \nabla_{yy}^2 L_{\alpha,\beta}(\bar{x},\bar{y},\bar{p}) ~~ & ~~\nabla_y G(\bar{x},\bar{y})^\top \\
                 \nabla_y G(\bar{x},\bar{y}) & -\alpha \bbI_s
            \end{pmatrix}$$  and $B$ is a matrix whose columns form a basis of $\reals^m \times K_{h^\ast}(\bar{p},G(\bar{x},\bar{y})-\alpha \bar{p})$.
\end{theorem}
\begin{proof} Since $G(\bar{x},\bar{y})-\alpha \bar{p}\in \ri(\partial h^*(\bar{p}))$, one has $-\psi_{\bar{x}}(\bar{y},\bar{p})\in \ri(\partial\tilde{h}(\bar{y},\bar{p}))$. Therefore, the relative interior condition in Theorem \ref{thm:hang} is satisfied. We will prove that condition \eqref{cond:b} of the theorem also holds. Since $\partial \tilde{h}(\bar{y},\bar{p})= \{0\}\times \partial h^*(\bar{p})$, by \cite[Proposition 4.44]{royset2021optimization} we have that for all $(\lambda_1,\lambda_2)\in \partial \tilde{h}(\bar{y},\bar{p})$
\begin{equation*}
    N_{\partial \tilde{h}(\bar{y},\bar{p})}(\lambda_1,\lambda_2) = \reals^m \times N_{\partial h^*(\bar{p})}(\lambda_2).
\end{equation*}
It follows by definition of the critical cone that 
\begin{equation}\label{eq:criticalcone}
    K_{\tilde{h}}\big((\bar{y},\bar{p}),(-\nabla_yg(\bar{x},\bar{y})-\nabla_yG(\bar{x},\bar{y})^\top \bar{p}-\beta \bar{y}, \, G(\bar{x},\bar{y})-\alpha \bar{p})\big)= \reals^m \times K_{h^\ast}(\bar{p},G(\bar{x},\bar{y})-\alpha \bar{p}).
\end{equation}
Let $(w,z)\in K=\reals^m \times K_{h^\ast}(\bar{p},G(\bar{x},\bar{y})-\alpha \bar{p})$, that is, $w\in \reals^m$ and $z\in \overline{K}=K_{h^\ast}(\bar{p},G(\bar{x},\bar{y})-\alpha \bar{p})$. Condition \eqref{cond:b} as applied to \eqref{opt:ge} amounts to 
\begin{equation*}
    (w,z)\in K ~~\text{ and } ~~\nabla \psi_{\bar{x}}(\bar{y},\bar{p})(w,z) \in K^\perp  \quad \Longrightarrow \quad w=0, ~z=0.
\end{equation*}
Since $\nabla\psi_{\bar{x}}(\bar{y},\bar{p})=-A$ and $K^\perp=\{0\}\times \overline{K}^\perp$, we obtain that $\nabla \psi_{\bar{x}}(\bar{y},\bar{p})(w,z) \in K^\perp$ is equivalent to
\begin{equation}\label{cond:ssoc}
        -\nabla_y G(\bar{x},\bar{y})^\top z=Mw, ~~~   -\nabla_yG(\bar{x},\bar{y})w+\alpha z\in \overline{K}^\perp,
\end{equation}
with $M=\beta \bbI_m+\nabla_{yy}^2G(\bar{x},\bar{y})\bar{p}+\nabla_{yy}^2g(\bar{x},\bar{y})$. 

The matrix $\nabla_{yy}^2g_i(\bar{x},\bar{y})$ is positive semidefinite for every $i$ because $g_i$ is convex in $y$ and the same is true for $\nabla_{yy}^2g(\bar{x},\bar{y})$. Since $h$ is nondecreasing and $\bar{p}\in \partial h(G(\bar{x},\bar{y})-\alpha \bar{p})$, one has $\bar{p}\geq 0$. Therefore, the matrix $M$ is positive definite. By taking inverse and replacing $w$ in the inclusion in \eqref{cond:ssoc}, we obtain
\begin{equation*}
    \big (\nabla_yG(\bar{x},\bar{y})M^{-1}\nabla_y G(\bar{x},\bar{y})^\top+\alpha \bbI_s\big ) z \in \overline{K}^\perp,
\end{equation*}
which, since $z\in \overline{K}$, implies
\begin{equation*}
    \big\langle z, \big (\nabla_yG(\bar{x},\bar{y})M^{-1}\nabla_y G(\bar{x},\bar{y})^\top+\alpha \bbI_s\big ) z\big\rangle =0.
\end{equation*}
The matrix in the previous equality is positive definite since $M$ has that property. Then the equality is only satisfied if $z=0$, which results in $w=-M^{-1}\nabla_y G(\bar{x},\bar{y})^\top z=0$. Therefore, the condition \eqref{cond:b} is satisfied.

Theorem \ref{thm:hang} tells us that the solution mapping of \eqref{opt:ge} (as a function of the left hand side) admits a single-valued $C^1$ localization around 0 for $\bar{y}$ and $\bar{p}$; this implies that the solution mapping itself is $C^1$ locally around 0 since we already know it is single-valued. Specifically, for $(q_1, q_2)$ sufficiently close to 0, one has: 
\begin{equation*}
\text{$(y',p')$ ~solves }~~ 
    \begin{pmatrix}
        q_1\\
        q_2
    \end{pmatrix} \in \begin{pmatrix}
        -\nabla_yg(\bar{x},y)-\nabla_y G(\bar{x},y)^\top p-\beta y\\
        -G(\bar{x},y) + \alpha p
    \end{pmatrix} + \{0\}\times \partial h^\ast(p)
\end{equation*}
if and only if
\begin{equation*}
    \begin{pmatrix}
        y'\\
        p'
    \end{pmatrix} = \begin{pmatrix}
        \bar{y} \\
        \bar{p}
    \end{pmatrix} +B(B^\top AB)^{-1}B^\top \begin{pmatrix}
        q_1 \\
        q_2
    \end{pmatrix} + e(q_1,q_2),
\end{equation*}
where the term $e(q_1,q_2)$ contains the second-order error. In particular, it follows from elementary calculus that the mapping $e:\reals^m\times\reals^s\to \reals^{m+s}$ satisfies $e(q_1,q_2)=o(|(q_1,q_2)|)$ whenever $q_1,q_2\to 0$, and therefore $\nabla e(0,0)=0$.
\par
Let us now consider $x'$ sufficiently close to $\bar{x}$. One has by straightforward algebra that $(y',p')$ is the solution of \eqref{opt:ge} at $x'$ if and only if
\begin{equation}\label{eq:q}
\begin{aligned}
    \begin{pmatrix}
        -\nabla_yg(\bar{x},y')-\nabla_yG(\bar{x},y')^\top p'+\nabla_yg(x',y')+\nabla_yG(x',y')^\top p'\\
        -G(\bar{x},y')+G(x',y')
    \end{pmatrix} &\in \\
    \begin{pmatrix}
        -\nabla_yg(\bar{x},y')-\nabla_y G(\bar{x},y')^\top p'-\beta y'\\
        -G(\bar{x},y') + \alpha p'
    \end{pmatrix} &+ \{0\}\times \partial h^\ast(p').
\end{aligned}
\end{equation}
Define the mapping $q:\reals^n\times \reals^m\times \reals^s\to \reals^m\times \reals^s$ by
\begin{equation*}
    q(x,y,p) = \begin{pmatrix}
        -\nabla_yg(\bar{x},y)-\nabla_yG(\bar{x},y)^\top p+\nabla_yg(x,y)+\nabla_yG(x,y)^\top p\\
        -G(\bar{x},y)+G(x,y)
    \end{pmatrix}.
\end{equation*}
The mapping $q$ is $C^1$ and thus continuous. By the previous discussion, \eqref{eq:q} implies
\begin{equation*}
     \begin{pmatrix}
        0\\
        0
    \end{pmatrix} = \begin{pmatrix}
        y'-\bar{y} \\
        p'-\bar{p}
    \end{pmatrix} -B(B^\top AB)^{-1}B^\top q(x',y',p') + e(q(x',y',p')).
\end{equation*}
Note that we have
\begin{equation*}
    \begin{aligned}
        &\lim_{(x',y',p')\to (\bar{x},\bar{y},\bar{p})} \frac{|e(q(x',y',p'))|}{|(x',y',p')-(\bar{x},\bar{y},\bar{p})|} =\\
        & = \lim_{(x',y',p')\to (\bar{x},\bar{y},\bar{p})} \frac{|q(x',y',p')|}{|(x',y',p')-(\bar{x},\bar{y},\bar{p})|}\frac{|e(q(x',y',p'))|}{|q(x',y',p')|} \\
        & \leq \kappa\cdot\lim_{(x',y',p')\to (\bar{x},\bar{y},\bar{p})} \frac{|e(q(x',y',p'))|}{|q(x',y',p')|}=0,
    \end{aligned}
\end{equation*}
where $\kappa$ is the Lipschitz constant of $q$ in a neighborhood around $(\bar{x},\bar{y},\bar{p})$. Therefore,
$$e(q(x',y',p'))=o(|(x',y',p')-(\bar{x},\bar{y},\bar{p})|)\  \text{ when } (x',y',p')\to (\bar{x},\bar{y},\bar{p})$$  and then $\nabla e(q_1(\bar{x},\bar{y},\bar{p}),q_2(\bar{x},\bar{y},\bar{p}))=0$.  
Recall that $(y',p')= S_{\alpha,\beta}(x')$ and define 
\begin{equation*}\Phi(x,y,p)=\begin{pmatrix}
        y-\bar{y} \\
        p-\bar{p}
    \end{pmatrix} -B(B^\top AB)^{-1}B^\top q(x,y,p) + e(q(x,y,p)).
\end{equation*}
Since $\Phi(x',y',p')=0$ for any $x'$ sufficiently close to $\bar{x}$, we can study the mapping $S_{\alpha,\beta}$ using the implicit function theorem. Denoting by $\nabla_{y,p}\Phi(\bar{x},\bar{y},\bar{p})$ the Jacobian of $\Phi(\bar{x},\cdot,\cdot)$ at $(\bar{y},\bar{p})$, we have
\begin{equation*}
    \nabla_{y,p}\Phi(\bar{x},\bar{y},\bar{p}) = \bbI_{m+s}+\nabla_{y,p}e\circ q(\bar{x},\bar{y},\bar{p}) -B(B^\top AB)^{-1}B^\top \nabla_{y,p} q(\bar{x},\bar{y},\bar{p})  =\bbI_{m+s},
\end{equation*}
which is invertible, and
\begin{equation*}
\begin{aligned}
    \nabla_x\Phi(\bar{x},\bar{y},\bar{p}) & =-B(B^\top AB)^{-1}B^\top \nabla_x q(\bar{x},\bar{y},\bar{p})+\nabla_x e(q(\bar{x},\bar{y},\bar{p})) \\
    &=-B(B^\top AB)^{-1}B^\top  \begin{pmatrix}
    \nabla_x\big(\nabla_yg(\bar{x},\bar{y})+\nabla_yG(\bar{x},\bar{y})^\top\bar{p}\big)\\
        \nabla_xG(\bar{x},\bar{y})
    \end{pmatrix}.
\end{aligned}
\end{equation*}
Thus, by the implicit function theorem, $S_{\alpha,\beta}$ is $C^1$ around $\bar{x}$ and the formula for $\nabla S_{\alpha,\beta}(\bar{x})$ in \eqref{for:gradSr} holds. 
\end{proof}

A parallel result ensures that the primal-dual solution mapping of the actual lower-level problem in \eqref{prob:LL} is single-valued and $C^1$ around a point $x$. However, it requires additional assumptions at $x$. In preparation for such a result, we say that the basic qualification holds at $y$ for $x$ when 
\begin{align}\tag{BQ}\label{CQ:llnor}
    u\in \partial^{\infty}h(G(x,y)) \ \text{ and }\ \nabla_yG(x,y)^\top u =0\quad \Longrightarrow \quad u=0.
\end{align}
By \cite[Theorem 4.64]{royset2021optimization}, under \eqref{CQ:llnor}, we have that $y\in Y(x)$ if and only if there exists $p\in \reals^s$ such that
\begin{equation}\label{opt:llnor}
    \begin{pmatrix}
        0\\
        0
    \end{pmatrix} \in \begin{pmatrix}
        -\nabla_yg(x,y)-\nabla_y G(x,y)^\top p\\
        -G(x,y)
    \end{pmatrix} + \{0\}\times \partial h^\ast(p).
\end{equation}
In parallel to the primal-dual solution map $S_{\alpha,\beta}$ of the regularized-lower level problem, we define 
\begin{equation*}
S(x)=\{(y,p)\mid y \text{ and }p \text{ satisfy } \eqref{opt:llnor} \text{ at }x\}. 
\end{equation*}
Unlike $S_{\alpha,\beta}$, the mapping $S$ is not necessarily single-valued. This is due to the fact that a Lagrangian of \eqref{prob:LL} expressed by 
\begin{equation}\label{eqn:LagrangianActual}
    L(x,y,p) = g(x,y)+\langle G(x,y),p\rangle -h^*(p)
\end{equation}
is convex in $y$ and concave in $p$, but loses the strong convexity and concavity of $L_{\alpha,\beta}$. (The proof of Proposition \ref{prop:lagrangian} includes essentially the arguments leading to this Lagrangian.)\par
We introduce an assumption equivalent to the standard coderivative criterion for solution mappings of generalized equations (cf. \cite[Theorem 9.40]{VaAn}). For $x\in \reals^n$ and $(y,p)\in S(x)$ such that \eqref{CQ:llnor} holds at $y$ for $x$, we say that the second-order sufficient condition \eqref{SSOC} holds at $y$ for $x$ and $p$ if
\begin{align}\tag{SOSC}\label{SSOC}
        \begin{matrix}
        -\nabla_y G(x,y)^\top z=\nabla_{yy}^2L(x,y,p)w \\
        (z,-\nabla_yG(x,y)w)\in  N_{\gph \partial h}(G(x,y),p)    %\qquad D^\ast(\partial h)(G(x,y),p)(\nabla_y G(x,y)w)
    \end{matrix}\quad \Longrightarrow \quad w=0, ~z=0.
\end{align}

\begin{theorem}\label{thm:c1nor}{\rm (local properties of actual solution mapping).}
    Given $x\in \reals^n$, suppose that there exist $(y,p)\in S(x)$ such that the basic qualification \eqref{CQ:llnor} holds at $y$ for $x$ and \eqref{SSOC} holds at $y$ for $x$ and $p$. Suppose also that the relative interior condition $G(x,y)\in \ri(\partial h^*(p))$ is satisfied.
    Then $S$ is single-valued and $C^1$ around $x$, with $S(x)=\{(y,p)\}$. Moreover, it satisfies
    \begin{equation}\label{for:gradS}
        \nabla S(x)= -B(B^\top AB)^{-1}B^\top\begin{pmatrix}
            \nabla_x\nabla_y L(x,y,p) \\
            \nabla_xG(x,y)
        \end{pmatrix},
    \end{equation}
    where 
    $$A=\begin{pmatrix}
            \nabla_{yy}^2 L(x,y,p) ~~&~~ \nabla_y G(x,y)^\top \\
                 \nabla_y G(x,y) & 0
            \end{pmatrix} $$ and $B$ is a matrix whose columns form a basis of $\reals^m \times K_{h^\ast}(p,G(x,y))$.
\end{theorem}
\begin{proof} 
By \cite[Proof of Theorem 5.3]{hang2024role}, \eqref{SSOC} is equivalent to
\begin{equation*}
    \begin{aligned}
        -\nabla_y G(x,y)^\top z&=\nabla_{yy}^2L(x,y,p)w \\
         -\nabla_yG(x,y)w&\in \big(K_{h^*}(p, G(x,y))\big)^\perp \\
        z&\in K_{h^*}(p, G(x,y))
    \end{aligned}\quad  \Longrightarrow \quad w=0,~ z=0.
\end{equation*} 
We can then use the same argument as in the proof of Theorem \ref{thm:c1}, with \eqref{SSOC} guaranteeing that the condition \eqref{cond:b} in Theorem \ref{thm:hang} is satisfied. By \cite[Theorem 3.2]{mordukhovich2017stability}, the existence of a single-valued Lipschitz localization of the solution mapping of \eqref{opt:ge} as a function of the left hand side around $0$ for $y$ and $p$ is equivalent to $y$ being a fully stable local minimizer of \eqref{prob:LL}, which in turn implies that $Y(x)$ is a singleton by convexity of problem \eqref{prob:LL}. Then, the convexity of the set $\{p' \mid (y,p')\in S(x)\}$, which comes from the convexity of $-L(x,y,\cdot)$, implies that $p$ is also unique and so $S$ itself is single-valued and $C^1$ around $x$. The formula in \eqref{for:gradS} follows then by repeating the argument in the proof of Theorem \ref{thm:c1}.
\end{proof}

The second-order condition \eqref{SSOC} is satisfied, for example, if the matrix $\nabla_{yy}^2L(x,y,p)$ is positive definite and the matrix $\nabla_yG(x,y)$ has full row rank. Indeed, by \cite[Theorem 3.3(b)]{hang2024role}, $(z,-\nabla_yG(x,y)w)\in  N_{\gph \partial h}(G(x,y),p)$ if and only if there exist faces $F_1,F_2$ of $K_h(G(x,y),p)$ such that $F_2\subset F_1$,
$-\nabla_y G(x,y)w\in F_1-F_2$, and $z\in (F_1-F_2)^\ast$. Since $$-\nabla_y G(x,y)^\top z=\nabla_{yy}^2L(x,y)w,$$ we have that if $\nabla_{yy}^2 L(x,y)$ is positive definite and $\nabla_y G(x,y)$ has full rank then $-\nabla_y G(x,y)w=Mz$ for some positive definite matrix $M$, from which we obtain that $w=0$ and $z=0$ following the argument around \eqref{cond:ssoc}. This corresponds to the usual form of a second-order sufficient condition for nonlinear programs when $h$ is the indicator function of $(-\infty,0]^s$.

We make no use of Theorem \ref{thm:c1nor} in algorithmic developments; the properties established in the forthcoming Theorem \ref{thm:lipnor} allow us to formulate an optimality condition for the bilevel problem \eqref{prob:UL} without relying on the relative interior condition. However, Theorem \ref{thm:c1nor} is of independent interest as it establishes conditions under which one can, in principle, apply the proposed approach directly, without passing through regularization.

\subsection{Global Properties}

Up to this point, our results describe only the local behavior of $S_{\alpha,\beta}$ and $S$ around points satisfying the relative interior condition. Section 5 explains how algorithms may avoid points where this condition fails by employing a sampling scheme. However, these local properties alone are insufficient to guarantee algorithmic convergence; certain global properties of $S_{\alpha,\beta}$ enter as well. In particular, we prove that $S_{\alpha,\beta}$ is $PC^1$ \cite{scholtes2012introduction}, that is, $S_{\alpha,\beta}$ is continuous, and for every $\bar x \in \reals^n$ there exist finitely many $C^1$ mappings $\phi_1,\dots,\phi_\ell$ such that, for all $x$ sufficiently close to $\bar x$, $S_{\alpha,\beta}(x)\in \{\phi_1(x),\dots,\phi_\ell(x)\}$. 

\begin{theorem}\label{thm:lip}{\rm (global properties of regularized solution mapping).} If $\alpha,\beta>0$, then $S_{\alpha,\beta}$ is $PC^1$. 
\end{theorem}
\begin{proof}  Fix $\bar{x}\in \reals^n$, and consider $x^\nu \to \bar{x}$. Let $(\bar{y},\bar{p})=S_{\alpha,\beta}(\bar{x})$ and $(y^\nu,p^\nu)=S_{\alpha,\beta}(x^\nu)$. By continuity of $g$ and $e_{\alpha}h \circ G$, we have for any $y\in \reals^m$
\begin{equation*}
    g(x^\nu,y)+ e_{\alpha }h(G(x^\nu,y))+\frac{\beta}{2}|y|^2 \to g(\bar{x},y)+ e_{\alpha }h(G(\bar{x},y))+\frac{\beta}{2}|y|^2.
\end{equation*}
Since $e_{\alpha }h(G(\bar{x},\cdot))+\frac{\beta}{2}|\cdot|^2$ is convex and real valued, we have by \cite[Proposition 4.18]{royset2021optimization} that 
\begin{equation*}
    g(x^\nu,\cdot)+ e_{\alpha }h(G(x^\nu,\cdot))+\frac{\beta}{2}|\cdot|^2 \eto g(\bar{x},\cdot)+e_{\alpha }h(G(\bar{x},\cdot))+\frac{\beta}{2}|\cdot|^2.
\end{equation*}
Moreover, since $\nargmin g(\bar{x},\cdot)+e_{\alpha }h(G(\bar{x},\cdot))+\frac{\beta}{2}|\cdot|^2 = \{\bar{y}\}$, it follows from \cite[Proposition 9.40]{royset2021optimization} that $y^\nu\to \bar{y}$. Similarly, the Lagrangians (see \eqref{eq:lagrangian}) satisfy
\begin{equation*}
    -L_{\alpha,\beta}(x^\nu,y^\nu,\cdot) \eto -L_{\alpha,\beta}(\bar{x},\bar{y},\cdot)
\end{equation*}
because $G(x^\nu,y^\nu)\to G(\bar{x},\bar{y})$ and $h^\ast$ is lsc. From \cite[Proposition 5.36]{royset2021optimization}, one has $\nargmin -L_{\alpha,\beta}(x^\nu,y^\nu,\cdot)=\{p^\nu\}$ for all $\nu\in \nats$ and $\nargmin -L_{\alpha,\beta}(\bar{x},\bar{y},\cdot)=\{\bar{p}\}$. This implies that $p^\nu \to \bar{p}$ by \cite[Proposition 9.40]{royset2021optimization}. Therefore, $S_{\alpha,\beta}$ is continuous at $\bar{x}$.
\par

Let $x$ be sufficiently close to $\bar{x}$ with $(y,p)=S_{\alpha,\beta}(x)$. By continuity of $G$ and $S_{\alpha,\beta}$, the pieces that are inactive at $G(\bar{x},\bar{y})-\alpha \bar{p}$ will remain inactive at $G(x,y)-\alpha p$. By the same argument, together with \cite[Lemma 2.1]{hang2024role}, the pieces in $I^+(G(\bar{x},\bar{y})-\alpha \bar{p},\tau)$ and $J^+(G(\bar{x},\bar{y})-\alpha \bar{p},\sigma)$ will remain active at $G(x,y)-\alpha p$. Therefore, we have that $I(G(x,y)-\alpha p)\in \cI(G(\bar{x},\bar{y})-\alpha \bar{p},\tau)$ and $J(G(x,y)-\alpha p)\in \cJ(G(\bar{x},\bar{y})-\alpha \bar{p},\sigma)$ for $\cI(G(\bar{x},\bar{y})-\alpha \bar{p},\tau)$, $\cJ(G(\bar{x},\bar{y})-\alpha \bar{p},\sigma)$ defined as in \eqref{def:cicj} for a representation $(\sigma,\tau)$ of $\bar{p}$ as in \eqref{decomposition}. We seek to construct a finite collection of $C^1$ mappings, which will correspond to the finite elements in $\cI(G(\bar{x},\bar{y})-\alpha \bar{p},\tau)$ and $\cJ(G(\bar{x},\bar{y})-\alpha \bar{p},\sigma)$, and then prove that $S_{\alpha,\beta}$ must be equal to one of them at $x$, namely the one corresponding to $I(G(x,y)-\alpha p)$ and $J(G(x,y)-\alpha p)$.

Let $I^*\in \cI(G(\bar{x},\bar{y})-\alpha \bar{p},\tau)$ (if $\cI=\{\emptyset\}$, then take $I^*=\emptyset$), $J^*\in \cJ(G(\bar{x},\bar{y})-\alpha \bar{p},\sigma)$ (note that $\cJ(G(\bar{x},\bar{y})-\alpha \bar{p},\sigma)$ only contains nonempty sets of indices since the maximum must be attained at least one of the pieces of $h$), pick any $\widehat{j}\in J^*$ and define
\begin{equation*}
    \widehat{h}(z) = \langle a^{\widehat{j}},z\rangle-\alpha_{\widehat{j}}+\iota_{\dom \widehat{h}}(z),
\end{equation*}
where
\begin{equation*}
    \dom \widehat{h} = \big\{ z ~\big|~ \langle a^j,z\rangle-\alpha_j=\langle a^{\widehat{j}},z\rangle-\alpha_{\widehat{j}}\  \forall j\in J^*, \langle b^i,z\rangle-\beta_i=0\ \forall i\in I^*\big\} .
\end{equation*}
Note that $\widehat{h}$ is epi-polyhedral and that, by \cite[Proposition 2.39]{royset2021optimization},
\begin{equation}\label{for:hhat}
    \partial \widehat{h}(z) = \left\{a^{\widehat{j}}+\sum_{i\in I^*}s_ib^i+\sum_{j\in J^*}t_j(a^j-a^{\widehat{j}})~\bigg|~ s_i\in \reals,t_j\in \reals  \right\}.
\end{equation}
In particular, $\partial \widehat{h}(z)$ is an affine space and then $\ri(\partial \widehat{h}(z))=\partial \widehat{h}(z)$. \par
Consider the system
\begin{equation}\label{sys:hhat}
        0 = -\nabla_y g(\bar{x},y)-\nabla_y G(\bar{x},y)^\top p-\beta y, ~~~ 0 \in -G(\bar{x},y)+\alpha p +\partial \widehat{h}^*(p),
\end{equation}
which is similar to \eqref{opt:ll}; $\bar{y}$ and $\bar{p}$ solve \eqref{sys:hhat} too. Indeed, $\bar{p}\in \partial\widehat{h}(G(\bar{x},\bar{y})-\alpha\bar{p})$, which is true by taking $s_j=\sigma_j$ and $t_i=\tau_i$ in \eqref{for:hhat}. Similar to \eqref{for:psih}, we define
\begin{equation}\label{sys:active}
    \psi(y,p)=\begin{pmatrix}
        -\nabla_y g(\bar{x},y)-\nabla_yG(\bar{x},y)^\top p-\beta y \\ -G(\bar{x},y)+\alpha p
    \end{pmatrix} \quad \text{and}\quad \tilde{h}(y,p)=\widehat{h}^*(p),
\end{equation}
which is a system that fits the framework of Theorem \ref{thm:hang}. We have that $0\in \psi(\bar{y},\bar{p})+\partial\tilde{h}(\bar{y},\bar{p})$ and that, since $\partial \widehat{h}(G(\bar{x},\bar{y})-\alpha\bar{p})$ is an affine space, $-\psi(\bar{y},\bar{p})\in \ri(\partial\tilde{h}(\bar{y},\bar{p}))$. In order to apply Theorem \ref{thm:hang}, we only need to verify that the condition \eqref{cond:b} holds. From the argument leading to \eqref{eq:criticalcone}, we deduce that
\begin{equation*}
    K_{\tilde{h}}((\bar{y},\bar{p}),(-\nabla_y g(\bar{x},\bar{y})-\nabla_yG(\bar{x},\bar{y})^\top \bar{p}-\beta \bar{y} ,-G(\bar{x},\bar{y})+\alpha \bar{p})) = \reals^m\times K_{\widehat{h}^*}(\bar{p},G(\bar{x},\bar{y})-\alpha \bar{p}).
\end{equation*}
Then we can follow the proof of Theorem \ref{thm:c1}, replacing $K_{h^*}(\bar{p},G(\bar{x},\bar{y})-\alpha \bar{p})$ by $K_{\widehat{h}^*}(\bar{p},G(\bar{x},\bar{y})-\alpha \bar{p})$, and obtaining that the solution mapping of \eqref{sys:hhat} is a $C^1$ function of the left-hand sides around 0. The argument to prove that it is a $C^1$ mapping in $x$ around $\bar{x}$ is also the same as the one for Theorem \ref{thm:c1}. With this, the mapping 
\begin{equation}\label{def:phiij}
\varphi_{I^*,J^*}(x)=\{(y,p)\mid (y,p) \text{ satisfy }\eqref{sys:hhat} \text{ at }x\}
\end{equation}
is single-valued and $C^1$ around $\bar{x}$. By repeating this for every combination of $I^*$ and $J^*$, we obtain a finite family of $C^1$ mappings defined in a neighborhood around $\bar{x}$.\par
Let now $x'$ be in a neighborhood of $\bar{x}$ such that $I(G(x',y')-\alpha p')\in \cI(G(\bar{x},\bar{y})-\alpha \bar{p},\tau)$, $J(G(x',y')-\alpha p')\in \cJ(G(\bar{x},\bar{y})-\alpha \bar{p},\sigma)$ for $(y',p')=S_{\alpha,\beta}(x')$ and the mappings $\varphi_{I^*,J^*}$ as defined in \eqref{def:phiij} are $C^1$ for all $I^*\in \cI(G(\bar{x},\bar{y})-\alpha \bar{p},\tau)$ and $J^*\in \cJ(G(\bar{x},\bar{y})-\alpha \bar{p},\sigma)$. Take $I^*=I(G(x',y')-\alpha p')$ and $J^*=J(G(x',y')-\alpha p')$. Once again, $y'$ and $p'$ satisfy the first equality in \eqref{sys:hhat} when $x=x'$ since $(y',p')=S_{\alpha,\beta}(x')$ and, by considering a representation $(\sigma',\tau')$ of $p'$ and taking $s_j=\sigma_j'$, $t_i=\tau_i'$ in \eqref{for:hhat}, we obtain that $$p'\in \partial \widehat{h}(G(x',y')-\alpha p').$$ Therefore, $S_{\alpha,\beta}(x')=\varphi_{I^*,J^*}(x')$. Since this argument is valid for any $\bar{x}\in \reals^n$, the conclusion follows. 
\end{proof}
\par

The importance of Theorem \ref{thm:lip} comes from the fact that $PC^1$ mappings are locally Lipschitz, differentiable almost everywhere, and $C^1$ in an open dense subset of $\reals^n$ \cite{scholtes2012introduction}, which makes nonconvex nonsmooth optimization methods like bundle methods, Goldstein-type methods, and their extensions to constrained optimization suitable for finding stationary points of $x\mapsto f(x,Y_{\alpha,\beta}(x))$ or the corresponding constrained problem when $X$ is a strict subset of $\reals^n$.

Under the second-order condition \eqref{SSOC}, the actual lower-level solution mapping is single-valued in a neighborhood of a point, as we see next, even if the relative interior condition is not satisfied at this point. This result will be used to define stationarity of $x\mapsto f(x,Y(x))$, possibly subject to constraints. 

\begin{theorem}\label{thm:lipnor}{\rm (local single-valuedness of actual solution mapping).}
    Given $x\in \reals^n$, suppose that the basic qualification \eqref{CQ:llnor} holds at $y$ for $x$ and that $(y,p)\in S(x)$. Suppose also that the condition \eqref{SSOC}
    holds at $y$ for $x$ and $p$. Then $S$ is single-valued in a neighborhood of $x$.
\end{theorem}
\begin{proof} 
    Let $x'$ be sufficiently close to $x$. Under \eqref{CQ:llnor}, $(y',p')\in S(x')$ if and only if
    \begin{equation*}
    \begin{aligned}
     \begin{pmatrix}
        -\nabla_yg(x,y')-\nabla_yG(x,y')^\top p'+\nabla_yg(x',y')+\nabla_yG(x',y')^\top p'\\
        -G(x,y')+G(x',y')
    \end{pmatrix} \in &\begin{pmatrix}
        -\nabla_yg(x,y')-\nabla_y G(x,y')^\top p'\\
        -G(x,y') 
    \end{pmatrix}\\
    &+ \{0\}\times \partial h^\ast(p').
    \end{aligned}
\end{equation*}
    Therefore, in order for $S$ to be single-valued around $x$ it is sufficient for the solution mapping of \eqref{opt:llnor} to be single-valued around 0 as a function of the left-hand side. By \cite[Theorem 5.2]{hang2024role}, this is equivalent to the condition \eqref{SSOC} being satisfied at $y$ for $x$ and $p$, concluding the proof.
\end{proof}

The section ends by discussing an inequality-constrained lower-level problem.

\vspace{0.3cm}

\begin{example}\label{ex:example1}
Let $h(z)=\sum_{i=1}^{s}\iota_{(-\infty,0]}(z_i)$. Then, $a^1=(0,0,0,...,0)$, $\alpha_1=0$ and $b^i=e_{i}$, $\beta_i=0$ for all $i=1, \dots, s$ in \eqref{def:h}, where $e_{i}$ is the vector with 1 in the i$^{th}$ coordinate and 0 elsewhere. Thus, 
\begin{equation*}
p\in \partial h(z) ~~~\Longleftrightarrow~~~    p = \sum_{i\in I(z)}\tau_ib^i.
\end{equation*}
From this, we obtain that $p\in \ri (\partial h(z))$ if and only if $\tau_i>0$ for all $i\in I(z)$, that is, if and only if the multiplier associated with every active constraint is strictly positive. Therefore, the relative interior condition reduces to strict complementary slackness when the lower-level problem is a nonlinear program with inequality constraints. Moreover, it is well known that \eqref{CQ:llnor} amounts to the Mangasarian-Fromovitz constraint qualification in this case (cf. \cite[Example 4.49]{royset2021optimization}). The condition that the matrix $\nabla_yG(x,y)$ has full row rank and the matrix $\nabla_{yy}^2g(x,y)+\nabla_{yy}^2G(x,y)p$ is positive definite, which implies \eqref{SSOC}, also implies the usual second-order sufficient condition for nonlinear programs with inequality constraints (e.g., \cite[Condition (SSOC)]{dempe2000bundle}).

Note also that for $p\in \ri(\partial h(z))$, $w\in K_h(z,p)$ if and only if $\langle b^i,w\rangle=0$ for all $i\in I(z)$, which implies that $w_{i}=0$ whenever $i\in I(z)$. From \eqref{eq:polar}, $w\in K=\reals^m\times K_{h^*}(p,G(x,y)-\alpha p)$ if and only if $w_{m+i}=0$ whenever $i\notin I(z)$. The columns of the matrix
\begin{equation*}
    B=\begin{pmatrix}
        \bbI_m & 0\\
        0 & M_{G(x,y)-\alpha p}
    \end{pmatrix}
\end{equation*}
form a basis for $K$, where the matrix $M_{z}$ is a truncated identity matrix where the columns corresponding to inactive pieces at $z$ are removed. If $A$ is defined as in Theorem \ref{thm:c1} and $z=G(x,y)-\alpha p$, then one has that 
\begin{equation}\label{eq:jacnlp}
    B^\top AB = \begin{pmatrix}
        \nabla_{yy}^2g(x,y)+\nabla_{yy}^2G(x,y)p+\beta \bbI_m & (\nabla_yG_i(x,y))_{i\in I(z)}^\top \\
        (\nabla_yG_i(x,y))_{i\in I(z)} & \alpha \bbI_{|I(z)|}
    \end{pmatrix},
\end{equation}
where $|I(z)|$ indicates the cardinality of $I(z)$. That is, $B^\top AB$ is simply the matrix $A$ obtained when considering only the active constraints at $G(x,y)-\alpha p$. (Note that $\nabla_{yy}^2G(x,y)p=(\nabla_{yy}^2G_i(x,y)p_i)_{i\in I(z)}$ since $p_i=0$ when $i\notin I(z)$.) This formula for the Jacobian of the solution mapping appears in \cite{fiacco1990nonlinear}, where it is obtained using classical tools from nonlinear programming; see \cite{dempe2000bundle} for an application in the context of bilevel optimization. Theorem \ref{thm:c1} is a generalization of this result and, thanks to the introduction of the regularization given by the Moreau envelope, does not rely on assumptions like the Mangasarian-Fromovitz constraint qualification. A similar situation occurs even when the relative interior condition is not satisfied. In this case, Theorem \ref{thm:lip} reduces to \cite[Theorem 2.4]{dempe2000bundle}, which relies on \cite[Theorem 2]{ralph1995directional}, from which the proof of Theorem \ref{thm:lip} draws inspiration.
\end{example}

\section{Convergence of the Regularized Lower-Level Problems}

This section shows that the regularized lower-level problem in \eqref{prob:LLr} furnishes a good approximation of the actual lower-level problem from \eqref{prob:LL}. In particular, we study the convergence of $S_{\alpha^\nu,\beta^\nu}(x^\nu)$ and $\nabla S_{\alpha^\nu,\beta^\nu}(x^\nu)$ when $x^\nu\to \bar{x}$ and $\alpha^\nu,\beta^\nu \searrow0$.

The first proposition establishes convergence of a sequence of primal-dual solutions. \textcolor{black}{We start with an intermediate lemma.}

\textcolor{black}{\begin{lemma}\label{lemma:epiconvergence}
    For $\bar{x}\in \reals^n$ and $\bar{y}\in Y(\bar{x})$, suppose that the basic qualification \eqref{CQ:llnor} holds at $\bar{y}$ for $\bar{x}$. For $x^\nu\to \bar{x}$ and $\alpha^{\nu},\beta^{\nu}\searrow 0$, define the functions $\varphi,\varphi^\nu:\reals^m\to \Reals$ as
    \begin{equation}\label{eq:varphinu}
        \varphi(y) = g(\bar{x},y)+ h(G(\bar{x},y)), \quad\quad\varphi^\nu(y) = e_{\alpha^\nu}h(G(x^\nu,y))+g(x^\nu,y)+\frac{\beta^\nu}{2}|y|^2.
    \end{equation}
    Then, $\varphi^\nu\eto \varphi$.
\end{lemma}}
\begin{proof}
    We use the characterization of epi-convergence in \cite[Theorem 4.15]{royset2021optimization}. Let $\hat{y}\in \reals^m$. We need to show that for all $y^\nu\to \hat{y}$, $\liminf \varphi^\nu(y^\nu)\geq \varphi(\hat{y})$, and that there exists $y^\nu\to \hat{y}$ such that $\limsup \varphi^\nu(y^\nu)\leq \varphi(\hat{y})$.
    
    \textcolor{black}{For the first part, let $y^\nu\to \hat{y}$. Continuity of $g$ ensures that $g(x^\nu,y^\nu)+\frac{\beta^\nu}{2}|y^\nu|^2\to g(\bar{x},\hat{y})\in \reals$. Moreover, since $G(x^\nu,y^\nu)\to G(\bar{x},\hat{y})$ and $e_{\alpha^\nu} h\eto h$ \cite[Theorem 1.25 and Proposition 7.4(d)]{VaAn}, we obtain that
    \begin{equation*}
        \liminf e_{\alpha^\nu}h(G(x^\nu,y^\nu)) \geq h(G(\bar{x},\hat{y})).
    \end{equation*}
    Then $\liminf \varphi^\nu(y^\nu) \geq \varphi(\hat{y})$.} 
    
    \textcolor{black}{For the second part, suppose $\hat{y}\in \dom \varphi$. (If $\hat{y}\notin \dom \varphi$, then $\limsup \varphi^\nu (y^\nu)\leq \varphi(\hat{y})$ holds trivially.)
    Note that $\partial^\infty h(G(\bar{x},\bar{y}))=\pos \{b^i\}_{i\in I(G(\bar{x},\bar{y}))}$. The basic qualification \eqref{CQ:llnor} at $\bar{y}$ for $\bar{x}$ implies that there exists no $u\in \pos \{b^i\}_{i\in I(G(\bar{x},\bar{y}))}$, $u\neq 0$, such that $\nabla_y G(\bar{x},\bar{y})^\top u=0$. That is, the system
    \begin{equation*}
        \nabla_y G(\bar{x},\bar{y})^\top M \lambda=0, \quad \lambda  \geq 0,
    \end{equation*}
    has no nonzero solution, where $M=(b_i)_{i\in I(G(\bar{x},\bar{y}))}$. By Gordan's theorem, there exists $d\in \reals^m$ such that $M^\top \nabla_y G(\bar{x},\bar{y})d<0$. This means that $\langle b^i,\nabla_yG(\bar{x},\bar{y}) d\rangle<0$ for all active constraints $i\in I(G(\bar{x},\bar{y}))$. Consequently, there is $\delta>0$ such that 
    \begin{equation*}
        \langle b^i,G(\bar{x},y_\delta)\rangle< \beta_i ~~~\forall i\in I, 
    \end{equation*}
where $y_\delta = \bar{y}+\delta d$. Let $t_\nu\searrow 0$ be such that $(x^\nu-\bar{x})/t_\nu \to 0$ and construct $y^\nu=(1-t_\nu)\hat{y}+t_\nu y_\delta$ for all $\nu\in \nats$. Clearly, $y^\nu\to \hat y$. From the convexity of $G(\bar x, \cdot)$, together with the fact that $b^i\geq0$ for all $i\in I$, one has for some positive $c$: 
    \begin{equation*}
    \begin{aligned}
        \langle b^i,G(\bar{x},y^\nu)\rangle&\leq (1-t_\nu)\langle b^i,G(\bar{x},\hat{y})\rangle +t_\nu\langle b^i,G(\bar{x},y_\delta)\rangle\\
        &\leq (1-t_\nu) \beta_i + t_\nu (\beta_i-c) < \beta_i-t_\nu c
    \end{aligned}
    \end{equation*}
for all $i\in I$, where the second inequality holds because $\hat{y}\in \dom h$ and $y_\delta$ satisfies every inequality strictly. Since $G$ is locally Lipschitz around $(\bar{x},\hat{y})$, say with constant $L_{G,\bar{x},\hat{y}}>0$, by being $C^2$, and $x^\nu-\bar{x}=o(t_\nu)$, there exists $\nu_0\in \nats$ such that for all $\nu\geq \nu_0$ one has $|x^\nu-\bar{x}|\leq c t_\nu/L_{G,\bar{x},\hat{y}}$, from which we obtain
    \begin{equation*}
        \langle b^i,G(x^\nu,y^\nu)\rangle \leq \langle b^i,G(\bar{x},y^\nu)\rangle  + L_{G,\bar{x},\hat{y}}|x^\nu-\bar{x}|\leq \beta_i. 
    \end{equation*}
    Thus, $G(x^\nu,y^\nu)\in \dom h$ for all sufficiently large $\nu$. With this, since $h$ is epi-polyhedral and then continuous on its domain, we obtain $h(G(x^\nu,y^\nu))\to h(G(\bar{x},\hat{y}))$ and conclude that $\limsup \varphi^\nu(y^\nu)\leq \varphi(\hat{y})$.}
\end{proof}

\begin{proposition}\label{prop:convS}
    For $\bar{x}\in \reals^n$ and $(\bar{y},\bar{p})\in S(\bar{x})$, suppose that \textcolor{black}{the basic qualification \eqref{CQ:llnor} holds at $\bar{y}$ for $\bar{x}$ and that} \eqref{SSOC} holds at $\bar{y}$ for $\bar{x}$ and $\bar{p}$. If $x^\nu\to \bar{x}$ and $\alpha^{\nu},\beta^{\nu}\searrow 0$, then $S_{\alpha^\nu,\beta^\nu}(x^\nu)\to (\bar{y},\bar{p})$.    
\end{proposition}
\begin{proof} 
    Consider the functions $\varphi$ and $\varphi^\nu$ defined in Lemma \ref{lemma:epiconvergence}.  \textcolor{black}{Convexity of $\varphi$ and \eqref{SSOC} ensure that $\nargmin \varphi=\{\bar{y}\}$. Convexity and the unique minimizer imply that $\varphi$ is level-bounded. With this,
    }
    since $\varphi^\nu$ and $\varphi$ are convex lsc functions and, by Lemma \ref{lemma:epiconvergence}, $\varphi^\nu\eto \varphi$, we have from \cite[Exercise 7.32(c) and Theorem 7.33]{VaAn} that $y^\nu \to \bar{y}$, where $\{y^\nu\}=Y_{\alpha^\nu,\beta^\nu}(x^\nu)$. To prove $p^\nu\to \bar{p}$, we have under \eqref{SSOC} that
    \begin{equation*}
        \{p^\nu\} = \nargmin -L_{\alpha^\nu,\beta^\nu}(x^\nu,y^\nu,\cdot) \quad \text{and} \quad \{\bar{p}\}=\nargmin -L(\bar{x},\bar{y},\cdot);
    \end{equation*}
    see \cite[Proposition 5.36]{royset2021optimization} and the notation from \eqref{eq:lagrangian} and \eqref{eqn:LagrangianActual}. Using the fact that $G(x^\nu,y^\nu)\to G(\bar{x},\bar{y})$, it follows from \cite[Proposition 4.19(b)]{royset2021optimization} that $-L_{\alpha^\nu,\beta^\nu}(x^\nu,y^\nu,\cdot)\eto -L(\bar{x},\bar{y},\cdot)$. Since these functions are convex and $-L(\bar{x},\bar{y},\cdot)$ is level bounded by being convex and having a unique minimizer, we can use \cite[Theorem 7.33]{VaAn} to conclude that $p^\nu\to \bar{p}$.
\end{proof}

We start the analysis of the convergence of the Jacobians by considering the case in which the relative interior condition is satisfied at the limit point, and then the Jacobian there is also well defined.

\begin{proposition}\label{prop:conv}
    Given $x^\nu \to \bar{x}$ and $\alpha^\nu,\beta^\nu\searrow 0$, suppose that there exists $(\bar{y},\bar{p})\in S(\bar{x})$ such that \eqref{CQ:llnor} holds at $\bar{y}$ for $\bar{x}$, \eqref{SSOC} holds at $\bar{y}$ for $\bar{x}$ and $\bar{p}$, and $G(\bar{x},\bar{y})\in \ri (\partial h^\ast(\bar{p}))$. Then $S_{\alpha^\nu,\beta^\nu}(x^\nu)\to (\bar{y},\bar{p})$, $S_{\alpha^\nu,\beta^\nu}$ is $C^1$ around $x^\nu$ for all sufficiently large $\nu$, and $\nabla S_{\alpha^\nu,\beta^\nu}(x^\nu)\to \nabla S(\bar{x})$ along the subsequence of points at which $S_{\alpha^\nu,\beta^\nu}$ is $C^1$. 
\end{proposition}
\begin{proof}  The limit $S_{\alpha^\nu,\beta^\nu}(x^\nu)\to (\bar{y},\bar{p})$ follows from Proposition \ref{prop:convS}. Let $(y^\nu,p^\nu)=S_{\alpha^\nu,\beta^\nu}(x^\nu)$. In order to study the convergence of $\nabla S_{\alpha^\nu,\beta^\nu}(x^\nu)$, we first prove that this expression is well defined for sufficiently large $\nu$. Since $G(x^\nu,y^\nu)-\alpha^\nu p^\nu\to G(\bar{x},\bar{y})$, there exists $\nu_0\in \nats$ such that for all $\nu\geq \nu_0$ the \textcolor{black}{indices that are inactive at $G(\bar{x},\bar{y})$ are also inactive at $G(x^\nu,y^\nu)-\alpha^\nu p^\nu$.} Moreover, since $p^\nu\to \bar{p}$, there exist \textcolor{black}{from \cite[Lemma 2.1]{hang2024role}} representations $(\sigma^\nu,\tau^\nu)$ and $(\bar{\sigma},\bar{\tau})$ and a natural number which we also call $\nu_0$ such that $I^+(G(x^\nu,y^\nu)-\alpha^\nu p^\nu,\tau^\nu)=I^+(G(\bar{x},\bar{y}),\bar{\tau})$ and $J^+(G(x^\nu,y^\nu)-\alpha^\nu p^\nu,\sigma^\nu)=J^+(G(\bar{x},\bar{y}),\bar{\sigma})$ for all $\nu\geq \nu_0$. \textcolor{black}{Since $G(\bar{x},\bar{y})\in \ri (\partial h^*(\bar{p}))$ and thus every active index is in $I^+(G(\bar{x},\bar{y}),\bar{\tau})$ or $J^+(G(\bar{x},\bar{y}),\bar{\sigma})$}, this implies that the relative interior condition $G(x^\nu,y^\nu)-\alpha^\nu p^\nu\in \ri(\partial h^*(p^\nu))$ holds and $S_{\alpha^\nu,\beta^\nu}$ is $C^1$ at $x^\nu$ for all $\nu\geq \nu_0$.

The matrix $B$ in Theorems \ref{thm:c1} and \ref{thm:c1nor} depends only on the sets of active indices and the sets of indices with positive coefficients, by \eqref{sys:kh}. Therefore,
\begin{equation*}
    \nabla S_{\alpha^\nu,\beta^\nu}(x^\nu) =-B(B^\top A^\nu B)^{-1}B^\top \begin{pmatrix}
        \nabla_x\big(\nabla_y g(x^\nu,y^\nu)+\nabla_y G(x^\nu,y^\nu)^\top p^\nu\big)\\
        \nabla_xG(x^\nu,y^\nu)
    \end{pmatrix},
\end{equation*}
where
\begin{equation*}
    A^\nu=\begin{pmatrix}
            \nabla_{yy}^2g(x^\nu,y^\nu)+\nabla_{yy}^2G(x^\nu,y^\nu)p^\nu+\beta^\nu \bbI_m ~~&~~ \nabla_y G(x^\nu,y^\nu)^\top \\
                 \nabla_y G(x^\nu,y^\nu) ~~&~~ -\alpha^\nu \bbI_s
    \end{pmatrix}
\end{equation*}
and $B$ is a matrix whose columns form a basis for $\reals^m\times K_{h^*}(\bar{p},G(\bar{x},\bar{y}))$. Since 
\begin{equation*}
    A^\nu\to A=
        \begin{pmatrix}
            \nabla_{yy}^2g(\bar{x},\bar{y})+\nabla_{yy}^2G(\bar{x},\bar{y})\bar{p} ~~&~~ \nabla_y G(\bar{x},\bar{y})^\top \\
                 \nabla_y G(\bar{x},\bar{y}) ~~&~~ 0
        \end{pmatrix}
\end{equation*}
and 
\begin{equation*}
    \begin{pmatrix}
        \nabla_x\big(\nabla_y g(x^\nu,y^\nu)+\nabla_y G(x^\nu,y^\nu)^\top p^\nu\big)\\
        \nabla_xG(x^\nu,y^\nu)
    \end{pmatrix} \to \begin{pmatrix}
        \nabla_x\big(\nabla_y g(\bar{x},\bar{y})+\nabla_y G(\bar{x},\bar{y})^\top \bar{p}\big)\\
        \nabla_xG(\bar{x},\bar{y})
    \end{pmatrix},
\end{equation*}
we conclude that $\nabla S_{\alpha^\nu,\beta^\nu}(x^\nu)\to \nabla S(\bar{x})$. 
\end{proof}
Proposition \ref{prop:conv} provides a result under which the chain rule applied to the composition of $f$ and $Y_{\alpha,\beta}$ yields quantities that approximate those obtained by applying the chain rule to $f(\cdot,Y(\cdot))$, relying only on assumptions on the limit point of the approximating sequence. The next theorem addresses the case in which the relative interior condition is not satisfied at the limit point but holds instead at each point of the sequence, moving closer to the situation found in the forthcoming algorithms. 

\begin{theorem}\label{prop:conv2}{\rm (convergence of Jacobian of regularized solution mappings).}
    For $\bar{x}\in \reals^n$ and $(\bar{y},\bar{p})\in S(\bar{x})$, suppose that \textcolor{black}{\eqref{CQ:llnor} holds at $\bar{y}$ for $\bar{x}$ and that} \eqref{SSOC} holds at $\bar{y}$ for $\bar{x}$ and $\bar{p}$. If $\alpha^\nu,\beta^\nu\searrow 0$, $x^\nu \to \bar{x}$, $(y^\nu,p^\nu)=S_{\alpha^\nu,\beta^\nu}(x^\nu)$, and the relative interior condition $G(x^\nu,y^\nu)-\alpha^\nu p^\nu\in \ri (\partial h^*(p^\nu))$ is satisfied for all $\nu\in \nats$, then $S_{\alpha^\nu,\beta^\nu}$ is $C^1$ around $x^\nu$, $S_{\alpha^\nu,\beta^\nu}(x^\nu)\to (\bar{y},\bar{p})$ and there exists a subsequence $N\subset \nats$ such that 
    \[
    \nabla S_{\alpha^\nu,\beta^\nu}(x^\nu) \Nto M,
    \]
    where $M$ satisfies $M^\top w\in \partial \langle S(\cdot), w\rangle(\bar x) $ for all $w\in \reals^{m+s}$. 
\end{theorem}
\begin{proof} 
    The convergence of $y^\nu\to \bar{y}$ and $p^\nu\to \bar{p}$ follows by Proposition \ref{prop:convS}. We only need to prove that there exists some $M$ such that $\nabla S_{\alpha^\nu,\beta^\nu}(x^\nu)\Nto M$, since that implies $M^\top w\in \partial\langle S(\cdot),w\rangle(\bar x) $ for every $w\in \reals^{m+s}$ by the chain rule and the definition of subgradients.
    
    Suppose that the active indices are the same for every element of the sequence (if this is not the case, we can take a subsequence that satisfies this property). Then there exists a matrix $B$ such that for every $\nu\in \nats$,
    \begin{equation*}
    \nabla S_{\alpha^\nu,\beta^\nu}(x^\nu) =-B(B^\top A^\nu B)^{-1}B^\top \begin{pmatrix}
        \nabla_x\big(\nabla_y g(x^\nu,y^\nu)+\nabla_y G(x^\nu,y^\nu)^\top p^\nu\big)\\
        \nabla_xG(x^\nu,y^\nu)
    \end{pmatrix},
\end{equation*}
    with 
    \begin{equation*}
        A^\nu = \begin{pmatrix}
            \nabla_{yy}^2g(x^\nu,y^\nu)+\nabla_{yy}^2G(x^\nu,y^\nu)p^\nu+\beta^\nu \bbI_m ~~ & ~~\nabla_y G(x^\nu,y^\nu)^\top \\
                 \nabla_y G(x^\nu,y^\nu) & -\alpha^\nu \bbI_s
        \end{pmatrix}.
    \end{equation*}
The sequence of matrices $A^\nu$ converges component-wise to 
    \begin{equation*}
        A=\begin{pmatrix}
            \nabla_{yy}^2g(\bar{x},\bar{y})+\nabla_{yy}^2G(\bar{x},\bar{y})\bar{p}~~ &~~ \nabla_y G(\bar{x},\bar{y})^\top \\
                 \nabla_y G(\bar{x},\bar{y}) & 0
        \end{pmatrix}.
    \end{equation*}
    If $B^\top AB$ is invertible, then it follows that the sequence $\nabla S_{\alpha^\nu,\beta^\nu}(x^\nu)$ converges to
    \begin{equation*}
        M=-B(B^\top A B)^{-1}B^\top \begin{pmatrix}
        \nabla_x\big(\nabla_y g(\bar{x},\bar{y})+\nabla_y G(\bar{x},\bar{y})^\top \bar{p}\big)\\
        \nabla_xG(\bar{x},\bar{y})\end{pmatrix}.
    \end{equation*}
    To establish this invertibility, suppose, by contradiction, that there exists $u\neq 0$ such that $B^\top ABu=0.$ Since $B$ is a basis of $K=\reals^m\times K_{h^\ast}(p^\nu,G(x^\nu,y^\nu)-\alpha^\nu p^\nu)$, where $\nu$ is any natural number, there are $w\in \reals^m$ and $z\in K_{h^\ast}(p^\nu,G(x^\nu,y^\nu)-\alpha^\nu p^\nu)$ such that $(w,z)=Bu$. Since $\ker(B^\top)=(\spn(B))^\perp$, we conclude that $B^\top ABu=0$ if and only if
    \[
    \big(\nabla_{yy}^2g(\bar{x},\bar{y})+\nabla_{yy}^2G(\bar{x},\bar{y})\bar{p}\big) w+\nabla_y G(\bar{x},\bar{y})^\top z=0, ~~~
    -\nabla_y G(\bar{x},\bar{y}) w\in  (K_{h^\ast}(p^\nu,G(x^\nu,y^\nu)-\alpha^\nu p^\nu))^\perp.
    \]
    From \cite[Proposition 3.1]{hang2024role}, we have that for sufficiently large $\nu\in \nats$, there are faces $F_1$ and $F_2$ of $K_{h}(G(\bar{x}
    ,\bar{y}),\bar{p})$ with $F_2\subset F_1$ satisfying $F_1-F_2 = K_h(G(x^\nu,y^\nu)-\alpha^\nu p^\nu,p^\nu)$. Since the relative interior condition is satisfied at $x^\nu$ for all $\nu\in \nats$, we have that
    \begin{equation*}
        K_h(G(x^\nu,y^\nu)-\alpha^\nu p^\nu,p^\nu) =  \big(K_{h^\ast}(p^\nu,G(x^\nu,y^\nu)-\alpha^\nu p^\nu)\big)^\perp.
    \end{equation*}
    Therefore, $w$ and $z$ satisfy $-\nabla_y G(\bar{x},\bar{y}) w\in F_1-F_2$ and $z\in (F_1-F_2)^\perp$, which, by \cite[Theorem 3.3(b)]{hang2024role}, is equivalent to 
    \begin{equation*}
        (z,-\nabla_yG(\bar{x},\bar{y})w)\in  N_{\gph \partial h}(G(\bar{x},\bar{y}),\bar{p}).
    \end{equation*}
    Then, since \eqref{SSOC} is satisfied at $\bar{y}$ for $\bar{x}$ and $\bar{p}$, $w=0$ and $z=0$. Lastly, since $\ker B=\{0\}$, we conclude that $u=0$ and therefore $B^\top AB$ must be invertible.
\end{proof}

Theorem \ref{prop:conv2} enables the construction of a conceptual algorithm to solve the bilevel problem \eqref{prob:UL}. Consider $\bar{x}\in \reals^n$ such that $S$ (and therefore $Y$) is single-valued and locally Lipschitz around $\bar{x}$ (for example, because \eqref{SSOC} holds at the minimizer of \eqref{prob:LL} for $\bar{x}$). Locally near $\bar x$, the bilevel problem is then equivalent to minimizing the hyper-objective $f(x,Y(x))$ subject to $x\in X$. Thus, $\bar{x}$ is a stationary point of the bilevel problem if
\begin{equation*}
    0\in \partial (f(\cdot,Y(\cdot)))(\bar{x}) +N_{X}(\bar{x}).
\end{equation*}

We present now a conceptual algorithm for finding such stationary points available under the assumption that the relative interior condition is satisfied at $x^\nu$ at every iteration, which holds if the assumptions for either Proposition \ref{prop:conv} or Theorem \ref{prop:conv2} are satisfied. At each iteration $\nu\in \nats$, set $u^\nu= \nabla f(x^\nu,Y_{\alpha^\nu,\beta^\nu}(x^\nu))$ and find $x^\nu\in X$ and $w^\nu\in N_{X}(x^\nu)$ such that 
\begin{equation*}
    \big|\begin{pmatrix}
        \bbI_n & \nabla Y_{\alpha^\nu,\beta^\nu}(x^\nu)
    \end{pmatrix} u^\nu +w^\nu \big|\leq \varepsilon^\nu.
\end{equation*}
If $x^\nu\to \bar{x}$, \eqref{SSOC} holds at $\bar{y}$ for $\bar{x}$ and $\bar{p}$ with $\{(\bar{y},\bar{p})\}= S(\bar{x})$, and $\alpha^\nu,\beta^\nu,\varepsilon_\nu\searrow0$, then by Theorem \ref{prop:conv2} there exists a subsequence $N\subset \nats$ and a matrix $M$ such that $\nabla Y_{\alpha^\nu,\beta^\nu}(x^\nu)\Nto M$, $u^\nu\Nto \bar{u}= \nabla f(\bar{x},Y(\bar{x}))$, $w^\nu \Nto \bar{w}\in N_{X}(\bar{x})$, and 
\begin{equation*}
    \partial f(\cdot,Y(\cdot))(\bar{x})+N_X(\bar{x})\ni \begin{pmatrix}
        \bbI_n & M
    \end{pmatrix} \bar{u} + \bar{w}=0,
\end{equation*}
which implies that $\bar{x}$ is stationary. 

However, finding such sequences $\{x^\nu\}$ and $\{w^\nu\}$ is in general computationally challenging. Therefore, in order to implement this algorithm, the optimality condition will be relaxed and, leveraging the fact that $S$ is locally Lipschitz around points that satisfy \eqref{SSOC}, we will look instead for Clarke stationary points of the hyper-objective in the unconstrained case, and an appropriate constrained notion of stationarity when $X\subset \reals^n$. Moreover, descent directions will not be drawn from the subdifferential of the regularized hyper-objective, but instead from its Goldstein subdifferential, for which existing algorithms are available. The next section discusses implementation details of such algorithms for both $X = \reals^n$ and $X\subset \reals^n$.

\section{Bilevel Optimization Algorithms using Gradient Sampling}

The overall optimization scheme will be carried out using gradient sampling for locally Lipschitz functions. The idea of this class of algorithms is that, if a function is differentiable almost everywhere, we can sample around every iteration point to compute gradients in order to approximate the Goldstein subdifferential without ever computing subgradients at points where the function is not differentiable; see, e.g., \cite{burke2020gradient}. Unlike a direct application of gradient sampling to a fixed locally Lipschitz function, the hyper-objective evolves as the regularization parameters decrease. We show that the resulting iterates become approximately stationary for these changing regularized hyper-objectives and establish conditions under which their cluster points are stationary for the original bilevel problem. We develop both unconstrained and constrained variants of this scheme.

\subsection{Preliminaries}

This subsection addresses two technical matters before we state the algorithms. First, gradient sampling methods require that the objective function is $C^1$ almost everywhere. We know from \cite[Proposition 4.1.5]{scholtes2012introduction} that every $PC^1$ mapping from $\reals^n$ is differentiable almost everywhere and $C^1$ in an open dense subset of $\reals^n$. However, as shown in \cite{burke2020gradient}, it is possible to construct examples for which the objective function satisfies such properties but gradient sampling fails regardless. 

Second, Theorem \ref{thm:c1} establishes that if the relative interior condition holds at $x$, then $S_{\alpha,\beta}$ is $C^1$ in a neighborhood of $x$. However, the converse fails, as can be seen in the following example, and therefore it is necessary to ensure that the set of points at which the formula of Theorem \ref{thm:c1} holds has full measure.

\vspace{0.3cm}

\begin{example}\label{ex:example3} 
    For a fixed $x\in \reals$, consider the problem of minimizing $xy$ over all $y\in \reals$ such that $x^2\geq 0$. We can take $g(x,y)=xy$, $G(x,y)=(0,-x^2)$, and $h(z_1,z_2)=z_1+\iota_{(-\infty,0]}(z_2)$ so that this problem has the structure of \eqref{prob:LL}. For $\alpha,\beta>0$, the regularized lower-level problem amounts to minimizing $xy+\tfrac{1}{2}\beta y^2$ since the constraint $x^2\geq 0$ is always satisfied. The primal-dual solution mapping of this problem is $S_{\alpha,\beta}(x) = (-x/\beta,1,0)$, which is $C^1$ for every $x\in \reals$. However, at $\bar{x}=0$ the constraint $x^2\geq 0$ is active, and its corresponding dual variable is $\bar{p}_2=0$. Note that, for $g_1(\bar{x},\bar{y})=0$, where $\bar{y}=Y_{\alpha,\beta}(\bar{x})$, we have that this is the only piece that attains the maximum and then the associated dual variable is $\bar{p}_1=1$. Therefore, $\bar{p} \notin \ri(\partial h (G(\bar{x},\bar{y})-\alpha\bar{p}) = \{ p \, |\, p_1=1, p_2\geq 0\}$ and the relative interior condition fails for a $C^1$ solution mapping.
\end{example}

\vspace{0.3cm}

These issues will be addressed in Theorem \ref{thm:alg}, which is the counterpart of Theorem \ref{prop:conv2} when the relative interior assumption is replaced by the assumption that the points of the sequence are chosen randomly, under the mild additional assumption that $g$ and $G$ are analytic, which will ensure that we have smoothness and a computable Jacobian formula almost everywhere.

\begin{lemma}\label{lemma:tame}
    If $g$ and $G$ are analytic, then $S_{\alpha,\beta}$ is tame and, in particular, $C^1$ almost everywhere.
\end{lemma}
\begin{proof} 
    From \cite{bolte2009tame}, a function is tame if the intersection of its graph with any closed ball is definable in an o-minimal structure. In particular, we will prove that for any $r>0$, the set $\gph S_{\alpha,\beta}\cap [-r,r]^{n+m+s}$ is globally subanalytic. For that, we use the definition of subanalytic and semianalytic sets provided in \cite{bierstone1988semianalytic}. \par
    \textcolor{black}{Since $h^*$ is epi-polyhedral, $\gph \partial h^*$ is the union of a finite collection of polyhedral sets and is therefore semianalytic. From the equivalence between \eqref{opt:ge} and \eqref{opt:ll}, we obtain
    \begin{equation*}
        \gph S_{\alpha,\beta} = \big\{ (x,y,p) ~\big|~ \nabla_y g(x,y) +\nabla_yG(x,y)^\top p+\beta y=0,~~ (p,G(x,y)-\alpha p)
   \in\gph\partial h^*\big\}.
    \end{equation*}
    Since the preimage of a semianalytic set under an analytic mapping is semianalytic, we conclude that $\gph S_{\alpha,\beta}$ is semianalytic. Then, for $r>0$, the set $\gph S_{\alpha,\beta}\cap [-r,r]^{n+m+s}$ is semianalytic and bounded and in particular is globally subanalytic.
    } 
    With this, $\gph S_{\alpha,\beta}\cap [-r,r]^{n+m+s}$ belongs to an o-minimal structure, and so $S_{\alpha,\beta}$ is tame.
\end{proof}

\begin{theorem}\label{thm:alg}{\rm (sampling yields smooth points and convergence of Jacobian).}
    Suppose that $g$ and $G$ are analytic. For $\bar{x}\in \reals^n$ and $(\bar{y},\bar{p})  \in S(\bar{x})$, suppose that \textcolor{black}{\eqref{CQ:llnor} holds at $\bar{y}$ for $\bar{x}$ and that} \eqref{SSOC} holds at $\bar y$ for $\bar x$ and $\bar p$, one has $\alpha^\nu,\beta^\nu\searrow 0$, and $x^\nu\to \bar{x}$. If $\tilde{x}^\nu$ are sampled independently from the uniform distribution over $\ball (x^\nu,\epsilon^\nu)$ and $\epsilon^\nu\searrow 0$, then with probability 1, $S_{\alpha^\nu,\beta^\nu}$ is $C^1$ at $\tilde{x}^\nu$ and there exist a subsequence $N\subset\nats$ and a matrix $M$ such that $\nabla S_{\alpha^\nu,\beta^\nu}(\tilde{x}^\nu)\Nto M$ and 
    \[
    M^\top w\in\partial  \langle S(\cdot), w\rangle(\bar x) ~~\forall w\in \reals^{m+s}.
    \]
    Moreover, denoting $(\tilde{y}^\nu,\tilde{p}^\nu)=S_{\alpha^\nu,\beta^\nu}(\tilde{x}^\nu)$, $\nabla S_{\alpha^\nu,\beta^\nu}(\tilde{x}^\nu)$ can be computed using \eqref{for:gradSr}, replacing $B$ by a matrix whose column form a basis of $K^\perp$, with $K$ being the solution set of 
    \begin{equation}\label{sys:khri}
\begin{aligned}
    \langle a^i-a^j,w\rangle& =0 \quad\quad i,j\in J(G(\tilde{x}^\nu,\tilde{y}^\nu)-\alpha \tilde{p}^\nu), \\
    \langle b^i,w\rangle&=0 \quad\quad  i\in I(G(\tilde{x}^\nu,\tilde{y}^\nu)-\alpha \tilde{p}^\nu). \\
\end{aligned}
\end{equation}
\end{theorem}
\begin{proof} 
    The fact that, with probability 1, $S_{\alpha^\nu,\beta^\nu}$ is $C^1$ at $\tilde{x}^\nu$ comes directly from Lemma \ref{lemma:tame}.    Recall from the proof of Theorem \ref{thm:lip} that around each $\tilde{x}^\nu$ the mapping $S_{\alpha^\nu,\beta^\nu}$ takes the value of one of the elements of a finite collection of $C^1$ mappings $\varphi_{I^*,J^*}$, where $I^*$ and $J^*$ correspond to some subset of the active indices at $\tilde{x}^\nu$. The idea of the proof is to show that, with probability 1, the piece corresponding to the active indices $\varphi= \varphi_{I^0,J^0}$, with $I^0=I(G(\tilde{x}^\nu,\tilde{y}^\nu)-\alpha^\nu\tilde{p}^\nu)$, $J^0=J(G(\tilde{x}^\nu,\tilde{y}^\nu)-\alpha^\nu \tilde{p}^\nu)$, and $(\tilde{y}^\nu,\tilde{p}^\nu)=S_{\alpha^\nu,\beta^\nu}(\tilde{x}^\nu)$, is essentially active as defined in \cite{scholtes2012introduction}. In that case, we would have, by \cite[Proposition 4.1.3]{scholtes2012introduction},
    \begin{equation*}
        \nabla S_{\alpha^\nu,\beta^\nu}(\tilde{x}^\nu) = \nabla \varphi(\tilde{x}^\nu),
    \end{equation*}
    and it follows from the argument in the proof of Theorem \ref{thm:lip} that $\nabla \varphi(\tilde{x}^\nu)$ can be computed with the formula in \eqref{for:gradSr}. Specifically, we have that the function $\psi$ in \eqref{sys:active} is defined the same way as $\psi_x$ in \eqref{opt:ge}, which implies that the matrix $A$ in \eqref{for:gradSr} is the same in both cases. Furthermore, 
    \begin{equation*}
        K_{\widehat{h}^*}(p,z) = (K_{\widehat{h}}(z,p))^\perp = (N_{\partial \widehat{h}(z)}(p))^\perp,
    \end{equation*}
    where
    \begin{equation*}
        \partial \widehat{h}(z) = \left\{a^{\widehat{j}}+\sum_{j\in J^*}s_j(a^j-a^{\widehat{j}})+\sum_{i\in I^*}t_ib^i\mset s_i\in \reals,t_j\in \reals  \right\}.
    \end{equation*}
    Therefore, by replacing $h$ by $\widehat{h}$ in the definition of the matrix $B$ in \eqref{for:gradSr}, we obtain that it can be computed by solving the system \eqref{sys:khri}. 
    
    In order for $\varphi$ to be an essentially active piece at $\tilde{x}^\nu$, there must be an open subset $U$ of $\ball(x^\nu,\epsilon^\nu)$ whose closure contains $\tilde{x}^\nu$ such that $S_{\alpha^\nu,\beta^\nu}(x)=\varphi(x)$ for all $x\in U$. In the following, we will prove that for each combination $I^*,J^*$ of index sets, the set in which $S_{\alpha^\nu,\beta^\nu}(x)=\varphi_{I^*,J^*}(x)$ is either a set of this form or has measure zero. In particular, this would imply that $\varphi$ is an essentially active piece at $\tilde{x}^\nu$ with probability 1. 
    
    To prove our claim, we will show that the set $C=\{ x\in \reals^n \, |\, S_{\alpha^\nu,\beta^\nu}(x)=\varphi_{I^*,J^*}(x)\}$ is tame. \textcolor{black}{Indeed, for $r>0$, one has
    \begin{equation*}
        C\cap [-r,r]^n = \left \{ x\in [-r,r]^n ~\left|~ \exists y,p\text{ s.t. }\  \begin{aligned}
            &0  = \nabla_y g(x,y)+\nabla_y G(x,y)^\top p+\beta y, \\
            &(p,G(x,y)-\alpha^\nu p)\in \gph \partial h^*\cap \gph \partial \hat{h}^*
        \end{aligned} \right \}.\right.
    \end{equation*}
    Thus, $C\cap [-r,r]^n$ is the canonical projection over $\reals^n$ of a globally subanalytic set (note that $y$ and $p$ are bounded since $x$ is bounded and $S_{\alpha^\nu,\beta^\nu}$ is continuous) and therefore globally subanalytic, which implies that $C$ is tame by the definition in \cite{bolte2009tame}.} From \cite[Theorem 6.6]{coste2000introduction}, the set $C$ can then be partitioned into finitely many $C^k$ submanifolds for some $k\in \nats$. In particular, since manifolds of $\reals^n$ with dimension smaller than $n$ have measure zero, we have that at each $x\in C$, either there exists an $\epsilon>0$ such that $C\cap \ball(x,\varepsilon)$ has measure zero or $x\in U\subset C$, where $U$ is a submanifold of dimension $n$, which implies $x\in \cl (\nt U)=\cl U$ and then $\varphi$ is essentially active at $x$. 
\end{proof}

\subsection{Algorithm for Unconstrained Upper-Level Problem}

Algorithm \ref{alg:cap} extends the gradient sampling algorithm in \cite{kiwiel2007convergence} by evolving the objective function using the regularization parameters $\alpha, \beta$. These parameters are updated every time a stationarity target is reached and the sampling radius is shrunk.

\begin{algorithm}[ht]
    \caption{}\label{alg:cap}
\begin{algorithmic}
\Require{$x^0 \in \reals^n$, $ \eta^{{\rm opt}},\epsilon^{{\rm opt}},\alpha^{{\rm opt}},\beta^{{\rm opt}}\geq0$, $\eta^0,\alpha^0,\beta^0,\epsilon^0,>0$, $\mu_{\alpha},\mu_\beta,\mu_\epsilon,\mu_\eta\in (0,1)$, $\delta,\gamma\in (0,1)$, $N_{sam}\geq n+1$. Set $\Lambda=\emptyset$.}
\For {$\nu\in \nats$}
\State  Sample $\{x^{\nu,i}\}_{i=1}^{N_{sam}}$ independently and uniformly over $\mathbb{B}(x^\nu,\epsilon^\nu)$.
\For {$i=1,...,N_{sam}$}
\State Solve \eqref{prob:LLr} with $x=x^{\nu,i}$, $\alpha=\alpha^\nu$ and $\beta=\beta^\nu$ and compute $B$ in \eqref{for:gradSr}.
\State Compute $w^{\nu,i}=\begin{pmatrix}
    \bbI_n & \nabla Y_{\alpha^\nu,\beta^\nu}(x^{\nu,i})
\end{pmatrix}\nabla f(x^{\nu,i},Y_{\alpha^\nu,\beta^\nu}(x^{\nu,i}))$.
\EndFor
\State Compute the minimum norm element $w^\nu$ of $\con\{w^{\nu,i}, i = 1, \dots, N_{sam}\}$.
\If {$|w^\nu|\leq \eta^{{\rm opt}}$, $\epsilon^\nu\leq \epsilon^{{\rm opt}}$, $\alpha^{\nu}\leq \alpha^{{\rm opt}}$ and $\beta^{\nu}\leq \beta^{{\rm opt}}$}
\State STOP.\EndIf
\If {$|w^\nu|\leq \eta^\nu$}
\State Set $\eta^{\nu+1}=\mu_\eta \eta^\nu$, $\epsilon^{\nu+1}=\mu_\epsilon \epsilon^\nu$, $\alpha^{\nu+1}=\mu_\alpha \alpha^\nu$, $\beta^{\nu+1}=\mu_\beta\beta^\nu$. Set $x^{\nu+1}=x^\nu$. Replace $\Lambda$ by $\Lambda\cup \{\nu\}$.
\Else: 
\State Set $\eta^{\nu+1}= \eta^\nu$, $\epsilon^{\nu+1}= \epsilon^\nu$, $\alpha^{\nu+1}= \alpha^\nu$, $\beta^{\nu+1}=\beta^\nu$.
\State Set $d^\nu=-\frac{w^\nu}{|w^\nu|}$ and 
\begin{equation*}
    t_\nu = \max\big\{t\mset f(x^\nu+td^\nu,Y_{\alpha^\nu,\beta^\nu}(x^\nu+td^\nu))<f(x^\nu,Y_{\alpha^\nu,\beta^\nu}(x^\nu))-\delta t |w^\nu|, t\in \{1,\gamma, \gamma^2,...\}\big\}.
\end{equation*}
\State Set $x^{\nu+1}=x^\nu+t_\nu d^\nu$. \EndIf
\EndFor
\end{algorithmic}
\end{algorithm}

\textcolor{black}{}  

\begin{theorem}\label{thm:conv}{\rm (convergence of Algorithm \ref{alg:cap} with decreasing radius).}
    Suppose that $g$ and $G$ are analytic. Let $\{x^\nu\}_{\nu\in \Lambda}$ be the sequence generated by Algorithm \ref{alg:cap} with $\alpha^{{\rm opt}}=\beta^{{\rm opt}}=\epsilon^{{\rm opt}}=\eta^{{\rm opt}}=0$. With probability 1, the following hold: 
    \begin{enumerate}
        \renewcommand{\theenumi}{\roman{enumi}}
    \renewcommand{\labelenumi}{(\theenumi)}
    \item The algorithm does not stop \textcolor{black}{and $\dist (0,\partial_{\epsilon^\nu} f(\cdot,Y_{\alpha^\nu,\beta^\nu}(\cdot))(x^\nu))\to 0$ for $\nu\in \Lambda$}. 
    
    \item If $\bar{x}$ is a cluster point of $\{x^\nu\}_{\nu\in \Lambda}$ \textcolor{black}{with $(\bar{y},\bar{p})\in S(\bar{x})$ such that \eqref{CQ:llnor} holds at $\bar{y}$ for $\bar{x}$ and \eqref{SSOC} holds at $\bar{y}$ for $\bar{x}$ and $\bar{p}$}, then $\bar{x}$ is a Clarke stationary point of $f(\cdot,Y(\cdot))$, i.e., $0 \in \bar \partial f(\cdot, Y(\cdot))(\bar x)$.
    \end{enumerate}
\end{theorem}
\begin{proof} 
    Consider the event that there is $\nu_1\in \nats$ such that $\eta^\nu$, $\epsilon^\nu$, $\alpha^\nu$ and $\beta^\nu$ are constant for $\nu\geq \nu_1$. Since $\alpha^\nu$ and $\beta^\nu$ are fixed and, by Theorem \ref{thm:lip} and Lemma \ref{lemma:tame}, $f(\cdot,Y_{\alpha^\nu,\beta^\nu}(\cdot))$ is a locally Lipschitz function that is smooth almost everywhere and bounded from below, we obtain from the proof of \cite[Theorem 3.3]{kiwiel2007convergence} that the probability of this event is 0. Therefore, we can assume that $\eta^\nu\to 0$, $\epsilon^\nu\to 0$, $\alpha^\nu\to 0$, and $\beta^\nu\to 0$. \textcolor{black}{Since $\eta^\nu\to 0$, by construction we have that $|w^\nu|\to 0$, which by construction implies $\dist(0,\partial_{\epsilon^\nu} f(\cdot,Y_{\alpha^\nu,\beta^\nu}(\cdot))(x^\nu))\to 0$ for $\nu\in \Lambda$.}
    
    Let $\bar{x}$ be a cluster point of $\{x^\nu\}_{\nu\in \Lambda}$, i.e., $x^\nu \Nto \bar{x}$ for some subsequence indexed by $N\subset \Lambda$, and $(\bar{y},\bar{p})\in S(\bar{x})$. Suppose that \eqref{SSOC} holds at $\bar{y}$ for $\bar{x}$ and $\bar{p}$. Note that for a fixed $i$, the sequence of sampled points $\{x^{\nu,i}\}$ satisfies $x^{\nu,i}\Nto \bar{x}$.
    
    Therefore, by Theorem \ref{thm:alg} and the chain rule, there exists a subsequence (whose set of indices we also denote by $N$) such that $w^{\nu,1}\Nto \bar{w}^1\in \partial f(\cdot,Y(\cdot))(\bar{x})$. By the same argument, from this subsequence we can obtain a subsequence such that $w^{\nu,2}\Nto \bar{w}^2\in \partial f(\cdot,Y(\cdot))(\bar{x})$. Repeating this, we obtain a subsequence of $\{x^\nu\}$ such that $w^\nu\Nto 0$, $w^{\nu,i}\Nto \bar{v}^i\in \partial f(\cdot,Y(\cdot))(\bar{x})$ for all $i\in \{1,\hdots,N_{sam}\}$ and for every $\nu\in N$ we have $w^\nu\in \con \{w^{\nu,i} \}$. This implies that $0\in \con(\partial f(\cdot,Y(\cdot))(\bar{x}))$, which, since $f(\cdot,Y(\cdot))$ is locally Lipschitz around $\bar{x}$, implies that $0\in \bar{\partial}f(\cdot,Y(\cdot))(\bar x)$.
\end{proof}

\textcolor{black}{While Clarke stationarity is surely a necessary optimality condition, the conclusion\\ $\dist (0,\partial_{\epsilon^\nu} f(\cdot,Y_{\alpha^\nu,\beta^\nu}(\cdot))(x^\nu))\to 0$ in Theorem \ref{thm:conv} can be interpreted as one too. Under an assumption about objective function values near a local minimizer $\bar x$, we have the following weaker optimality condition (related to those from gradient-free optimization; see, e.g., \cite{ermoliev1995minimization,andrieu2024gradient}) expressed in terms of Goldstein subdifferentials of the regularized hyper-objectives.}

\textcolor{black}{
\begin{proposition}\label{prop:weak}
    For an open neighborhood $V$ of a point $\bar x\in \reals^n$, suppose that, for every $x\in V$, $Y(x)$ is bounded, $\min_{y\in Y(x)}f(x,y)=\max_{y\in Y(x)}f(x,y)$, and the basic qualification \eqref{CQ:llnor} holds at some $y\in Y(x)$ \, for $x$. If $\bar{x}$ is a local minimizer of the bilevel problem \eqref{prob:UL} and $\alpha^\nu,\beta^\nu,\epsilon^\nu\searrow 0$, then 
    \[
    \text{there exists a sequence $\bar x^\nu\to \bar{x}$ such that }  \dist (0,\partial_{\epsilon^\nu} f(\cdot,Y_{\alpha^\nu,\beta^\nu}(\cdot))(\bar x^\nu))\to 0.
    \]
\end{proposition}}
\begin{proof}
    \textcolor{black}{Let $W\subset V$ be compact and such that $\bar{x}\in \nt W$. We will first prove the epi-convergence $\psi^\nu\eto \psi$ for the functions given by $\psi^\nu = f(\cdot,Y_{\alpha^\nu,\beta^\nu}(\cdot))+\iota_W$ and $\psi = \inf_{y\in Y(\cdot)} f(\cdot,y)+\iota_W$. Let $x\in W$, $x^\nu \to x$, and $\varphi$ and $\varphi^\nu$ be as in \eqref{eq:varphinu} with the trivial change that $\bar x$ in \eqref{eq:varphinu} is replaced by $x$. Then Lemma \ref{lemma:epiconvergence} applies and $\varphi^\nu\eto \varphi$. Setting $y^\nu=Y_{\alpha^\nu,\beta^\nu}(x^\nu)$, from the fact that $\varphi^\nu\eto \varphi$, $Y(x)$ is bounded, and $\phi$ is proper, we can extract a subsequence such that $y^\nu\Nto \hat{y}\in Y(x)$; cf. \cite[Exercise 7.32(c) and Theorem 7.33]{VaAn}. Then, $f(x^\nu,y^\nu)\Nto f(x,\hat{y})= \psi(x)$ because the upper-level value is the same for every $\hat{y}\in Y(x)$ by assumption. Since every subsequence has a further subsequence converging to $\psi(x)$, we conclude that $f(x^\nu,y^\nu)\to \psi(x)$. In particular, we have $\liminf\psi^\nu(x^\nu)\geq \psi(x)$. Choosing $x^\nu=x$ for all $\nu\in \nats$, we also have $\psi^\nu(x^\nu)=f(x^\nu,y^\nu)\to \inf_{y\in Y(x)} f(x,y)=\psi(x)$. Together, we conclude that $\psi^\nu\eto \psi$.}
    
    \textcolor{black}{The rest of the proof follows the proof of \cite[Theorem 3]{andrieu2024gradient}. Let $\bar{x}$ be a local minimizer of \eqref{prob:UL}. Define $\gamma(x)=\psi(x)+|x-\bar{x}|^2$ and $\gamma^\nu(x)=\psi^\nu(x)+|x-\bar{x}|^2$ and suppose that $W$ (if necessary, taking a smaller neighborhood) is such that $\bar{x}$ is the unique minimizer of $\gamma$. Let $\bar x^\nu$ be a minimizer of $\gamma^\nu$ for all $\nu\in \nats$. From the argument above, we have that $\psi^\nu\eto \psi$ and $\psi^\nu(x)\to \psi(x)$ for all $x\in \reals^n$. Therefore, $\gamma^\nu\eto \gamma$. Moreover, the term $\iota_W$ makes the sequence of functions $\gamma^\nu$ level-bounded. With this and the fact that $\bar{x}$ is the unique minimizer of $\gamma$, we conclude that $\bar x^\nu\to \bar{x}$ by \cite[Theorem 7.33]{VaAn}. Then, for all sufficiently large $\nu$ (such that $\bar x^\nu$ is in the interior of $W$), the Fermat rule gives
    \begin{equation*}
        0\in \bar{\partial} \gamma^\nu(\bar x^\nu) = \bar{\partial} \psi^\nu(\bar x^\nu)+2(\bar x^\nu-\bar{x}). 
    \end{equation*}
    Thus, $w^\nu\in \bar{\partial} \psi^\nu(\bar x^\nu)$ with $w^\nu=-2(\bar x^\nu-\bar{x})$. This implies $|w^\nu|\to 0$, from which the conclusion in the statement of the proposition follows.}
\end{proof}

\textcolor{black}{Clarke stationarity of a cluster point of the iterates of Algorithm \ref{alg:cap} is only guaranteed by Theorem \ref{thm:conv} when \eqref{SSOC} is satisfied for the lower-level problem at this cluster point. If this is not the case, then the mapping $Y$ is not necessarily single-valued around this point and the concept of stationarity for the actual bilevel problem must be defined differently, for example through the weaker notion in Proposition \ref{prop:weak}. Next, we consider a sufficient assumption---motivated by \cite{dempe2000bundle}---for \eqref{SSOC} to be satisfied at {\em any} cluster point produced by Algorithm \ref{alg:cap}. Informally, the assumption states that for small enough values of $\alpha$ and $\beta$, each approximately stationary point of the regularized upper-level objective function has a neighborhood in which \eqref{SSOC} is satisfied. That is, \eqref{SSOC} may not hold everywhere, but it does hold in a set large enough to contain the cluster points of the iterates of Algorithm 1.}

   Suppose that there exist a set $D \subset \reals^n$ and positive scalars $\theta$, $\alpha^*, \beta^*, \epsilon^*, \delta^*$ such that
\begin{equation*}
    D + \ball(0,\theta)
    \;\subset\;
    \big\{\, x \in \reals^n \mid \exists \{(y,p)\}=S(x) \text{ and } \text{\eqref{SSOC} is satisfied at } y \text{ for } x \text{ and } p\,\big\}
\end{equation*}
and 
\begin{equation*}
    0 < \alpha < \alpha^*, ~0 < \beta < \beta^*, ~0 < \epsilon < \epsilon^*, ~0 < \delta < \delta^*,~
    \partial_\epsilon \big( f(\cdot, Y_{\alpha,\beta}(\cdot)) \big)(x)
    \;\cap\;
    \mathbb{B}(0,\delta)
    \neq \emptyset
    ~~\Longrightarrow~~
    x \in D.
\end{equation*}
Then, if $\{x^\nu\}_{\nu\in \Lambda}$ is defined as in Algorithm \ref{alg:cap}, then $x^\nu\in D$ for all sufficiently large $\nu$. In particular, if $\bar{x}$ is a cluster point of this sequence and $(\bar{y},\bar{p})\in S(\bar{x})$, then \eqref{SSOC} holds at $\bar{y}$ for $\bar{x}$ and $\bar{p}$.

This assumption is trivially satisfied when \eqref{SSOC} holds at the minimizer of \eqref{prob:LL} for every $x\in \reals^n$, which occurs in the well-studied setting where the objective function of \eqref{prob:LL} is smooth and strongly convex and the problem is unconstrained, but it extends this classical framework in the context of constrained lower-level problems by allowing multiple minimizers or even infeasibility for some values of $x$. The following example illustrates such situations.

\vspace{0.3cm}

\begin{example}
    Consider the problem of minimizing $y^2$ over $x,y\in \reals$ subject to $y$ being a minimizer of
    \begin{equation}\label{prob:ex3}
    \begin{aligned}
        \underset{y\in \reals}{\text{minimize}}& \quad y^2+xy \\
        \text{s.t }\quad  &  x+y\leq 1,\  x-y\leq 100.
    \end{aligned}
    \end{equation}
    If $x<2$, then the unique minimizer of \eqref{prob:ex3} is $y=-\frac{x}{2}$ and both constraints are inactive at this point, which implies that \eqref{SSOC} holds since the lower-level objective is smooth and strongly convex. However, \eqref{prob:ex3} is infeasible when $x>50.5$, which makes this problem not fit classical frameworks. On the other hand, for any $\alpha,\beta>0$, the regularized lower-level problem is feasible and its primal solution mapping is given by
    \begin{equation*}
        Y_{\alpha,\beta}(x) =\begin{cases}
            \frac{-x}{2+\beta} \quad &\text{if } x\leq\frac{2+\beta}{1+\beta}, \\
            \frac{1-(\alpha+1)x}{2\alpha+\alpha\beta+1} \quad &\text{if } \frac{2+\beta}{1+\beta}<x<\frac{1+100(2\alpha+\beta\alpha+1)}{3\alpha+\alpha\beta+2}, \\
            \frac{-99-\alpha x}{2\alpha+\alpha\beta+2} \quad &\text{if } x\geq \frac{1+100(2\alpha+\beta\alpha+1)}{3\alpha+\alpha\beta+2}.
        \end{cases}
    \end{equation*}
    Since $Y_{\alpha,\beta}$ is piecewise linear and its slope is always negative, the only stationary point of $f(\cdot,Y_{\alpha,\beta}(\cdot))$ is $x=0$. Moreover, there exists a neighborhood $U$ of 0 such that $x\in U$ whenever
    \begin{equation*}
       |x-x'|\leq \epsilon~ \text{ and }  ~| \nabla f(\cdot,Y_{\alpha,\beta}(\cdot))(x')|\leq \delta
    \end{equation*}
    for some $x'$ and sufficiently small $\alpha,\beta,\epsilon,\delta>0$. We conclude that our sufficient assumption is satisfied for this problem. 
\end{example}

\vspace{0.3cm}

An alternative computational approach is to run Algorithm \ref{alg:cap} with a fixed sampling radius ($\epsilon^\nu = \epsilon^0$ for all $\nu$), decreasing only the parameters $\eta$, $\alpha$, and $\beta$ to 0. Note that, unlike in \cite[Theorem 3.6]{kiwiel2007convergence}, we cannot set $\eta^0=0$ unless we run the algorithm directly on \eqref{prob:UL} (without introducing a regularization), which is a setting possible only when \eqref{SSOC} is satisfied everywhere.

\begin{theorem}\label{thm:convfixeps}{\rm (convergence of Algorithm \ref{alg:cap} with fixed radius.)}
    Suppose that $g$ and $G$ are analytic. Let $\{x^\nu\}_{\nu \in \Lambda}$ be the sequence generated by Algorithm \ref{alg:cap} with $\alpha^{{\rm opt}}=\beta^{{\rm opt}}=\eta^{{\rm opt}}=0$ and $\epsilon^0=\epsilon^{{\rm opt}}=\epsilon^*>0$. With probability 1, the algorithm does not stop. Moreover, let $\bar{x}$ be a cluster point of $\{x^\nu\}_{\nu\in \Lambda}$. If $\epsilon^*$ is such that \eqref{CQ:llnor} and \eqref{SSOC} hold at $y\in Y(x)$ for all $x\in \ball(\bar{x},\epsilon)$, then $\bar{x}$ satisfies $0\in \partial_{\epsilon} f(\cdot,Y(\cdot))(\bar{x})$.
\end{theorem}
\begin{proof} 
    As in the proof of Theorem \ref{thm:conv}, we can assume that $\eta^\nu\to 0$, $\alpha^\nu\to 0$, and $\beta^\nu\to 0$ since this occurs with probability 1. 
    
    Let $\bar{x}$ be a cluster point of $\{x^\nu\}$, i.e., $x^\nu \Nto \bar{x}$ for some subsequence indexed by $N\subset\Lambda$. Note that for a fixed $i$, the sequence of sampled points $\{x^{\nu,i}\}$ satisfies $|x^{\nu,i}- \bar{x}|\leq \epsilon$ for every $\nu\in \nats$. Therefore, since the sequence is bounded, it has a cluster point $\bar{x}^{i}\in \ball(\bar{x}, \epsilon)$. By assumption, \eqref{SSOC} is satisfied at  $\bar{x}^{i}$. Then, by Theorem \ref{thm:alg} and the chain rule, there exists $\bar{x}^1\in \ball(\bar{x},\epsilon)$ and a subsequence such that $w^{\nu,1}\Nto \bar{w}^1\in \partial f(\cdot,Y(\cdot))(\bar{x}^1)$. By the same argument, from this subsequence we can obtain a subsequence such that $w^{\nu,2}\Nto \bar{w}^2\in \partial f(\cdot,Y(\cdot))(\bar{x}^2)$, where $\bar{x}^2\in \ball(\bar{x},\epsilon)$. Repeating this, we obtain a subsequence of $\{x^\nu\}$ whose set of indices we also denote by $N$ such that $w^\nu\Nto 0$, $w^{\nu,i}\Nto \bar{w}^i\in \partial f(\cdot,Y(\cdot))(\bar{x}^{i})$ for all $i\in \{1,\hdots,N_{sam}\}$ and for every $\nu\in N$ we have $w^\nu\in \con \{w^{\nu,i} \}$. This implies that $0\in \con ( \partial f(\cdot,Y(\cdot))(\ball(\bar{x},\epsilon)))$ which, by definition, implies that $0\in \partial_{\epsilon} f(\cdot,Y(\cdot))(\bar x)$.
\end{proof}

\subsection{Algorithm for Constrained Upper-Level Problem}

The following analysis extends the SQP-GS method in \cite{curtis2012sequential} by accounting for updates in the objective function, which correspond to updates to the regularization parameters. For $\rho\geq0$, we define the penalty-augmented objective functions
\begin{equation}\label{eq:penalty}
    \phi_{\rho,\alpha,\beta}(x) = \rho f(x,Y_{\alpha,\beta}(x))+ \sum_{k=1}^r\max\{c_k(x),0\}, \qquad \phi_\rho(x)=\rho f(x,Y(x))+ \sum_{k=1}^r\max\{c_k(x),0\},
\end{equation}
where $\phi_\rho$ is defined whenever $Y(x)$ is a singleton.

Following \cite{curtis2012sequential}, we say that $x$ is \textcolor{black}{Curtis-Overton} stationary for $\phi_{\rho,\alpha,\beta}$ if, for all positive $\epsilon$, 0 is the minimizer of the problem
\begin{equation}\label{prob:sqp}
\begin{aligned}
    \underset{d\in \reals^n}{\text{minimize}}\  &q_{\rho,\alpha,\beta}(d;x,\mathbf{B}_\epsilon)\\    &= \rho \sup_{x'\in \ball(x,\varepsilon)}\lset f(x,Y_{\alpha,\beta}(x))+ \big\langle\nabla\big(f(\cdot,Y_{\alpha,\beta}(\cdot) \big)(x'), d\big\rangle\rset \\
    &+ \sum_{k=1}^r \sup_{x'\in \ball(x,\varepsilon)} \lset \max\{c_k(x)+\langle\nabla c_k(x'), d\rangle,0\} \rset +\frac{1}{2}|d|^2,
\end{aligned}
\end{equation}
where $\mathbf{B}_\epsilon=\{\ball(x,\varepsilon),\hdots,\ball(x,\varepsilon)\}$ is a list containing $r+1$ copies of $\ball(x,\epsilon)$; the third argument of $q_{\rho,\alpha,\beta}$ takes a list of $r+1$ sets, where the first one is used in the first supremum in \eqref{prob:sqp} and the other $r$ in each supremum in the sum. The definition of stationarity for $\phi_\rho$ is analogous, with a minimization problem analogous to \eqref{prob:sqp} whose objective function we denote by $q_\rho$. In \eqref{prob:sqp} and the remainder of the subsection, whenever we take gradients of a smooth almost everywhere function over a set, we are actually referring to every point in the set at which the function is differentiable.

As a measure to quantify lack of stationarity for $\phi_{\rho,\alpha,\beta}$, we define
\begin{equation*}
    \Delta q_{\rho,\alpha,\beta}(d;x,\mathbf{B}_\epsilon)=\phi_{\rho,\alpha,\beta}(x)-q_{\rho,\alpha,\beta}(d;x,\mathbf{B}_\epsilon); 
\end{equation*}
it is nonnegative for every $d\in \reals^n$ and measures the suboptimality of the zero-vector for \eqref{prob:sqp}. In the following Algorithm \ref{alg:const}, we abuse notation slightly be allowing the third argument of $q_{\rho,\alpha,\beta}$ to also be a finite set. 

\begin{algorithm}[ht]
    \caption{}\label{alg:const}
\begin{algorithmic}
\Require $x^0 \in \reals^n$, $\alpha^{{\rm opt}},\beta^{{\rm opt}},\epsilon^{{\rm opt}}\geq0$, $\alpha^0,\beta^0,\epsilon^0>0$,$\rho^0>0$, $\theta^0>0$,$\eta>0$, $\mu_\rho,\mu_\theta,\mu_\alpha,\mu_\beta,\mu_\epsilon\in (0,1)$, $\delta,\gamma\in (0,1)$, $N_{sam}\geq n+1$. Set $\Lambda=\emptyset$.
\For {$\nu\in \nats$}
\State  Sample $\{x^{\nu,i}\}_{i=1}^{N_{sam}}$, $\{x_k^{\nu,i}\}_{i=1}^{N_{sam}}$ for $k=1,\hdots,r$ independently and uniformly in $\mathbb{B}(x^\nu,\epsilon^\nu)$. Set $\mathbf{B}^\nu=\big\{ \{x^{\nu,i}\}_{i=1}^{N_{sam}}, \{x_1^{\nu,i}\}_{i=1}^{N_{sam}}, \dots, \{x_r^{\nu,i}\}_{i=1}^{N_{sam}}\big\}.$
\For {$i=1,...,N_{sam}$}
\State Solve \eqref{prob:LLr} with $x=x^{\nu,i}$, $\alpha=\alpha^\nu$ and $\beta=\beta^\nu$ and compute $B$ in \eqref{for:gradSr}.
\State Compute $w^{\nu,i}=\begin{pmatrix}
    \bbI_n & \nabla Y_{\alpha^\nu,\beta^\nu}(x^{\nu,i})
\end{pmatrix}\nabla f(x^{\nu,i},Y_{\alpha^\nu,\beta^\nu}(x^{\nu,i}))$ and $w_k^{\nu,i}=\nabla c_k(x_k^{\nu,i})$ for $k=1,\dots,r$.
\EndFor
\State Compute $d^\nu$ by solving problem \eqref{prob:sqp} for $q_{\rho^\nu,\alpha^\nu,\beta^\nu}(d;x^\nu,\mathbf{B}^\nu)$.
\If {$\Delta q_{\rho^\nu,\alpha^\nu,\beta^\nu}(d^\nu;x^\nu,\mathbf{B}^\nu)\leq \eta(\epsilon^\nu)^2$, $\epsilon^\nu\leq \epsilon^{{\rm opt}}$, $\alpha^{\nu}\leq \alpha^{{\rm opt}}$ and $\beta^{\nu}\leq \beta^{{\rm opt}}$}
\State STOP.\EndIf
\If {$\Delta q_{\rho^\nu,\alpha^\nu,\beta^\nu}(d^\nu;x^\nu,\mathbf{B}^\nu)\leq \eta(\epsilon^\nu)^2$}
\If {$\sum_{k=1}^r \max\{c_k(x^\nu),0\}\leq \theta^\nu$}
\State Set $\rho^{\nu+1}=\rho^\nu$, $\theta^{\nu+1}=\mu_{\theta}\theta^\nu$.
\Else :
\State Set $\rho^{\nu+1}=\mu_\rho \rho^\nu$, $\theta^{\nu+1}=\theta^\nu$.
\EndIf
\State Set $\epsilon^{\nu+1}=\mu_\epsilon \epsilon^\nu$, $\alpha^{\nu+1}=\mu_\alpha \alpha^\nu$, $\beta^{\nu+1}=\mu_\beta\beta^\nu$. Set $x^{\nu+1}=x^\nu$. Replace $\Lambda$ by $\Lambda\cup\{\nu\}$.
\Else: 
\State Set $\epsilon^{\nu+1}= \epsilon^\nu$, $\alpha^{\nu+1}= \alpha^\nu$, $\beta^{\nu+1}=\beta^\nu$, $\rho^{\nu+1}=\rho^\nu$, $\theta^{\nu+1}=\theta^\nu$.
\State Compute 
\begin{equation*}
\begin{aligned}
    t_\nu = \max\big\{t\mset\phi_{\rho^\nu,\alpha^\nu,\beta^\nu}(x^\nu+td^\nu)<\phi_{\rho^\nu,\alpha^\nu,\beta^\nu}(x^\nu)-&\delta t \Delta q_{\rho^\nu,\alpha^\nu,\beta^\nu}(d^\nu;x^\nu,\mathbf{B}^\nu),\\
    &t\in \{1,\gamma, \gamma^2,...\}\big\}.
\end{aligned}
\end{equation*}
\State Set $x^{\nu+1}=x^\nu+t_\nu d^\nu$. \EndIf
\EndFor
\end{algorithmic}
\end{algorithm}

We present below the main convergence result for Algorithm \ref{alg:const}. It extends \cite[Theorem 3.3]{curtis2012sequential} by considering regularized problems; the assumption about bounded functions and partial derivatives in its statement is the analog of Assumption 3.2 in the same article.

\begin{theorem}\label{thm:const}{\rm (convergence of Algorithm \ref{alg:const}).}
    Suppose that $g$ and $G$ are analytic. Let $\{x^\nu\}_{\nu\in \Lambda}$ be the sequence generated by Algorithm \ref{alg:const} with $\alpha^{{\rm opt}}=\beta^{{\rm opt}}=\epsilon^{{\rm opt}}=0$, and suppose that $\{x^\nu\}_{\nu\in \nats}$ lies in a convex set over which $f$ and $c$ and their first partial derivatives are bounded. With probability 1, the algorithm does not stop. Moreover, if $\bar{x}$ is a cluster point of $\{x^\nu\}_{\nu\in \Lambda}$ \textcolor{black}{with $(\bar{y},\bar{p})\in S(\bar{x})$ such that \eqref{CQ:llnor} holds at $\bar{y}$ for $\bar{x}$ and \eqref{SSOC} holds at $\bar{y}$ for $\bar{x}$ and $\bar{p}$}, then $\bar{x}$ is a Curtis-Overton stationary point of $\phi_{\bar{\rho}}$, where $\bar{\rho}=\lim \rho^\nu$.
\end{theorem}
\begin{proof} 
    We obtain from the proof of \cite[Theorem 3.3]{curtis2012sequential} that, with probability 1, $\alpha^\nu\to0$, $\beta^\nu\to 0$ and $\epsilon^\nu\to 0$. \par
    Let $\bar{x}$ be a cluster point of $\{x^\nu\}_{\nu\in \Lambda}$ \textcolor{black}{such that there exist $(\bar{y},\bar{p})\in S(\bar{x})$ satisfying \eqref{CQ:llnor} at $\bar{y}$ for $\bar{x}$ and \eqref{SSOC} at $\bar{y}$ for $\bar{x}$ and $\bar{p}$}. From the same argument as in the proof of Theorem \ref{thm:conv}, there is a subsequence of $\{x^\nu\}_{\nu\in \Lambda}$ such that for all $i\in \{1,\hdots,N_{sam}\}$, one has
    \begin{equation}\label{a}
        w^{\nu,i} \Nto \bar{w}^i\in \partial f(\cdot,Y(\cdot))(\bar{x}).
    \end{equation}
    Moreover, since $c$ is locally Lipschitz, we can find a subsequence such that \eqref{a} is satisfied and for all $k\in \{1,\hdots,r\}$ and $i\in \{1,\hdots,N_{sam}\}$ 
    \begin{equation*}
        w_k^{\nu,i} \Nto \bar{w}_k^i \in \partial c_k(\bar{x}).
    \end{equation*}
    Then, using a standard argument, it is possible to show that the sequence of functions
    \begin{equation*}
    \begin{aligned}
        q^\nu(d) &= \rho^\nu \max_{i\in \{1,\hdots,N_{sam}\}}\lset f(x^\nu,Y_{\alpha,\beta}(x^\nu))+ \langle w^{\nu,i}, d\rangle \rset \\
        &+ \sum_{k=1}^r \max_{i\in \{1,\hdots,N_{sam}\}} \lset \max\{c_k(x^\nu)+\langle w_k^{\nu,i}, d\rangle,0\} \rset +\frac{1}{2}|d|^2
    \end{aligned}
    \end{equation*}
    epi-converges to
    \begin{equation*}
    \begin{aligned}
        q(d) &= \bar{\rho} \max_{i\in \{1,\hdots,N_{sam}\}}\lset f(\bar{x},Y(\bar{x}))+ \langle\bar{w}^{i}, d\rangle \rset \\
        &+ \sum_{k=1}^r \max_{i\in \{1,\hdots,N_{sam}\}} \lset \max\{c_k(\bar{x})+\langle\bar{w}_k^{i}, d\rangle,0\} \rset +\frac{1}{2}|d|^2,
    \end{aligned}
    \end{equation*}
    when $\nu\to \infty$ along the subsequence. 
    
    From the fact that $\Delta q_{\rho^\nu,\alpha^\nu,\beta^\nu}(d^\nu;x^\nu,\mathbf{B}^\nu)\leq \eta(\epsilon^\nu)^2$ for every $\nu$ in the subsequence, we conclude that $0\in \eta(\epsilon^\nu)^2\text{-}\nargmin q^\nu$. Then, since $q^\nu\eto q$ and $\epsilon^\nu\to 0$, we must have $0\in \nargmin q$. We claim that this implies that $\bar{x}$ is Curtis-Overton stationary for $\varphi_{\bar{\rho}}$, which would conclude the proof. \par
    In order to prove our claim, note that, since $f(\cdot,Y(\cdot))$ is locally Lipschitz at $\bar{x}$, we have that for all $\epsilon>0$
    \begin{equation*}
        \{\bar{w}^i\}_{i=1}^{N_{sam}} \subset \cl \lset \nabla\left( f(\cdot,Y(\cdot))\right)(x) \mset x\in \ball(\bar{x},\epsilon)\rset.
    \end{equation*}
    Therefore, for any $d\in \reals^n$,
    \begin{equation*}
        \max_{i\in \{1,\hdots,N_{sam}\}} \langle\bar{w}^i, d\rangle \leq \sup_{x\in \ball(\bar{x},\epsilon)} \langle   \nabla\left(f(\cdot,Y(\cdot) \right)(x), d\rangle. 
    \end{equation*}
    Similarly, we have that for all $d\in \reals^n$ and for all $k\in \{1,\hdots,l\}$,
    \begin{equation*}
        \max_{i\in \{1,\hdots,N_{sam}\}} \langle\bar{w}_k^i, d\rangle \leq \sup_{x\in \ball(\bar{x},\epsilon)}  \langle \nabla c_k(x),d\rangle. 
    \end{equation*}
    Then we obtain that
    \begin{equation*}
        \forall d\in \reals^n,\forall\epsilon>0, \quad q(d)\leq q_{\bar{\rho}} (d;\bar{x},\mathbf{B}_\epsilon).
    \end{equation*}
    With this,
    \begin{equation*}
        q_{\bar{\rho}} (0;\bar{x},\mathbf{B}_\epsilon) = q(0)= \inf_{d\in \reals^n} q(d) \leq \inf_{d\in \reals^n} q_{\bar{\rho}} (d;\bar{x},\mathbf{B}_\epsilon),
    \end{equation*}
    and so $\bar{x}$ is a Curtis-Overton stationary point for $\phi_{\bar{\rho}}$.
\end{proof}

\textcolor{black}{If there exists a local minimizer $x\in \reals^n$  of $\phi_{\rho,\alpha,\beta}$ for some $\rho>0$ such that $\sum_{k=1}^r\max\{c_k(x),0\}=0$, then $x$ is also a minimizer of $f(\cdot,Y_{\alpha,\beta}(\cdot))$ over $X$. The same notion can be extended to $\phi_\rho$, given that the hyper-objective of \eqref{prob:UL} is locally Lipschitz around $x$, which, by Theorem \ref{thm:lipnor}, is true when \eqref{SSOC} is satisfied at $y\in Y(x)$ for $x$. A sufficient condition to guarantee the existence of such $\rho$ is to have $f(\cdot,Y_{\alpha,\beta}(\cdot))$ be calm at the minimizer $x$; for more details, see \cite[Theorem 2.1]{curtis2012sequential} and discussion thereafter. Next, we present a result that indicates that, under some conditions, calmness is achieved for both the actual and the regularized problem, making Algorithm \ref{alg:const} appropriate for solving the bilevel problem  \eqref{prob:UL} and not only for minimizing its penalized surrogate $\phi_\rho$. We say \eqref{prob:UL} is calm at $\bar{x}$ with modulus $\epsilon>0$ if there exists a neighborhood $V$ of $\bar{x}$ such that for all $x\in V$ and $u\in \reals^{r}$ satisfying $c(x)\leq u$, we have $\psi(\bar{x})-\psi(x)\leq \epsilon |u|$, where 
\begin{equation*}
    \psi(x) = \inf_{y\in Y(x)} f(x,y).
\end{equation*}}\textcolor{black}{For $\alpha^\nu,\beta^\nu>0$, we define $\psi^\nu(x)=f(x,Y_{\alpha^\nu,\beta^\nu}(x))$ and use the
same definition of calmness for the problem of minimizing $\psi^\nu$ over $X$, with $\psi$ replaced by
$\psi^\nu$.}

\begin{proposition}\label{prop:calmness}
\textcolor{black}{ Let $\alpha^\nu,\beta^\nu\searrow0$.   Suppose that for all $\nu\in \nats$, $x^\nu$ is a local minimizer of 
\[
\psi^\nu = f\big(\cdot, Y_{\alpha^\nu,\beta^\nu}(\cdot)\big)~\text{ over } ~ X,
\]
and that $x^\nu\to \bar{x}$, where $\bar{x}$ is a local minimizer of \eqref{prob:UL}. Suppose also that $(\bar{y},\bar{p})\in S(\bar{x})$, \eqref{CQ:llnor} holds at $\bar{y}$ for $\bar{x}$, \eqref{SSOC} holds at $\bar{y}$ for $\bar{x}$ and $\bar{p}$, and that the feasible set $X$ is defined in terms of $C^1$ functions $c_1, \dots, c_r$ satisfying the Mangasarian-Fromovitz constraint qualification at $\bar{x}$.} 
    
\textcolor{black}{    Then, the bilevel problem \eqref{prob:UL} is calm at $\bar{x}$ with some modulus $\epsilon^\star>0$ and there exist $\nu_0\in \nats$ such that for all $\nu\geq \nu_0$ the problem of minimizing $\psi^\nu$ over $X$ is calm at $x^\nu$ with the same modulus $\epsilon^\star$. In particular, there exists a common exact-penalty parameter $\rho^\star>0$ such that $\bar{x}$ and $x^\nu$ are local minimizers of $\phi_\rho$ and $\phi_{\rho,\alpha^\nu,\beta^\nu}$, respectively, for all $\nu\geq\nu_0$.}
\end{proposition}
\begin{proof}
    \textcolor{black}{We first argue that the functions $\{Y_{\alpha^\nu,\beta^\nu}\}_{\nu\in \nats}$ are locally Lipschitz with common constant $L>0$ in a neighborhood of $\bar{x}$ by using the fact that \eqref{SSOC} holds at $\bar{y}$ for $\bar{x}$ and $\bar{p}$. From the argument in the proof of Theorem \ref{prop:conv2}, we have that there exist $\nu_0\in \nats$ and a neighborhood $V$ of $\bar{x}$ such that for all $\nu\geq \nu_0$ and $x\in V$ such that $S_{\alpha^\nu,\beta^\nu}(x)$ is smooth, 
    \begin{equation*}
        \|\nabla S_{\alpha^\nu,\beta^\nu}(x)- M\|\leq \kappa,
    \end{equation*}
    for some $\kappa>0$, matrix norm $\|\cdot\|$, and matrix $M$, possibly depending on the set of active indices at $x$. Since there are only finitely many possible sets of active indices, we conclude that $\|\nabla S_{\alpha^\nu,\beta^\nu}\|$ is bounded on $V$, say by $L>0$, for all $\nu\geq \nu_0$. In particular, this implies that $Y_{\alpha^\nu,\beta^\nu}$ is locally Lipschitz with constant $L$ for all $\nu\geq \nu_0$. We also conclude from the same argument that $Y$ is locally Lipschitz on $V$ with constant $L$.} \par
    \textcolor{black}{From \cite[Theorem 2.4]{burke1991exact}, if the Mangasarian-Fromovitz constraint qualification holds at $\bar{x}$, there exists a neighborhood of $\bar{x}$ such that the constraint system is regular with a common modulus $\kappa$, as defined in \cite[Definition 2.2]{burke1991exact}. Then, by \cite[Corollary 2.3.2]{burke1991exact}, the family of functions $\{\psi, \psi^\nu, \nu\in\nats\}$ being equi-Lipschitz with constant $L$ and the constraint system $c(x)\leq0$ being regular at $\bar{x}$ and each $x^\nu$ with common modulus $\kappa$ imply that the corresponding problems are calm at their respective local minimizer with the same modulus $\epsilon^*=\kappa L$. The existence of the common penalty parameter follows then by \cite[Corollary 2.3.1]{burke1991exact}.}  
\end{proof}

\textcolor{black}{
\subsection{Computational Implementation}
We discuss next how the quantities appearing in Algorithms \ref{alg:cap} and \ref{alg:const} can be computed using standard optimization and linear-algebra routines. The results about the algorithms assume exact solutions, but in practice an implementation will involve solving the following steps up to some numerical tolerance.}\par
\textcolor{black}{
Fix $x\in \reals^n$ and $\alpha,\beta>0$. Although the regularized lower-level problem is convex and smooth, the presence of the Moreau envelope makes direct gradient-based methods impractical, as evaluating its gradient requires solving an inner optimization problem. Instead, note that, from the definition of Moreau envelopes, solving \eqref{opt:ll} is equivalent to minimize
\begin{equation}\label{prob:llcomputation}
     h(w)+\frac{1}{2\alpha}|w-G(x,y)|^2+g(x,y)+\frac{\beta}{2}|y|^2
\end{equation}
over all $w\in \reals^s$, $y\in \reals^m$. When $G$ is linear in $y$, this is a convex problem and can be solved efficiently, for example, by using a proximal gradient method. Otherwise, proximal composite methods may apply. Then, in order to obtain the value of $p$, it is sufficient to set $p=(G(x,y)-w)/\alpha$. Indeed, we have from the optimality conditions for \eqref{prob:llcomputation} 
\[
    0  = \nabla_y{G(x,y)}^\top (G(x,y)-w)/\alpha + \nabla_y g(x,y)+\beta y, ~~~ (G(x,y)-w)/\alpha\in \partial h(w),
\]
and then the formula for $p$ follows from \eqref{opt:ge}.} 

\textcolor{black}{
We now turn to the computation of the Jacobian $\nabla Y_{\alpha,\beta}(x)$, for which the first step is to compute the matrix $B$. The active constraint indices and maximizing pieces are obtained by evaluating $\langle a^j,w\rangle-\alpha_j$ for $j\in J$ and $\langle b^i,w\rangle-\beta_i$ for $i\in I$, where $w$ is the proximal variable from \eqref{prob:llcomputation}. From the proof of Theorem \ref{thm:alg}, $B$ is a matrix whose columns form a basis of the orthogonal of the solution set of the homogeneous system in \eqref{sys:khri}. Then, computing $B$ is equivalent to computing an orthonormal basis of the null space of the matrix that defines the system, which can be achieved with standard linear-algebra techniques such as QR factorization. With this, the Jacobian $\nabla Y_{\alpha,\beta}(x)$ is available and therefore the regularized hyper-gradient is too. }\par
\textcolor{black}{
In an iteration of Algorithms \ref{alg:cap} and \ref{alg:const}, the lower-level solve and computation of the Jacobian must be realized at every sampled point $x^{\nu,i}$. This implies that the cost of implementing one iteration of the algorithm scales with the dimension of the problem. This is a known weakness of gradient sampling algorithms (c.f. discussion in Section 4.3 of \cite{burke2020gradient}), and some work has been done to address this issue; in Section 6, we show some evidence that taking $N_{sam}<n$ can produce good computational results in practice at significantly reduced computing time. Regardless of the value of $N_{sam}$, once all the gradients are computed, both the minimum norm problem in Algorithm \ref{alg:cap} and \eqref{prob:sqp} in Algorithm \ref{alg:const} are standard quadratic programming problems and can be solved efficiently by off-the-shelf optimization solvers. }\par
\textcolor{black}{
Once a candidate direction $d^\nu$ is computed, the Armijo line search requires again multiple lower-level solves, this time at the points $x^\nu+td^\nu$. However, this process can be sped up significantly by using warm-start techniques.
}

\section{Numerical Results}

In this section, we present \textcolor{black}{three} examples of bilevel optimization problems from the literature, discuss implementation details, and illustrate the performance of Algorithms \ref{alg:cap} and \ref{alg:const} even under relaxed assumptions. We implement the algorithms in Python and solve the lower-level problems using cvxpy (version 1.9.2) \cite{diamond2016cvxpy,agrawal2018rewriting} or pyomo (version 6.10.1) \cite{hart2011pyomo,bynum2021pyomo} with ipopt (version 3.14.20) \cite{wachter2006implementation}. All experiments are run on a laptop using an Intel i5-11357 @ 2.40 GHz processor and 8 GB of RAM.

\subsection{\textcolor{black}{Capacity Allocation}}
\textcolor{black}{We consider the problem of transporting a demand $D>0$ across a network of $n$ parallel links connecting a single origin-destination pair. The upper-level planner chooses the capacity allocation $x=(x_1,\hdots,x_n)\in X$, where $X$ represents the budget and design constraints, and its objective is to minimize the total travel time of the system. Given $x$, the lower-level problem assigns a load $y_i$ to each link $i$. The travel time on link $i$ is described by $c_i(y_i,x_i)$, and the load on that link cannot exceed the available capacity $\theta_ix_i$, $\theta_i>0$. The lower-level assignment is determined by a capacitated Wardrop equilibrium, obtained by minimizing the corresponding Beckmann potential. The resulting problem is
\begin{equation}\label{prob:capacityallocation}
    \underset{x\in X,y\in \reals^n}{\text{minimize}} ~ \sum_{i=1}^n y_ic_i(y_i,x_i) + \Psi(x)~\text{ s.t. }~    y\in \nargmin_{y'\geq 0} \bigg\{\sum_{i=1}^n \int_0^{y'_i}c_i(t,x_i)dt ~\bigg| ~ \sum_{i=1}^ny'_i=D, ~y'_i\leq \theta_ix_i\bigg\},
\end{equation}
where $\Psi$ may be a function that is zero for every feasible $x$. Although the lower-level problem is infeasible when the installed capacity is insufficient to accommodate the prescribed demand $D$, its regularized counterpart remains well defined and therefore extends the hyper-objective beyond its original domain. This extension may assign artificially favorable objective values to insufficient capacity allocations because the regularized response may not satisfy the demand constraint exactly. To prevent this behavior, we let 
\begin{equation*}
    \Psi(x)= \kappa \max\bigg\{0,D-\sum_{i=1}^n \theta_i x_i\bigg\},
\end{equation*}
with $\kappa > 0$. This modification provides an equivalent reformulation of the actual problem because $\Psi(x)$ vanishes at every $x$ at which the lower-level problem is feasible. It provides an extension of the upper-level objective that is well suited in regions where the lower-level problem is infeasible. 
}

\textcolor{black}{
Under an appropriate choice of $X$ and $c_i$, the problem satisfies the assumptions of Theorem \ref{thm:const}. Let 
\begin{equation*}
    X=\bigg\{ x\in \reals^n~\bigg|~ x_i\geq \underline{x}>0 \, \forall i, \ \sum_{i=1}^n x_i \leq C\bigg\}
\end{equation*}
for some budget $C>0$. Suppose that for all $x\in X$ and $i\in \{1,\dots,n\}$, $\sum_{j\neq i} \theta_ix_i<D$. That is, $C$ is small enough that all links must be used at every feasible capacity $x$. We consider BPR cost functions of the form
\begin{equation*}
    c_i(y,x) = \tau_i \big(1+a(y/x)^b\big),
\end{equation*}
where $\tau_i>0$ is the free-flow time on link $i$ and $a=0.15$, $b=4$ as commonly used in the literature \cite{BPR1964}. With  
\begin{equation*}
\begin{aligned}
g(x,y)=\sum_{i=1}^n \int_0^{y_i}c_i(t,x_i)dt, ~~ g_0(x,y)=\sum_{i=1}^n y_i-D,~~ \  g_j(x,y)=-y_i, \ i\in \{1,\dots,n\}, \\
g_j(x,y)=y_i-\theta_ix_i, \ i\in \{n+1,\dots,2n\},~~  h(z_0,z_1,\dots,z_{2n})=\iota_{\{0\}}(z_0)+\sum_{i=1}^{2n}\iota_{(-\infty,0]}(z_i),
\end{aligned}
\end{equation*}
the universal assumptions laid out after \eqref{prob:LL} are satisfied.} 

\textcolor{black}{It is also possible to show that \eqref{SSOC} is satisfied for all $x\in X$ such that $\sum_{i=1}^n \theta_i x_i >D$.
Suppose that $x\in X$ satisfies $\sum_{i=1}^n \theta_i x_i>D$, and let
$y=Y(x)$. By the assumption $\sum_{j\neq i}\theta_jx_j<D$
for every $i$, we must have $y_i>0$ for every link $i$. Indeed, if $y_i=0$, the remaining links could carry at most $\sum_{j\neq i}\theta_jx_j<D$, contradicting $\sum_jy_j=D$. We have that
\begin{equation*}
\nabla_{yy}^2g(x,y)
=
\operatorname{diag}
\left(
\tau_iab\frac{y_i^{b-1}}{x_i^b}
\right)_{i=1}^n.
\end{equation*}
Because $x_i>0$ and $y_i>0$, this matrix is positive definite. Moreover, $\sum_{i=1}^n \theta_i x_i>D$ implies that at least one capacity constraint is inactive. The active lower-level constraints therefore consist of the demand equality together with a proper subset of the capacity constraints, and so their gradients are linearly independent. Positive definiteness of $\nabla_{yy}^2g(x,y)$ and linear independence of the rows of $\nabla_y G(x,y)$ imply \eqref{SSOC}. Then, a sufficiently large $\kappa$ and a sufficiently small $\rho_0$ can drive the iterates of Algorithm \ref{alg:const} to the region where \eqref{SSOC} holds.
}

\textcolor{black}{The result presented here gives light to the scope and limitation of our method: the doubly regularized lower-level problem may produce spurious solutions in regions where \eqref{SSOC} is not satisfied, but it is sometimes possible to account for this behavior and, through reformulations as exemplified by the $\Psi$-function above, amend it. This can be particularly useful in applications where it is hard to find points in the lower-level feasible region, but there is a natural way to measure infeasibility.}

\textcolor{black}{We generate instances of problem \eqref{prob:capacityallocation} for $n=10, 50, 100$. In all cases, we generate three instances randomly. The demand $D$ is selected uniformly from $[100n, 150n]$, free-flow times uniformly from $\tau_i\in [1.0, 5.0]$ and maximum capacity multipliers uniformly from $\theta_i\in [2.0,3.0]$ for every link $i$. The budget $C>0$ is randomly generated uniformly from those satisfying our assumptions. We take $\underline{x}=10.0$, so that the point $x_i=\underline{x}$ for all $i$ lies in $X$ but produces an infeasible lower-level problem. In order to illustrate the fact that our method does not require lower-level feasibility except at a cluster point, we start Algorithm \ref{alg:const} at $x_i=\underline{x}$ for all $i$ and 9 other random points in $X$ that produce an infeasible lower-level problem and 10 random points in $X$ that produce a feasible lower-level problem. We use $\alpha^0=\beta^0=10^{-3}$, $\epsilon^0=\rho^0=10^{-2}$, $\theta^0=10^{-1}$, $\eta=10^3$, $\delta=10^{-1}$, $\gamma=0.5$, $\mu_i=0.5$ for all $i$, and we set the maximum number of iterations to 100. We do not set $\alpha^{opt}$, $\beta^{{\rm opt}}$ and $\epsilon^{{\rm opt}}$ in this experiment and let the algorithm run until the maximum number of iterations is reached instead. We deviate from the theory and use a fixed number of gradient samples $N_{sam}=5$; we compare in Subsection 6.3 the performance of the algorithm with this implementation and with one that uses the theoretically prescribed number of samples.} 

\textcolor{black}{At each iteration of the algorithm, we find $Y_{\alpha^\nu,\beta^\nu}$ by solving the problem in \eqref{prob:llcomputation}, which is convex due to $G$ being linear, with pyomo using ipopt as a solver. We use a tolerance of $10^{-8}$ for the active set identification. Since \eqref{prob:capacityallocation} fits the framework of Example \ref{ex:example1}, we do not compute $B$ by solving a linear system and use the formula in \eqref{eq:jacnlp} instead.}
\begin{table}[ht]
\color{black}
    \centering
    \begin{tabular}{c|c|c|c|c}
         & Time (s) & Obj. value & Infeasibility & Random Search \\
         \hline 
         $n=10$ & 52.0 $\pm$ 2.4 & 6006 $\pm$ 120 & 0.057$\pm$ 0.038 & 7344 \\
         
         $n=50$ & 97.9 $\pm$ 10.4 & 33871$\pm$ 352 & 0.427$\pm$ 0.308 & 43665 \\
         
         $n=100$ & 135.5 $\pm$ 5.2 & 67930 $\pm$ 721 & 0.616 $\pm$ 0.309 & 87257
    \end{tabular}
    \caption{Average and standard deviation of different metrics for Algorithm \ref{alg:const} over 20 random initializations for different values of $n$ in \eqref{prob:capacityallocation}.}
    \label{tab:capacityallocation}
\end{table}

\textcolor{black}{Table \ref{tab:capacityallocation} reports the average and standard deviation for the computation time (column 2) and the upper-level objective value (column 3) for the 20 initialization points for different values of $n$. In column 4 we report the sum of the constraint violation over all the upper-level constraints. We compare in each case with the best value obtained by a random search over 1000 random points $x\in X$ such that $\sum_{i=1}^n x_i=C$ and report those values in column 5. Algorithm \ref{alg:const} outperforms random search across all problem sizes. Figure \ref{fig:capacityallocation} shows the evolution of the regularized upper-level objective value $f(x^\nu,Y_{\alpha^\nu,\beta^\nu}(x^\nu))$ for every random initialization, both at points where the lower-level problem is feasible (in blue) and infeasible (in red). After a small number of iterations, the iterates settle in the region of lower-level feasibility in both cases and the trajectories become indistinguishable from each other.}
\begin{figure}[ht]
    \centering
    \begin{minipage}{0.3\textwidth}
        \includegraphics[width=0.95\linewidth]{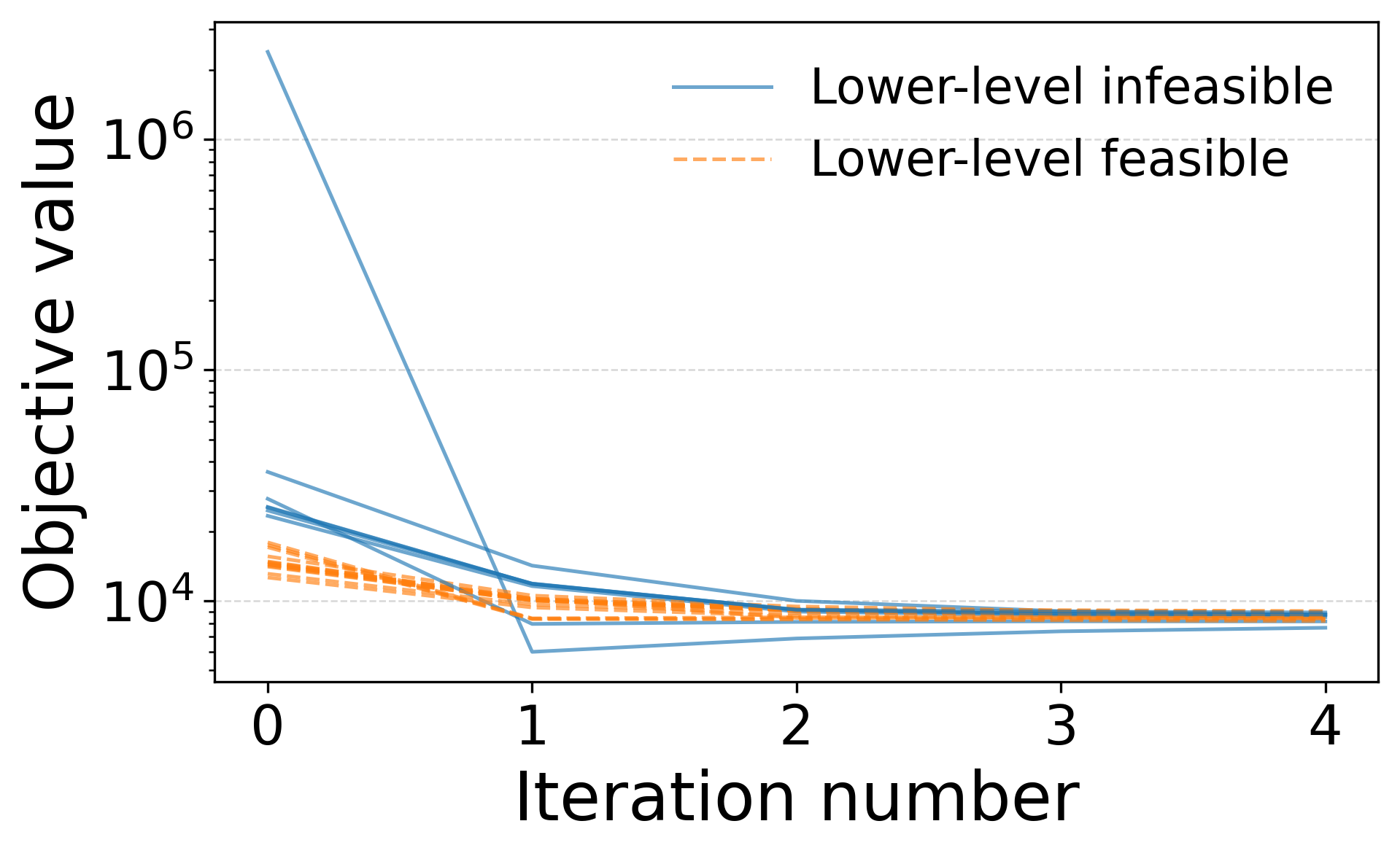}
    \end{minipage}\hfill
    \begin{minipage}{0.3\textwidth}
        \includegraphics[width=0.95\linewidth]{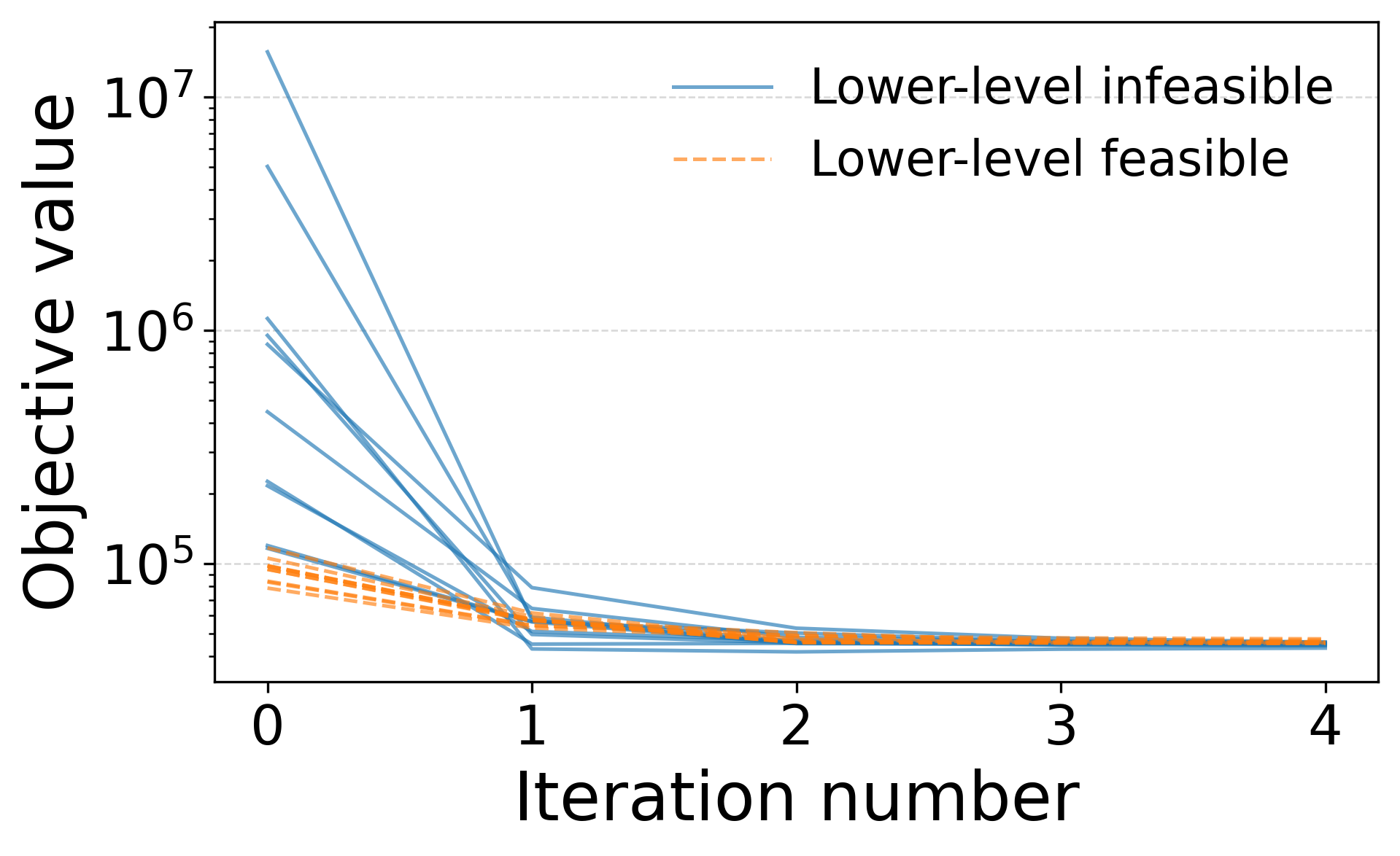}
    \end{minipage}\hfill
    \begin{minipage}{0.3\textwidth}
        \includegraphics[width=0.95\linewidth]{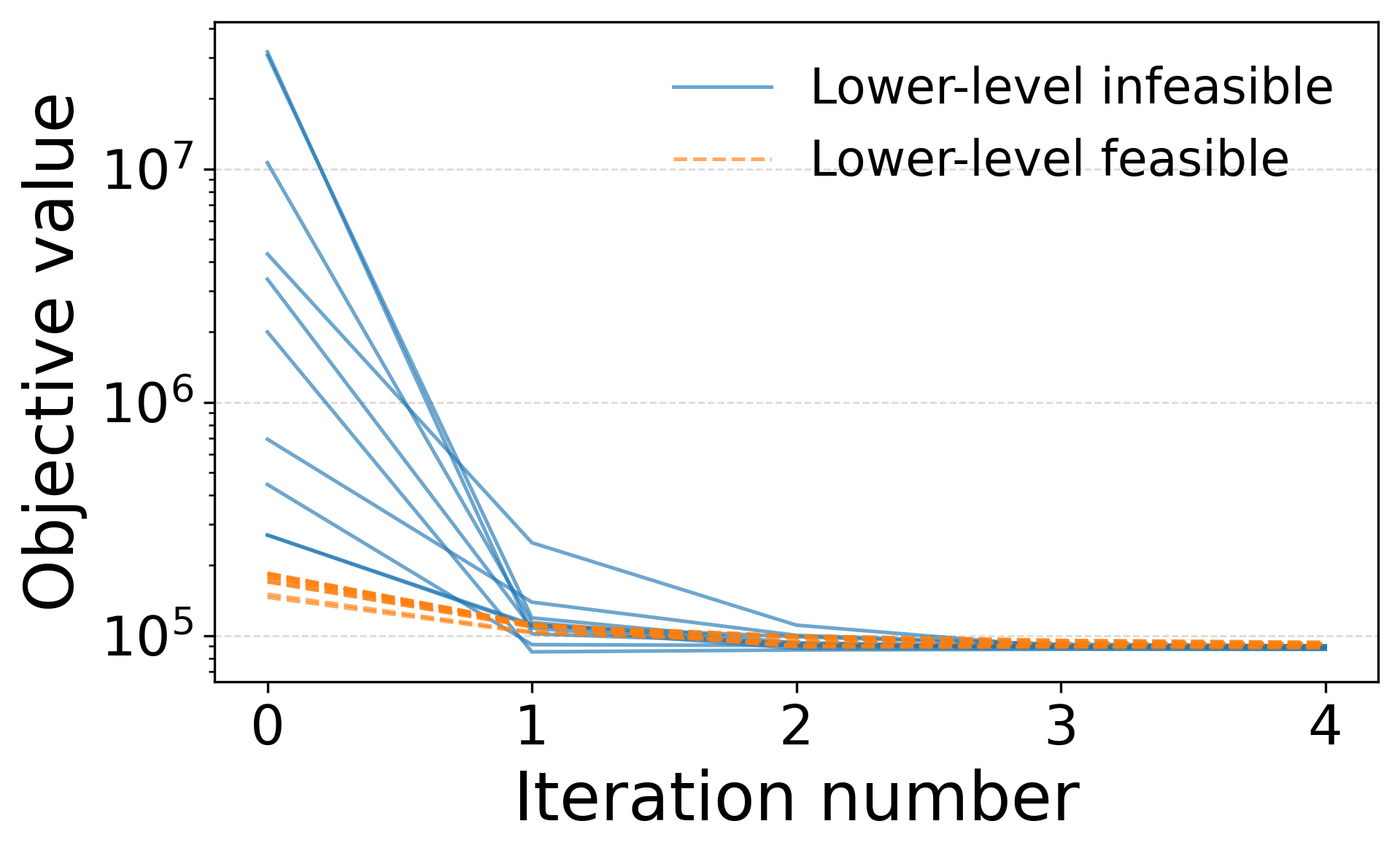}
    \end{minipage}
    \caption{Evolution of the upper-level objective in the first 5 iterations of Algorithm \ref{alg:const} for 20 random initializations and different values of $n$ ($n=10$ on the left, $n=50$ in the middle and $n=100$ on the right).}
    \label{fig:capacityallocation}
\end{figure}

\subsection{Elastic-net Regularization and Hyperparameter Tuning}

Given data $\{(a_i, b_i)\}_{i=1}^n$, with $a_i\in \reals^d$ and $b_i\in \reals$, and a partition into a training and a validation dataset (indexed by $I_{tr}$ and $I_{val}$ respectively), we are interested in finding the value of the regularization parameters $\lambda_1$ and $\lambda_2$ in an elastic-net model \cite{zou2005regularization} such that training on $I_{tr}$ achieves the minimum possible loss on $I_{val}$. This results in the problem
\begin{equation}\label{prob:elasticnet}
    \underset{\lambda_1,\lambda_2\in \reals, x\in \reals^d}{\text{minimize}} ~\frac{1}{|I_{val}|}\|\widehat A x-\widehat b\|_2^2 ~~\text{subject to }~    x\in \nargmin_{x'} \|A x'-b\|_2^2 + \exp(\lambda_1) \|x'\|_1+\tfrac{1}{2}\exp(\lambda_2)\|x'\|_2^2,
\end{equation}
where $b=(b_i)_{i\in I_{tr}}$, $\widehat{b}=(b_i)_{i\in I_{val}}$ and each row of $A\in \reals^{|I_{tr}|\times d}$ (resp. $\widehat{A}\in \reals^{|I_{val}|\times d}$) is a training (resp. validation) point $a_i^\top$. Following \cite{gao2026moreau}, we generate the rows of the matrix $A$ using a Gaussian distribution with mean 0 and covariance $\text{Cov}(a_{ij},a_{ik})=0.5^{|j-k|}$. The responses are generated from a true model $x_{true}$ with 15 components set equal to 1 and the rest equal to 0. The signal-to-noise ratio is equal to 2. We set $d=100$, $|I_{tr}|=100$, and $|I_{val}|=100$.

The problem fits the setting in \eqref{prob:UL} and \eqref{prob:LL}. The upper-level problem is unconstrained, and the nonnegativity of the regularization parameters is enforced indirectly by using an exponential reparameterization. Moreover, the upper-level objective function is $C^1$ in the lower-level variable and does not depend directly on $\lambda=(\lambda_1,\lambda_2)$. For the lower-level problem, one can set $g(\lambda, x') = \|A x'-b\|_2^2+\tfrac{1}{2} \exp(\lambda_2)\|x'\|_2^2$, 
\[
G(\lambda, x')=\big(\exp(\lambda_1)x',-\exp(\lambda_1)x'\big), ~~~~ h(z_1,z_2)=\sum_{i=1}^d\max\{z_{1i},z_{2i}\},
\]
where $z_1=(z_{11}, \dots, z_{1d})$ and $z_2=(z_{21}, \dots, z_{2d})$. Thus, the universal assumptions laid out after \eqref{prob:LL} hold. Moreover, $g$ and $G$ are analytic. \textcolor{black}{Therefore, since strong convexity of $g$ implies that the lower-level mapping is single-valued for all $\lambda$, the cluster points of the iterates of Algorithm \ref{alg:cap} satisfy the optimality condition in Proposition \ref{prop:weak}}. On the other hand, the matrix $\nabla_x G(\lambda,x)$ does not have full rank even if we only consider the active pieces as long as $x_i=0$ for some component $i$. In fact, numerical tests indicate that \eqref{SSOC} is not satisfied at any point and the matrix in \eqref{for:gradS} is never invertible. Nevertheless, Algorithm \ref{alg:cap} performs well, \textcolor{black}{which seems to indicate that the doubly regularized lower-level problem provides an appropriate approximation of the actual problem even under the weaker stationarity notion from Proposition \ref{prop:weak}.} 

We use $\alpha^0=\beta^0=\epsilon^0=\rho^0=1$ and $\alpha^{{\rm opt}}=\beta^{{\rm opt}}=\epsilon^{{\rm opt}}=10^{-5}$, $N_{sam}=5$, $\gamma=0.5$, $\delta=10^{-4}$, $\mu_i=0.5$ for all $i$, and the maximum number of iterations is 1000. \textcolor{black}{Since $G$ is linear in $x'$, we find lower-level solutions by solving the problem in \eqref{prob:llcomputation}, which is convex, with cvxpy. We use a tolerance of $10^{-10}$ for the active set identification and compute $B$ in the Jacobian calculation by solving \eqref{sys:khri} analytically, noting that a solution $w$ of \eqref{sys:khri} has $w_i\neq 0$ only if $z_{1i}=z_{2i}$ and both pieces are active, which implies $x_i=-x_i=0$. } Additionally, we generate a test dataset with $|I_{test}|=300$ and report the average error of the model over this set (column 4) to illustrate its generalization capabilities. 

Table \ref{tab:elasticnet} reports the average computing times (column 2) across 20 randomly generated instances for different methods including Algorithm \ref{alg:cap}. In column 3, we report the average validation error obtained by the elastic-net model $x$ trained with the optimal solution $\lambda^{{\rm opt}}$ produced by the various methods, i.e.,
\begin{equation*}
    x\in \nargmin_{x'\in \reals^d}\|A x'-b\|_2^2 + \exp(\lambda^{{\rm opt}}_1) \|x'\|_1+\tfrac{1}{2}\exp(\lambda^{{\rm opt}}_2)\|x'\|_2^2.
\end{equation*}

\begin{table}[ht]
    \centering
    \begin{tabular}{c|c|c|c}
        Method & Time (s) & Validation error & Test error \\
        \hline
        Algorithm \ref{alg:cap} & 4.41 & 7.46 & 7.66 \\
        Algorithm \ref{alg:cap}R & 11.35 & 8.56 & 8.73  \\
        Algorithm \ref{alg:cap}E & 0.85 & 7.80 & 7.98 \\
        \textcolor{black}{Bundle Algorithm} & \textcolor{black}{7.14} & \textcolor{black}{7.33} & \textcolor{black}{7.69}
    \end{tabular}
    \caption{Average results for 20 instances of elastic-net problem \eqref{prob:elasticnet} across algorithms.}
    \label{tab:elasticnet}
\end{table}

While the first row of Table \ref{tab:elasticnet} reports the performance of Algorithm \ref{alg:cap} applied to \eqref{prob:elasticnet}, 
we can also reformulate the lower-level problem using inequality constraints by replacing $x'$ by its positive and negative part:
\begin{equation}\label{prob:enref}
    \underset{x^+,x^-\geq 0}{\text{minimize}} ~\|A (x^+-x^-)-b\|_2^2 + \exp(\lambda_1) (x^++x^-)+\tfrac{1}{2}\exp(\lambda_2)\|x^+-x^-\|_2^2.
\end{equation}
This reformulation has twice as many lower-level variables. Algorithm \ref{alg:cap}R in the third row in Table \ref{tab:elasticnet} refers to the use of Algorithm \ref{alg:cap} to solve the reformulation in \eqref{prob:enref}. As expected the larger number of variables under consideration for Algorithm \ref{alg:cap}R produces longer computing times compared to Algorithm \ref{alg:cap} (column 2 in Table \ref{tab:elasticnet}), with a difference of more than twice the time of Algorithm \ref{alg:cap}. The validation error is also smaller for Algorithm \ref{alg:cap}, as the reformulation appeared to be more sensitive to numerical instabilities under small $\alpha$ and $\beta$. The test error shows good generalization capabilities for both versions.

The fourth row in Table \ref{tab:elasticnet} (Algorithm \ref{alg:cap}E) reports results for Algorithm \ref{alg:cap} when stopping the first time we reach a point that satisfies $|w^\nu|\leq \eta^{\nu}$. The resulting validation error is only 5\% higher than the validation error produced by Algorithm \ref{alg:cap} when stopping upon reaching convergence, but the computation time is 5 times smaller. 

\textcolor{black}{We compare the performance of the variants of Algorithm \ref{alg:cap} with the bundle method in \cite{dempe2000bundle} applied to the reformulation in \eqref{prob:enref}. We implemented the bundle method by solving the regularized lower-level problem associated with \eqref{prob:enref} with cvxpy and solving the system in the Jacobian calculation with numpy. We use the same tolerances as for Algorithm \ref{alg:cap}R for the lower-level solves, the Jacobian calculation, and the active system computation. In the notation of \cite{dempe2000bundle}, we set $\alpha^0=1.0$, $\epsilon=10^{-5}$, $\varrho^0=1.0$, $\delta^0=1.0$, $\kappa=0.9$, $\omega=1.0$, $m_L=0.1$, $u^0=1.0$, and we increase the value of $u$ by a factor of 2 when necessary to compute the direction in the trust region. We set a maximum of 100 iterations. This method achieves similar validation and test error to Algorithm \ref{alg:cap}, but requires more computation time. The difference in computing time appears to be driven by the larger size of the lower-level problem and Jacobian system induced by the lower-level reformulation, as can be seen by the fact that despite the bundle method being slower than Algorithm \ref{alg:cap}, it is faster than Algorithm \ref{alg:cap}R when both methods are applied to \eqref{prob:enref}.}

\begin{figure}[ht]
    \centering
    \includegraphics[width=0.6\linewidth]{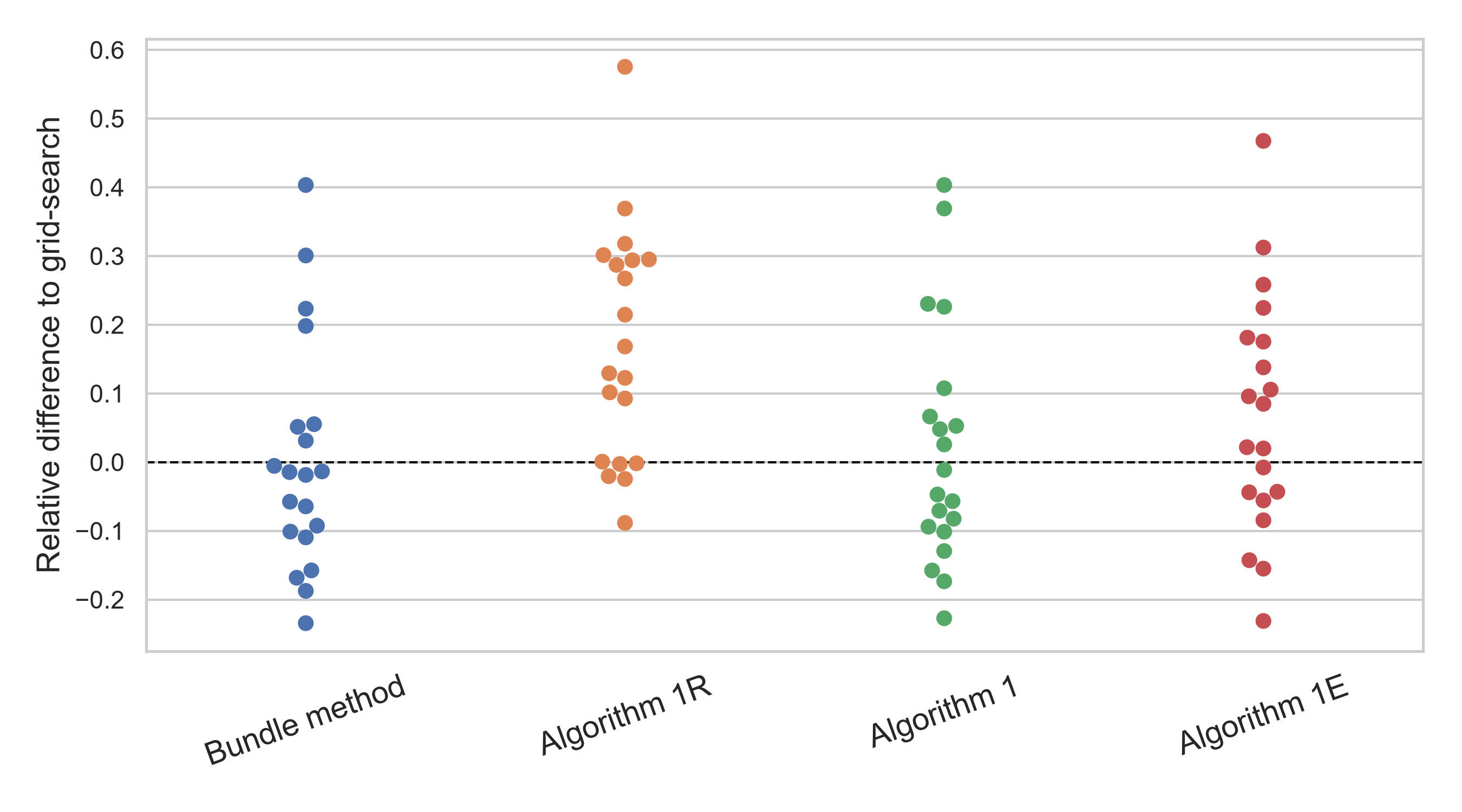}
    \caption{Relative validation error difference for different algorithms applied to \eqref{prob:elasticnet} with respect to a grid search baseline (illustrated by a dashed line) across 20 randomly generated instances.}
    \label{fig:relative_diff}
\end{figure}

\textcolor{black}{Elaborating on the setting from Table \ref{tab:elasticnet}, Figure \ref{fig:relative_diff} reports the distribution of the relative validation errors for the different methods and compared them to the results from a baseline 21-by-21 grid search varying $\lambda_1$ and $\lambda_2$ from $-5$ to $5$. It shows that Algorithm \ref{alg:cap} can often find improvements over grid search even in this low-dimensional example, but occasionally stops at stationary points with more than 20\% worse performance. The situation is similar for the bundle method and Algorithm \ref{alg:cap}, whereas Algorithm \ref{alg:cap}R exhibit more positive differences and heavier upper tails.}

\subsection{Data Poisoning}

For a third example, we consider the attacker problem in a data poisoning setting. We follow \cite{jagielski2018manipulating} and adopt the attacker's goal of maximizing the loss of a regression model on a known but untainted validation set by manipulating the training data producing the model. \textcolor{black}{This approach seeks to generate a large amount of generic misclassifications, eventually leading to a denial of service, and has been studied in the literature \cite{munoz2017towards,jagielski2018manipulating,mei2015using}.} Unlike in \cite{jagielski2018manipulating}, we assume that the poisoned data points are not completely fabricated by the attacker, but come instead from real data points whose responses can be perturbed within a budget. We also let the attacker have access to every training data point, but the generalization to the case in which the attacker can only access a subset of the training dataset is straightforward. The problem formulation is
\begin{equation}\label{prob:datapoisoning}
    \underset{\|w\|_2^2\leq c, x\in \reals^d}{\text{maximize}} ~\frac{1}{|I_{val}|}\|\widehat A x-\widehat b\|_2^2 ~\text{subject to }~    x\in \nargmin_{x'} \|A x'-b-w\|_2^2 + \lambda_1 \|x'\|_1+\tfrac{1}{2}\lambda_2\|x'\|_2^2,
\end{equation}
where $A$, $\widehat{A}$, $b$, and $\widehat{b}$ are defined in the same way as in \eqref{prob:elasticnet}. The upper-level variable $w_i$ represents the perturbation injected by the attacker at data point $i$, the parameter $c\in \reals$ is the total budget of the attack, and $\lambda_1$ and $\lambda_2$ are considered to be fixed and known. The upper-level problem is constrained and its objective function is lower bounded. Since one can take
\begin{equation*}
g(w, x')= \|A x'-b-w\|_2^2+\tfrac{1}{2}\lambda_2\|x'\|_2^2, ~~ ~  G(w, x')=(x',-x'), ~~ ~  h(z_1,z_2)=\lambda_1\sum_{i=1}^d\max\{z_{1i},z_{2i}\},
\end{equation*}
the universal assumptions laid out after \eqref{prob:LL} hold. Moreover, $g$ and $G$ are analytic. \textcolor{black}{We observe empirically that \eqref{SSOC} is usually satisfied at the minimizers of the lower-level problem of \eqref{prob:datapoisoning} when $w$ is sufficiently large. Therefore, the limit assumption in Theorem \ref{thm:const} seems to hold according to this empirical evidence. Furthermore, the constraint system satisfies the Mangasarian-Fromovitz constraint qualification at every $x$, and therefore Proposition \ref{prop:calmness} guarantees the existence of an exact penalty for problem \eqref{prob:datapoisoning}.} 
We generate the datasets using the same process described for elastic-net regularization in Section 6.1, setting $d=50$, $|I_{val}|=100$, and $|I_{test}|=300$, and vary $|I_{tr}|\in \{50,100,150,200,250\}$. We fix $\lambda_1=\exp(3)$ and $\lambda_2=\exp(2)$, set $c=100$, which corresponds to roughly 1-5\% of the average norm squared of $b_{tr}$, and run Algorithm \ref{alg:const} with $\alpha^0=\beta^0=0.01$, $\epsilon^0=\rho^0=0.1$, $\alpha^{{\rm opt}}=\beta^{{\rm opt}}=10^{-6}$, $\epsilon^{{\rm opt}}=10^{-4}$, $\gamma=0.5$, $\eta=30$, $\delta=10^{-4}$, $\theta^0=10^{-4}$, $\mu_{\alpha}=\mu_{\beta}=0.2$, and $\mu_{\epsilon}=0.8$. \textcolor{black}{We solve the lower-level problem and compute $B$ and the Jacobian matrix following the same procedure as in Section 6.1. The tolerance for active set identification is $10^{-10}$.} We run the algorithm with both the minimum number of samples required in Theorem \ref{thm:const}, $N_{sam}=|I_{tr}|+1$, and a small fixed number of samples, $N_{sam}=5$. The results for the case $|I_{tr}|=100$ are reported in Table \ref{tab:datapoisoning}.

\begin{table}[ht]
    \centering
    \begin{tabular}{c|c | c | c | c}
        Method & Time (s) & Validation error & Test error & Infeasibility \\
        \hline
        Algorithm \ref{alg:const}, $N_{sam}=101$ & 34.3 & 16.8 & 13.4 & $2.7\cdot 10^{-3}$\\
        Algorithm \ref{alg:const}, $N_{sam}=5$ & 2.9 & 17.0 & 13.3 & $2.8\cdot 10^{-3}$\\
        Algorithm \ref{alg:const}R, $N_{sam}=101$ & 47.7 & 16.0 & 12.7 & $3.5\cdot 10^{-3}$ \\
        Algorithm \ref{alg:const}R, $N_{sam}=5$ & 4.1 & 16.1 & 12.8 & $3.5\cdot 10^{-3}$ \\
        Random Search & 26.2 & 11.3 & 10.7 & - 
    \end{tabular}
    \caption{Average results for 5 instances of data poisoning problem in \eqref{prob:datapoisoning} with $|I_{tr}|=100$ and different methods to solve it.}
    \label{tab:datapoisoning}
\end{table}

Column 2 in Table \ref{tab:datapoisoning} reports the average computation time, while column 3 reports the average normalized validation error defined as in Table \ref{tab:elasticnet}. We introduce a reformulation of \eqref{prob:datapoisoning} based on replacing $x'$ by its positive and negative part, in a similar way as in \eqref{prob:enref}, and report the results of applying Algorithm \ref{alg:const} to this reformulation in the rows of Table \ref{tab:datapoisoning} corresponding to Algorithm \ref{alg:const}R. The larger amount of variables in this reformulation produces computation times 32\% higher compared to Algorithm \ref{alg:const}.

Both versions achieve a validation error that is more than double the error in the clean dataset, where the validation error is 8.48, even though the maximum allowed perturbation is small (3\% of the magnitude of the response vector). Furthermore, the last row of Table \ref{tab:datapoisoning} reports the results obtained with a random search over 2000 points in the set of feasible $w$; both Algorithm \ref{alg:const} and Algorithm \ref{alg:const}R significantly outperform it. 

Although the attacker is principally interested in the error on the validation set, we report in column 4 of Table \ref{tab:datapoisoning} the error on a test dataset to illustrate that the poisoning remains effective even if the defender were to use the model on an unknown dataset. In column 5, we report the magnitude of the upper-level constraint violation, which is small for both versions of the algorithm \textcolor{black}{regardless of the choice of $N_{sam}$}.

\textcolor{black}{Interestingly, using only $N_{sam}=5$ does not cause a decrease in performance; in fact, from Table \ref{tab:datapoisoning} this choice occasionally yields a lower validation error than using $N_{sam}=n_{tr}+1$, the theoretically minimum number of samples. This behavior is consistent across all dataset sizes and for both Algorithm \ref{alg:const} and Algorithm \ref{alg:const}R, with the relative difference in validation error between the two sampling choices remaining under 3\% in all cases. Figure \ref{fig:datapoisoning} shows the average computing time of solving an instance when the training dataset size increases. When $N_{sam}=n_{tr}+1$, the computing time grows approximately linearly with $n_{tr}$ for both formulations, although it is substantially lower for Algorithm \ref{alg:cap} than for Algorithm \ref{alg:cap}R. By contrast, with the fixed choice $N_{sam}=5$, the computing time remains small across all tested dataset sizes.}

\begin{figure}
    \centering
    \includegraphics[width=0.6\linewidth]{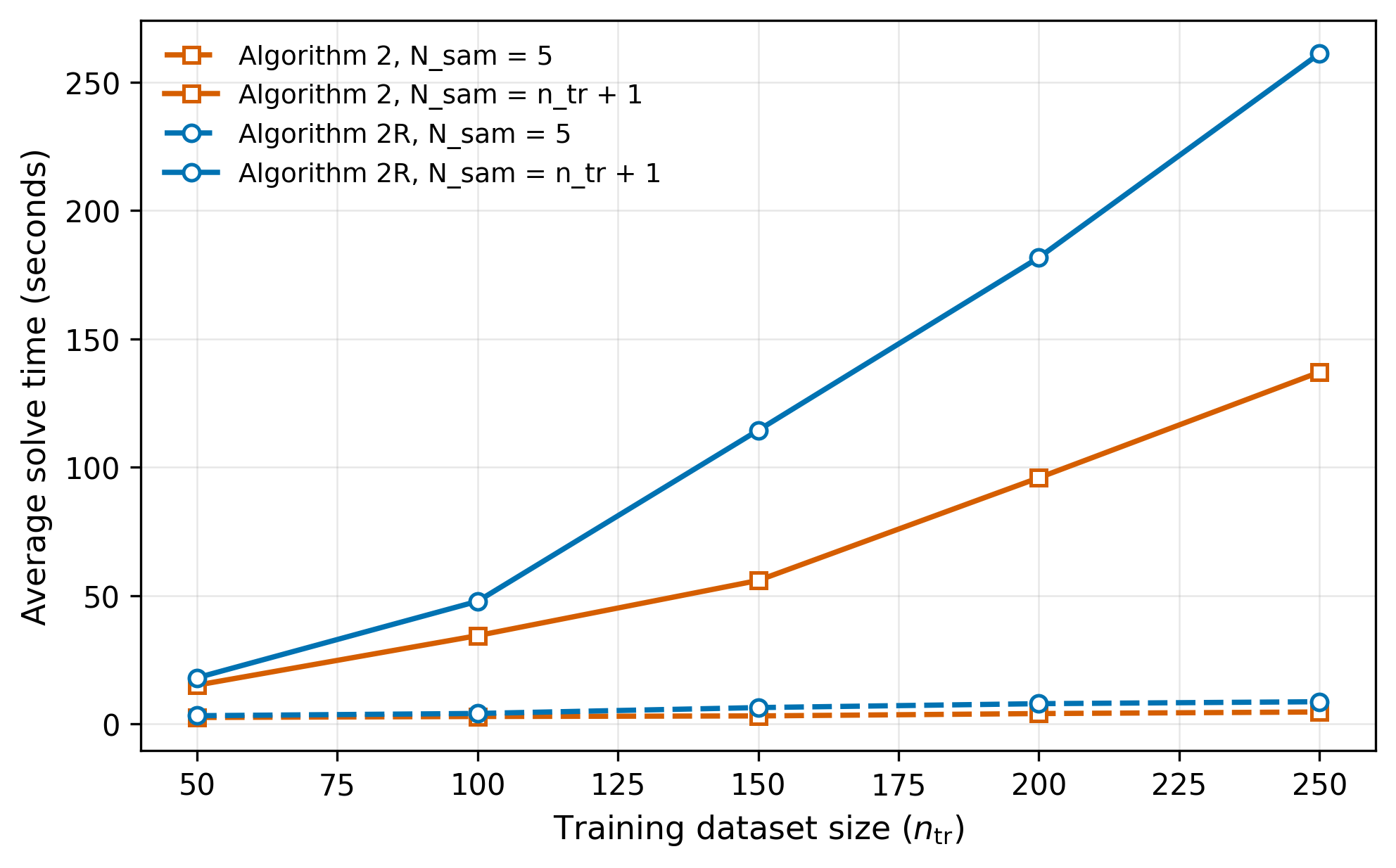}
    \caption{Average computing time across different training dataset sizes $n_{tr}$ for different versions of Algorithm \ref{alg:const} and values of $N_{sam}$.}
    \label{fig:datapoisoning}
\end{figure}

\section{Conclusions}
In this article, we introduced a novel double regularization for a class of bilevel optimization problems whose lower level can be expressed as a convex extended nonlinear program. We showed that the regularized lower-level problem enjoys enhanced regularity properties pertaining to its primal-dual solution mapping, and we leveraged these properties to construct gradient sampling algorithms for both unconstrained and constrained upper-level problems. The unconstrained algorithm computes approximately stationary points of the regularized hyper-objective in terms of its Goldstein subdifferential, while the constrained algorithm combines sampled hyper-gradients with a sequential quadratic programming framework to address upper-level constraints. We established conditions under which the iterates of these algorithms converge subsequentially to a Clarke stationary point of the actual problem in the unconstrained case and to a Curtis-Overton stationary point of the actual problem in the presence of upper-level constraints. \textcolor{black}{We also provided a weaker asymptotic optimality result under milder assumptions, which in particular do not require local single-valuedness of the lower-level solution mapping.} We showed through numerical experiments the value of the regularization and the flexible problem structure. 

Obtaining quantitative results about the quality of the approximation given by the regularized problem and nonasymptotic convergence guarantees for the algorithms are important directions for further studies to provide more robust guarantees for these methods, since all convergence results presented in this article are asymptotic.

\medskip

\noindent {\bf Acknowledgments.} This work is supported in part by the Office of Naval Research (N00014-24-1-2492, N00014-24-1-2741) and the National Science Foundation (CMMI-2432337). The authors thank Jinqi Gong for contributing with numerical results and to Lai Tian for valuable discussions and suggestions.

\bibliographystyle{plain}
\bibliography{refs}

@article{BeckBienstockSchmidtThurauf.23,
	Author = {Y. Beck and D. Bienstock and M. Schmidt and J. Th\"{u}rauf},
	Journal = {Journal of Optimization Theory and Applications},
	Volume = {198},
	Title = {On a Computationally Ill-Behaved Bilevel Problem with a Continuous and Nonconvex Lower Level},
	Pages = {428-447},
	Year = {2023}}

@book{VaAn,
	Author = {R. T. Rockafellar and R. J-B Wets},
	Edition = {3rd printing-2009},
	Publisher = {Springer},
	Series = {Grundlehren der Mathematischen Wissenschaft},
	Title = {Variational Analysis},
	Volume = {317},
	Year = {1998}}

@article{hang2024role,
  title={Role of subgradients in variational analysis of polyhedral functions},
  author={Hang, Nguyen TV and Jung, Woosuk and Sarabi, Ebrahim},
  journal={Journal of Optimization Theory and Applications},
  volume={200},
  number={3},
  pages={1160--1192},
  year={2024},
  publisher={Springer}
}

@book{royset2021optimization,
  title={An optimization primer},
  author={Royset, Johannes O and Wets, Roger J-B},
  volume={440},
  year={2021},
  publisher={Springer}
}

@article{mordukhovich2016generalized,
  title={Generalized differentiation of piecewise linear functions in second-order variational analysis},
  author={Mordukhovich, Boris S and Sarabi, M Ebrahim},
  journal={Nonlinear Analysis},
  volume={132},
  pages={240--273},
  year={2016},
  publisher={Elsevier}
}

@book{fiacco1990nonlinear,
  title={Nonlinear programming: sequential unconstrained minimization techniques},
  author={Fiacco, Anthony V and McCormick, Garth P},
  year={1990},
  publisher={SIAM}
}

@article{dempe2000bundle,
  title={A bundle algorithm applied to bilevel programming problems with non-unique lower level solutions},
  author={Dempe, Stephan},
  journal={Computational Optimization and Applications},
  volume={15},
  pages={145--166},
  year={2000},
  publisher={Springer}
}

@article{ralph1995directional,
  title={Directional derivatives of the solution of a parametric nonlinear program},
  author={Ralph, Daniel and Dempe, Stephan},
  journal={Mathematical programming},
  volume={70},
  number={1},
  pages={159--172},
  year={1995},
  publisher={Springer}
}

@book{scholtes2012introduction,
  title={Introduction to piecewise differentiable equations},
  author={Scholtes, Stefan},
  year={2012},
  publisher={Springer}
}

@incollection{burke2020gradient,
  title={Gradient sampling methods for nonsmooth optimization},
  author={Burke, James V and Curtis, Frank E and Lewis, Adrian S and Overton, Michael L and Sim{\~o}es, Lucas EA},
  booktitle={Numerical nonsmooth optimization: State of the art algorithms},
  editor={Bagirov, Adil M and Gaudioso, Manlio and Karmitsa, Napsu and M{\"a}kel{\"a}, Marko M and Taheri, Sona},
  pages={201--225},
  year={2020},
  publisher={Springer}
}

@article{curtis2012sequential,
  title={A sequential quadratic programming algorithm for nonconvex, nonsmooth constrained optimization},
  author={Curtis, Frank E and Overton, Michael L},
  journal={SIAM Journal on Optimization},
  volume={22},
  number={2},
  pages={474--500},
  year={2012},
  publisher={SIAM}
}

@article{kiwiel2007convergence,
  title={Convergence of the gradient sampling algorithm for nonsmooth nonconvex optimization},
  author={Kiwiel, Krzysztof C},
  journal={SIAM Journal on Optimization},
  volume={18},
  number={2},
  pages={379--388},
  year={2007},
  publisher={SIAM}
}

@article{bolte2009tame,
  title={Tame functions are semismooth},
  author={Bolte, J{\'e}r{\^o}me and Daniilidis, Aris and Lewis, Adrian},
  journal={Mathematical Programming},
  volume={117},
  number={1},
  pages={5--19},
  year={2009},
  publisher={Springer}
}

@article{bierstone1988semianalytic,
  title={Semianalytic and subanalytic sets},
  author={Bierstone, Edward and Milman, Pierre D},
  journal={Publications Math{\'e}matiques de l'IH{\'E}S},
  volume={67},
  pages={5--42},
  year={1988}
}

@article{zou2005regularization,
  title={Regularization and variable selection via the elastic net},
  author={Zou, Hui and Hastie, Trevor},
  journal={Journal of the Royal Statistical Society Series B: Statistical Methodology},
  volume={67},
  number={2},
  pages={301--320},
  year={2005},
  publisher={Oxford University Press}
}

@incollection{rockafellar1999extended,
  title={Extended nonlinear programming},
  author={Rockafellar, R Tyrrell},
  booktitle={Nonlinear Optimization and Related Topics},
  pages={381--399},
  year={1999},
  publisher={Springer},
  editor={Pillo, Gianni and Giannessi, Franco}
}

@book{dempe2002foundations,
  title={Foundations of bilevel programming},
  author={Dempe, Stephan},
  year={2002},
  publisher={Springer}
}

@inproceedings{khanduri2023linearly,
  title={Linearly constrained bilevel optimization: A smoothed implicit gradient approach},
  author={Khanduri, Prashant and Tsaknakis, Ioannis and Zhang, Yihua and Liu, Jia and Liu, Sijia and Zhang, Jiawei and Hong, Mingyi},
  booktitle={International Conference on Machine Learning},
  pages={16291--16325},
  year={2023},
  organization={PMLR}
}

@article{khanduri2025doubly,
  title={A Doubly Stochastically Perturbed Algorithm for Linearly Constrained Bilevel Optimization},
  author={Khanduri, Prashant and Tsaknakis, Ioannis and Zhang, Yihua and Liu, Sijia and Hong, Mingyi},
  journal={Preprint arXiv:2504.04545},
  volume={},
  year={2025}
}

@inproceedings{ji2021bilevel,
  title={Bilevel optimization: Convergence analysis and enhanced design},
  author={Ji, Kaiyi and Yang, Junjie and Liang, Yingbin},
  booktitle={International Conference on Machine Learning},
  pages={4882--4892},
  year={2021},
  organization={PMLR}
}

@article{ghadimi2018approximation,
  title={Approximation methods for bilevel programming},
  author={Ghadimi, Saeed and Wang, Mengdi},
  journal={Preprint arXiv:1802.02246},
  volume={},
  year={2018}
}

@inproceedings{chen2022single,
  title={A single-timescale method for stochastic bilevel optimization},
  author={Chen, Tianyi and Sun, Yuejiao and Xiao, Quan and Yin, Wotao},
  booktitle={International Conference on Artificial Intelligence and Statistics},
  pages={2466--2488},
  year={2022},
  organization={PMLR}
}

@inproceedings{liu2022bome,
 author = {Liu, Bo and Ye, Mao and Wright, Stephen and Stone, Peter and Liu, Qiang},
 booktitle = {Advances in Neural Information Processing Systems},
 pages = {17248--17262},
 publisher = {Curran Associates, Inc.},
 title = {BOME! Bilevel Optimization Made Easy: A Simple First-Order Approach},
 volume = {35},
 year = {2022}
}

@inproceedings{chen2024finding,
  title={On finding small hyper-gradients in bilevel optimization: Hardness results and improved analysis},
  author={Chen, Lesi and Xu, Jing and Zhang, Jingzhao},
  booktitle={The Thirty Seventh Annual Conference on Learning Theory},
  pages={947--980},
  year={2024},
  organization={PMLR}
}

@article{liu2024moreau,
  title={Moreau envelope for nonconvex bi-level optimization: A single-loop and hessian-free solution strategy},
  author={Liu, Risheng and Liu, Zhu and Yao, Wei and Zeng, Shangzhi and Zhang, Jin},
  journal={Preprint arXiv:2405.09927},
  volume={},
  year={2024}
}

@article{nghia2025geometric,
  title={Geometric characterizations of \text{L}ipschitz stability for convex optimization problems},
  author={Nghia, Tran TA},
  journal={SIAM Journal on Optimization},
  volume={35},
  number={2},
  pages={927--958},
  year={2025},
  publisher={SIAM}
}

@article{benko2024primal,
  title={Primal--dual stability in local optimality},
  author={Benko, Mat{\'u}{\v{s}} and Rockafellar, R Tyrrell},
  journal={Journal of Optimization Theory and Applications},
  volume={203},
  number={2},
  pages={1325--1354},
  year={2024},
  publisher={Springer}
}

@article{kunisch2013bilevel,
  title={A bilevel optimization approach for parameter learning in variational models},
  author={Kunisch, Karl and Pock, Thomas},
  journal={SIAM Journal on Imaging Sciences},
  volume={6},
  number={2},
  pages={938--983},
  year={2013},
  publisher={SIAM}
}

@article{alcantara2025unified,
  title={Unified smoothing approach for best hyperparameter selection problem using a bilevel optimization strategy},
  author={Alcantara, Jan Harold and Nguyen, Chieu Thanh and Okuno, Takayuki and Takeda, Akiko and Chen, Jein-Shan},
  journal={Mathematical Programming},
  volume={212},
  number={1},
  pages={479--518},
  year={2025},
  publisher={Springer}
}

@article{meng2001equivalent,
  title={An equivalent continuously differentiable model and a locally convergent algorithm for the continuous network design problem},
  author={Meng, Qiang and Yang, Hai and Bell, Michael GH},
  journal={Transportation Research Part B: Methodological},
  volume={35},
  number={1},
  pages={83--105},
  year={2001},
  publisher={Elsevier}
}

@article{zhu2023bilevel,
  title={A bilevel optimization model for the newsvendor problem with the focus theory of choice},
  author={Zhu, Xide and Li, Kevin W and Guo, Peijun},
  journal={4OR},
  volume={21},
  number={3},
  pages={471--489},
  year={2023},
  publisher={Springer}
}

@inproceedings{yang2019provably,
 author = {Yang, Zhuoran and Chen, Yongxin and Hong, Mingyi and Wang, Zhaoran},
 booktitle = {Advances in Neural Information Processing Systems},
 pages = {},
 publisher = {Curran Associates, Inc.},
 title = {Provably Global Convergence of Actor-Critic: A Case for Linear Quadratic Regulator with Ergodic Cost},
 volume = {32},
 year = {2019}
}

@inproceedings{jiang2021learning,
  title={Learning to defend by learning to attack},
  author={Jiang, Haoming and Chen, Zhehui and Shi, Yuyang and Dai, Bo and Zhao, Tuo},
  booktitle={International Conference on Artificial Intelligence and Statistics},
  pages={577--585},
  year={2021},
  organization={PMLR}
}

@article{burke2005robust,
  title={A robust gradient sampling algorithm for nonsmooth, nonconvex optimization},
  author={Burke, James V and Lewis, Adrian S and Overton, Michael L},
  journal={SIAM Journal on Optimization},
  volume={15},
  number={3},
  pages={751--779},
  year={2005},
  publisher={SIAM}
}

@book{coste2000introduction,
  title={An introduction to o-minimal geometry},
  author={Coste, Michel},
  year={2000},
  publisher={Istituti editoriali e poligrafici internazionali Pisa}
}

@inproceedings{kornowski2024first,
 author = {Kornowski, Guy and Padmanabhan, Swati and Wang, Kai and Zhang, Jimmy and Sra, Suvrit},
 booktitle = {Advances in Neural Information Processing Systems},
 pages = {141417--141460},
 publisher = {Curran Associates, Inc.},
 title = {First-Order Methods for Linearly Constrained Bilevel Optimization},
 volume = {37},
 year = {2024}
}

@inproceedings{shen2023penalty,
  title={On penalty-based bilevel gradient descent method},
  author={Shen, Han and Chen, Tianyi},
  booktitle={International Conference on Machine Learning},
  pages={30992--31015},
  year={2023},
  organization={PMLR}
}

@inproceedings{jagielski2018manipulating,
  title={Manipulating machine learning: Poisoning attacks and countermeasures for regression learning},
  author={Jagielski, Matthew and Oprea, Alina and Biggio, Battista and Liu, Chang and Nita-Rotaru, Cristina and Li, Bo},
  booktitle={2018 IEEE Symposium on Security and Privacy (SP)},
  pages={19--35},
  year={2018},
  organization={IEEE}
}

@inproceedings{zhang2022revisiting,
  title={Revisiting and advancing fast adversarial training through the lens of bi-level optimization},
  author={Zhang, Yihua and Zhang, Guanhua and Khanduri, Prashant and Hong, Mingyi and Chang, Shiyu and Liu, Sijia},
  booktitle={International Conference on Machine Learning},
  pages={26693--26712},
  year={2022},
  organization={PMLR}
}

@article{zheng2024safe,
  author={Zheng, Zhi and Gu, Shangding},
  journal={IEEE Transactions on Artificial Intelligence}, 
  title={Safe Multiagent Reinforcement Learning With Bilevel Optimization in Autonomous Driving}, 
  year={2025},
  volume={6},
  number={4},
  pages={829-842}
  }

@article{mou2019bi,
  title={A bi-level optimization formulation of priority service pricing},
  author={Mou, Yuting and Papavasiliou, Anthony and Chevalier, Philippe},
  journal={IEEE Transactions on Power Systems},
  volume={35},
  number={4},
  pages={2493--2505},
  year={2019},
  publisher={IEEE}
}

@article{zhang2009bilevel,
  title={Bilevel programming model and solution method for mixed transportation network design problem},
  author={Zhang, Haozhi and Gao, Ziyou},
  journal={Journal of Systems Science and Complexity},
  volume={22},
  number={3},
  pages={446--459},
  year={2009},
  publisher={Springer}
}

@inproceedings{chen2024lower,
  title={Lower-level duality based reformulation and majorization minimization algorithm for hyperparameter optimization},
  author={Chen, He and Xu, Haochen and Jiang, Rujun and So, Anthony Man-Cho},
  booktitle={International Conference on Artificial Intelligence and Statistics},
  pages={784--792},
  year={2024},
  organization={PMLR}
}

@inproceedings{gao2022value,
  title={Value function based difference-of-convex algorithm for bilevel hyperparameter selection problems},
  author={Gao, Lucy L and Ye, Jane and Yin, Haian and Zeng, Shangzhi and Zhang, Jin},
  booktitle={International Conference on Machine Learning},
  pages={7164--7182},
  year={2022},
  organization={PMLR}
}

@article{diamond2016cvxpy,
  author  = {Steven Diamond and Stephen Boyd},
  title   = {{CVXPY}: {A} {P}ython-embedded modeling language for convex optimization},
  journal = {Journal of Machine Learning Research},
  year    = {2016},
  volume  = {17},
  number  = {83},
  pages   = {1--5},
}

@article{agrawal2018rewriting,
  author  = {Agrawal, Akshay and Verschueren, Robin and Diamond, Steven and Boyd, Stephen},
  title   = {A rewriting system for convex optimization problems},
  journal = {Journal of Control and Decision},
  year    = {2018},
  volume  = {5},
  number  = {1},
  pages   = {42--60},
}

@article{hang2025smoothness,
  title={Smoothness of subgradient mappings and its applications in parametric optimization},
  author={Hang, Nguyen TV and Sarabi, Ebrahim},
  journal={Set-Valued and Variational Analysis},
  volume={33},
  number={4},
  pages={41},
  year={2025},
  publisher={Springer}
}

@inproceedings{kwon2023fully,
  title={A fully first-order method for stochastic bilevel optimization},
  author={Kwon, Jeongyeol and Kwon, Dohyun and Wright, Stephen and Nowak, Robert D},
  booktitle={International Conference on Machine Learning},
  pages={18083--18113},
  year={2023},
  organization={PMLR}
}

@book{dontchev2009implicit,
  title={Implicit functions and solution mappings},
  author={Dontchev, Asen L and Rockafellar, R Tyrrell},
  volume={},
  edition = {2.},
  year={2014},
  publisher={Springer}
}

@inproceedings{munoz2017towards,
  title={Towards poisoning of deep learning algorithms with back-gradient optimization},
  author={Mu{\~n}oz-Gonz{\'a}lez, Luis and Biggio, Battista and Demontis, Ambra and Paudice, Andrea and Wongrassamee, Vasin and Lupu, Emil C and Roli, Fabio},
  booktitle={Proceedings of the 10th ACM workshop on artificial intelligence and security},
  pages={27--38},
  year={2017}
}

@inproceedings{mei2015using,
  title={Using machine teaching to identify optimal training-set attacks on machine learners},
  author={Mei, Shike and Zhu, Xiaojin},
  booktitle={Proceedings of the aaai conference on artificial intelligence},
  volume={29},
  number={1},
  year={2015}
}

@article{ermoliev1995minimization,
  title={The minimization of semicontinuous functions: mollifier subgradients},
  author={Ermoliev, Yuri M and Norkin, Vladimir I and Wets, Roger JB},
  journal={Siam journal on control and optimization},
  volume={33},
  number={1},
  pages={149--167},
  year={1995},
  publisher={SIAM}
}

@article{andrieu2024gradient,
  title={Gradient-free optimization via integration},
  author={Andrieu, Christophe and Chopin, Nicolas and Fincato, Ettore and Gerber, Mathieu},
  journal={arXiv preprint arXiv:2408.00888},
  year={2024}
}

@book{BPR1964,
  author       = {{Bureau of Public Roads}},
  title        = {Traffic Assignment Manual},
  year         = {1964},
  publisher    = {U.S. Department of Commerce, Urban Planning Division},
  address      = {Washington, D.C.}
}

@article{burke1991exact,
  title={An exact penalization viewpoint of constrained optimization},
  author={Burke, James V},
  journal={SIAM Journal on control and optimization},
  volume={29},
  number={4},
  pages={968--998},
  year={1991},
  publisher={SIAM}
}

@article{mordukhovich2017stability,
  title={Stability analysis for composite optimization problems and parametric variational systems},
  author={Mordukhovich, Boris S and Sarabi, ME},
  journal={Journal of Optimization Theory and Applications},
  volume={172},
  number={2},
  pages={554--577},
  year={2017},
  publisher={Springer}
}

@inproceedings{kwon2024penalty,
  title={On penalty methods for nonconvex bilevel optimization and first-order stochastic approximation},
  author={Kwon, Jeongyeol and Kwon, Dohyun and Wright, Stephen and Nowak, Robert},
  booktitle={International Conference on Learning Representations},
  volume={2024},
  pages={47473--47518},
  year={2024}
}

@article{lu2026solving,
  title={Solving bilevel optimization via sequential minimax optimization},
  author={Lu, Zhaosong and Mei, Sanyou},
  journal={Mathematics of Operations Research},
  year={2026},
  publisher={INFORMS}
}

@article{lu2024first,
  title={First-order penalty methods for bilevel optimization},
  author={Lu, Zhaosong and Mei, Sanyou},
  journal={SIAM Journal on Optimization},
  volume={34},
  number={2},
  pages={1937--1969},
  year={2024},
  publisher={SIAM}
}

@book{bynum2021pyomo, title={Pyomo--optimization modeling in python}, author={Bynum, Michael L. and Hackebeil, Gabriel A. and Hart, William E. and Laird, Carl D. and Nicholson, Bethany L. and Siirola, John D. and Watson, Jean-Paul and Woodruff, David L.}, edition={Third}, volume={67}, year={2021}, publisher={Springer Science \& Business Media} }

@article{hart2011pyomo, title={Pyomo: modeling and solving mathematical programs in Python}, author={Hart, William E and Watson, Jean-Paul and Woodruff, David L}, journal={Mathematical Programming Computation}, volume={3}, number={3}, pages={219--260}, year={2011}, publisher={Springer} }

@article{chen2026set,
  title={Set smoothness unlocks clarke hyper-stationarity in bilevel optimization},
  author={Chen, He and Li, Jiajin and So, Anthony Man-Cho},
  journal={Advances in Neural Information Processing Systems},
  volume={38},
  pages={134964--134995},
  year={2026}
}

@article{gao2026moreau,
  title={Moreau Envelope-Based Difference of Weakly Convex Reformulation and Algorithm for Bilevel Programs},
  author={Gao, Lucy and Ye, Jane J and Yin, Hai-An and Zeng, Shang-Zhi and Zhang, Jin},
  journal={Journal of the Operations Research Society of China},
  pages={1--37},
  year={2026},
  publisher={Springer}
}

@article{royset2026stability,
  title={On Stability in Optimistic Bilevel Optimization},
  author={Royset, Johannes O},
  journal={SIAM Journal on Optimization},
  volume={36},
  number={2},
  pages={987--1016},
  year={2026},
  publisher={SIAM}
}

@article{wachter2006implementation,
  title={On the implementation of an interior-point filter line-search algorithm for large-scale nonlinear programming},
  author={W{\"a}chter, Andreas and Biegler, Lorenz T},
  journal={Mathematical programming},
  volume={106},
  number={1},
  pages={25--57},
  year={2006},
  publisher={Springer}
}

\end{document}